\documentclass[letter,final,oneside,notitlepage,10pt]{article}

\usepackage{amssymb}
\usepackage{amsmath}
\usepackage{amstext}
\usepackage[]{geometry}
\usepackage{mathrsfs}
\usepackage{euscript}
\usepackage[all]{xy}
\usepackage{xstring}
\usepackage{slashed}

\def\N {\mathbb{N}}
\def\Z {\mathbb{Z}}

\def\R {\mathbb{R}}
\def\C {\mathbb{C}}

\def\id{\mathrm{id}}

\def\hc#1{\mathrm{h}_{#1}}

\def\subset{\subseteq}

\renewcommand{\varepsilon}{\epsilon}
\renewcommand{\to}{\nobr\!\xymatrix@R=0cm@C=1.4em{\ar[r] &}\nobr}
\renewcommand{\mapsto}{\!\xymatrix@R=0cm@C=1.4em{\ar@{|->}[r] &}\!}
\renewcommand{\Rightarrow}{\!\xymatrix@R=0cm@C=1.4em{\ar@{=>}[r] &}\!}
\renewcommand{\Leftarrow}{\!\xymatrix@R=0cm@C=1.4em{\ar@{<=}[r] &}\!}
\newcommand{\incl}{\!\xymatrix@R=0cm@C=1.4em{\ar@{^(->}[r] &}\!}
\renewcommand{\hookrightarrow}{\incl}
\renewcommand\Leftrightarrow{\!\xymatrix@R=0cm@C=1.4em{\ar@{<=>}[r] &}\!}

\def\quand{\quad\text{ and }\quad}

\def\quomma{\quad\text{, }\quad}

\def\erf#1{(\ref{#1})}

\def\stackerf#1#2{\stackrel{\text{\erf{#1}}}{#2}}

\def\brackets#1{\IfStrEq{#1}{-}{}{(#1)}}
\def\subindex#1{\IfStrEq{#1}{-}{}{_{#1}}}

\newcommand{\alxydim}[2]{\begin{aligned}\xymatrix#1{#2}\end{aligned}}

\makeatletter
\def\bigset#1#2{\left\lbrace\;\begin{minipage}[c]{#1}\begin{center}#2\end{center}\end{minipage}\;\right\rbrace}

\makeatother

\newlength{\myl}

\def\nobr{~\hspace{-0.26em}}
\def\maps{\nobr:\nobr}
\def\df{\nobr := \nobr}
\def\eq{\nobr = \nobr}
\let\Oldin\in\renewcommand{\in}{\nobr\Oldin\nobr}
\let\Oldtimes\times\renewcommand{\times}{\nobr\Oldtimes}
\let\Oldotimes\otimes\renewcommand{\otimes}{\nobr\Oldotimes}

\def\aut{\mathrm{Aut}}
\def\AUT{\mathrm{AUT}}

\newlength{\widthtmp}
\def\length#1{\settowidth{\widthtmp}{#1}\the\widthtmp}

\def\trivlin{\mathbf{I}}

\usepackage{amsthm}
\usepackage{graphicx}
\usepackage[margin=2cm,font=normalsize,labelfont=bf]{caption}
\usepackage{verbatim}
\usepackage{enumerate}
\usepackage{fancyhdr}
\usepackage[]{geometry}
\usepackage[normalem]{ulem}
\usepackage{xstring}
\usepackage{needspace}
\usepackage{color}
\usepackage{amstext}

\makeatletter

\newcounter{denseversion}
\newcounter{authorcounter}
\newcounter{adresscounter}

\def\title#1{\gdef\@title{#1}}
\def\@title{}
\def\subtitle#1{\gdef\@subtitle{#1}}
\def\@subtitle{}

\def\authortagsused{0}
\def\adresstag#1{\if!#1!\else$^{\;#1\;}$\fi}

\renewcommand{\author}[2][]{
  \stepcounter{authorcounter}
  \if!#1!\else\gdef\authortagsused{1}\fi
  \ifnum\value{authorcounter}=1
    \def\@authorstringa{#2\adresstag{#1}}
    \def\@authorstringb{#2}
    \def\@authorstringc{#2\adresstag{#1}}
  \else
    \g@addto@macro\@authorstringa{\ and #2\adresstag{#1}}
    \g@addto@macro\@authorstringb{\ and #2}
    \g@addto@macro\@authorstringc{\\#2\adresstag{#1}}
  \fi}
\def\@author{\ifnum\value{denseversion}=0\@authorstringa\else\@authorstringb\fi}

\def\@adressstringa{}
\def\@adressstringb{}
\newcommand{\adress}[2][]{
  \stepcounter{adresscounter}
  \ifnum\value{adresscounter}=1
    \g@addto@macro\@adressstringa{\ifnum\authortagsused=0\def\br{\\}\else\def\br{, }\fi\adresstag{#1}#2}
    \g@addto@macro\@adressstringb{\def\br{\\}\adresstag{#1}\parbox[t]{14cm}{#2}}
  \else
    \g@addto@macro\@adressstringa{\\[\bigskipamount]\adresstag{#1}#2}
    \g@addto@macro\@adressstringb{\\[\medskipamount]\adresstag{#1}\parbox[t]{14cm}{#2}}
  \fi}

\def\preprint#1{\gdef\@preprint{#1}}
\def\@preprint{}
\def\keywords#1{\gdef\@keywords{#1}}
\def\@keywords{}
\def\msc#1{\gdef\@msc{#1}}
\def\@msc{}
\def\email#1{
   \gdef\@email{#1}
   \g@addto@macro\@authorstringc{ {\it (#1)}}}
\def\@email{}
\def\dedication#1{\gdef\@dedication{#1}}
\def\@dedication{}

\def\mybaselinestretch#1{\gdef\@mybaselinestretch{#1}}
\def\@mybaselinestretch{}

\def\refname{References}

\mybaselinestretch{1.2}
\renewcommand{\baselinestretch}{\@mybaselinestretch}

\def\denseversion{
  \setcounter{denseversion}{1}
  \newgeometry{left=3cm,right=3cm,top=3cm}
  \mybaselinestretch{1.1}
  \renewcommand{\baselinestretch}{\@mybaselinestretch}
  \normalfont
  \fancyfoot[C]{\itshape{--$\,\,$\thepage$\,\,$--}}}

\newlength{\myparskip}
\newlength{\myproofparskip}

\renewcommand{\emph}[1]{\def\reserved@a{it}\ifx\f@shape\reserved@a\uline{#1}\else\textit{#1}\fi}

\newcommand{\mytableofcontents}{
   \ifnum\value{denseversion}=0
     \tableofcontents
   \else
     \renewcommand{\baselinestretch}{0.8}
     \normalfont
     \tableofcontents
     \renewcommand{\baselinestretch}{\@mybaselinestretch}
     \normalfont
   \fi}

\newlength{\zeilenlaenge}
\def\putindent#1{
  \settowidth{\zeilenlaenge}{#1}
  \ifnum\zeilenlaenge>\textwidth
    #1
  \else
    \noindent #1
  \fi
}

\def\href#1#2{#2}

\def\kohyp{
  \usepackage{hyperref}
  \hypersetup{
    linktocpage = true,
    pdftitle = {\@title},
    pdfauthor = {\@author},
    pdfkeywords = {\@keywords},    
    bookmarksopen = true,
    bookmarksopenlevel = 1
  }}  
\def\showkeywords{\begin{flushleft}\footnotesize\textbf{Keywords}: \@keywords\end{flushleft}}
\def\showmsc{\begin{flushleft}\footnotesize\textbf{MSC 2010}: \@msc\end{flushleft}}

\newcounter{mythm}[subsection]
\newcounter{mainthm}

\def\setsecnumdepth#1{
  \setcounter{secnumdepth}{#1}
  \setcounter{mythm}{0}
  \ifnum \c@secnumdepth >0
    \ifnum \c@secnumdepth >1
      \def\themythm{\thesubsection.\arabic{mythm}}
      \numberwithin{equation}{subsection}
      \renewcommand\theequation{\thesubsection.\arabic{equation}}
    \else
      \def\themythm{\thesection.\arabic{mythm}}
      \numberwithin{equation}{section}
      \renewcommand\theequation{\thesection.\arabic{equation}}
    \fi
  \else
    \def\themythm{\arabic{mythm}}
  \fi}

\newenvironment{mythmenv}{\strut\ \setlength{\parskip}{\myproofparskip}}{\setlength{\parskip}{\myparskip}}

\newlength{\mythmskip}
\newlength{\mythmtopskip}
\newtheoremstyle{mythmstylea}{\mythmtopskip}{\mythmskip}{\it}{}{\bf}{.}{0em}{}
\newtheoremstyle{mythmstyleb}{\mythmtopskip}{\mythmskip}{}{}{\bf}{.}{0em}{}

\theoremstyle{mythmstylea}
\newtheorem{mytheorem}[mythm]{Theorem}
\newtheorem{mydefinition}[mythm]{Definition}
\newtheorem{mycorollary}[mythm]{Corollary}
\newtheorem{myproposition}[mythm]{Proposition}
\newtheorem{mylemma}[mythm]{Lemma}

\newtheorem{mymaintheorem}[mainthm]{Theorem}
\newtheorem{mymaincorollary}[mainthm]{Corollary}
\newtheorem{mymainproposition}[mainthm]{Proposition}
\newtheorem{mymaindefinition}[mainthm]{Definition}

\theoremstyle{mythmstyleb}
\newtheorem{myremark}[mythm]{Remark}
\newtheorem{myproblem}[mythm]{Problem}
\newtheorem{myexample}[mythm]{Example}
\newtheorem{myexercise}[mythm]{Exercise}

\newenvironment{theorem}[1][]{\begin{mytheorem}[#1]\begin{mythmenv}}{\end{mythmenv}\end{mytheorem}}
\newenvironment{definition}[1][]{\begin{mydefinition}[#1]\begin{mythmenv}}{\end{mythmenv}\end{mydefinition}}

\newenvironment{proposition}[1][]{\begin{myproposition}[#1]\begin{mythmenv}}{\end{mythmenv}\end{myproposition}}
\newenvironment{lemma}[1][]{\begin{mylemma}[#1]\begin{mythmenv}}{\end{mythmenv}\end{mylemma}}
\newenvironment{remark}[1][]{\begin{myremark}[#1]\begin{mythmenv}}{\end{mythmenv}\end{myremark}}

\newenvironment{example}[1][]{\begin{myexample}[#1]\begin{mythmenv}}{\end{mythmenv}\end{myexample}}

\renewenvironment{proof}[1][Proof]{\noindent #1. \begin{mythmenv}}{\hfill$\square$\end{mythmenv}\medskip}

\def\mytitle{}
\def\zmptitle{
  \begin{tabular}{cc}
    \begin{minipage}[c]{0.4\textwidth}
      \begin{flushleft}
        \includegraphics[width=110pt]{../../tex/zmp}
      \end{flushleft}  
    \end{minipage}&
    \begin{minipage}[c]{0.55\textwidth}
      \begin{flushright}
      {\small\sf\@preprint}
      \end{flushright}
    \end{minipage}
  \end{tabular}
  \vskip 2cm}

\def\maketitle{
  \setlength{\parskip}{\myparskip}  
  \newpage
  \noindent
  \mytitle
  \begin{center}
    \LARGE\@title\\
    \if!\@subtitle!\else\smallskip\LARGE\@subtitle\\\fi
    \bigskip
    \if!\@author!\else\bigskip\large\@author\\\fi
    \ifnum\value{denseversion}=0
      \if!\@adressstringa!\else\bigskip\normalsize\@adressstringa\\\fi
      \if!\@email!\else\ifnum\value{authorcounter}=1\bigskip\normalsize\textit{\@email}\\\else\fi\fi
    \else
    \fi
    \if!\@dedication!\else\bigskip\normalsize{\@dedication}\\\fi
  \end{center}
  \ifnum\value{denseversion}=0\vskip 1.5cm\else\vskip0.5cm\fi
  \if!\@draft!\thispagestyle{empty}\else\fi}

\def\kobiburl#1{
   \IfBeginWith
     {#1}
     {http://arxiv.org/abs/}
     {\kobibarxiv{#1}}
     {\kobiblink{#1}}}
\def\kobibarxiv#1{\href{#1}{\texttt{[arxiv:\StrGobbleLeft{#1}{21}]}}}
\def\kobiblink#1{
  \StrSubstitute{#1}{_}{\underline{\;\;}}[\mylink]
  \StrSubstitute{\mylink}{&}{\&}[\mylink]
  Available as: \mbox{\;}
  \href{#1}{\texttt{\mylink}}}

\newcommand{\etalchar}[1]{$^{#1}$}
\def\kobib#1{
  \begin{raggedright}
  \ifnum\value{denseversion}=0\else\small\fi

  \end{raggedright}
  \ifnum\value{denseversion}=0\else
      \noindent
      \if!\@authorstringc!\else
        \ifnum\authortagsused=0\ifnum\value{authorcounter}>1\normalsize\@authorstringc\\[\medskipamount]\else\fi\else\normalsize\@authorstringc\\[\medskipamount]\fi
      \fi
      \if!\@adressstringb!\else\normalsize\@adressstringb\\{}\fi
      \ifnum\authortagsused=0
        \ifnum\value{authorcounter}=1
          \if!\@email!\else\linebreak\normalsize\textit{\@email}\\{}\fi
        \else
        \fi
      \else
      \fi
  \fi
  }

\newenvironment{commentfigure}{\begin{comment}}{\end{comment}}
\newenvironment{sidewayscommentfigure}{\begin{minipage}}{\end{minipage}}

\def\draft#1#2#3#4{
  \ifnum#4=0
    \def\showcomments{ - Comments are not displayed}
  \else
    \renewenvironment{comment}{\begin{list}{}{\rightmargin=1cm\leftmargin=1cm}\item\sf\begin{small}}{\end{small}\end{list}}

    \def\showcomments{ - Comments are displayed}
  \fi
  \gdef\@draft{DRAFT - Version #1 - Last edited on #2 - Last edited by #3\showcomments}
  \fancyhead[C]{\footnotesize\tt\textcolor{red}{\@draft}}}
\def\@draft{}

\makeatother

\usepackage[latin1]{inputenc}
\usepackage[english]{babel}

\def\quot#1{``#1''}

\def\unifier{unifier}
\def\pt{pseudonatural transformation}
\def\pe{pseudonatural equivalence}

\newcommand{\loctrivfunct}[3]{\mathrm{Triv}^{#3}_{\pi}(#1)}
\newcommand{\loctrivfunctsmooth}[4]{\mathrm{Triv}^{#4}_{#3}(#1)^{\infty}}
\newcommand{\trans}[3]{\mathfrak{Des}_{\pi}^{#3}(#1)}
\newcommand{\ex}[1]{\mathrm{Ex}_{#1}}
\newcommand{\transsmooth}[3]{\mathfrak{Des}^{#2}_{\pi}(#1)^{\infty}}
\newcommand{\transsmoothpi}[4]{\mathfrak{Des}^{#2}_{#4}(#1)^{\infty}}

\newcommand{\transport}[4]{\mathrm{Trans}^{#2}_{#3}(M,#4)}
\newcommand{\transportX}[4]{\mathrm{Trans}^{#2}_{#3}(X,#4)}
\newcommand{\diffco}[3]{Z^{#2}_{#3}(#1)^{\infty}}

\newcommand{\gbunX}{\mathfrak{Bun}^{\nabla}_G(X)}

\newcommand{\sbunX}{\mathfrak{Bun}^{\nabla}_{S^1}(X)}
\newcommand{\upp}[3]{\mathcal{P}^{#1}_{#2}(#3)}

\newcommand{\fo}{\mathcal{D}}
\newcommand{\holo}{\mathscr{F}}
\newcommand{\ttc}[1]{\widehat{#1}}
\newcommand{\hbitor}{H\text{-}\mathrm{BiTor}}

\newcommand{\bigon}[5]{\alxydim{@C=1.2cm}{#1 \ar@/^1.5pc/[r]^{#3}="1" \ar@/_1.5pc/[r]_{#4}="2" \ar@{=>}"1";"2"|{#5} & #2}}

\newcommand{\quadrat}[9]{\alxydim{@C=1.2cm@R=1.2cm}{#1 \ar[r]^{#5} \ar[d]_{#6} & #2 \ar@{=>}[dl]|{#9} \ar[d]^{#7} \\ #3 \ar[r]_{#8} & #4}}

\newcommand{\dreieck}[7]{\alxydim{@R=1.2cm@C=0.3cm}{& #2  \ar[dr]^{#5} & \\ #1 \ar[ur]^{#4} \ar[rr]_{#6}="1" && #3 \ar@{=>}[ul];"1"|{#7}}
}

\newcommand{\kopfdreieck}[7]{\alxydim{@R=1.2cm@C=0.3cm}{#1 \ar[dr]_{#4} \ar[rr]^{#6}="1" && #3 \ar@{=>}"1";[dl]|{#1} \\ & #2  \ar[ur]_{#5} & }
}

\title{Connections on Non-Abelian Gerbes and their Holonomy}

\author[a]{Urs Schreiber}
\email{urs.schreiber@googlemail.com}

\author[b]{Konrad Waldorf}
\email{konrad.waldorf@mathematik.uni-regensburg.de}

\adress[a]{Faculty of Science, Radboud University Nijmegen\br Heyendaalseweg 135, 6525 AJ Nijmegen, The Netherlands}

\adress[b]{Fakultät für Mathematik, Universität Regensburg\br Universitätsstraße 31, D-93053 Regensburg, Germany}

\msc{Primary 53C08, Secondary 55R65, 18D05}

\kohyp
\usepackage{amstext}
\usepackage{latexsym}
\usepackage{amsmath}

\begin{document}

\denseversion
\renewcommand{\baselinestretch}{1}

\maketitle

\begin{abstract}
We introduce  an axiomatic framework for the parallel transport of connections on gerbes. It incorporates parallel transport along curves and along surfaces, and is formulated in terms of gluing axioms and smoothness conditions. The smoothness conditions are imposed with respect to a strict Lie 2-group, which plays the role of a band, or structure 2-group. Upon choosing certain examples of Lie 2-groups,  our axiomatic framework reproduces in a systematical way several known concepts of gerbes with connection: non-abelian differential cocycles, Breen-Messing gerbes, abelian  and non-abelian bundle gerbes. These relationships  convey a well-defined  notion of surface holonomy from our axiomatic framework to each of these concrete models. Till now, holonomy was only known  for abelian gerbes; our approach reproduces that known concept and extends it to  non-abelian gerbes. Several new features of surface holonomy are exposed under its extension to non-abelian gerbes; for example, it carries an action of the mapping class group of the surface. 
\showkeywords
\showmsc
\end{abstract}

\setcounter{tocdepth}{2}
\mytableofcontents

\def\pg#1{\\{\itshape #1.}}
\def\hc#1{\mathrm{h}_{#1}}
\def\red{_{r\!e\!d}}
\def\Red{{\mathcal{R}ed}}

\setsecnumdepth{1}

\section{Introduction}

 Giraud introduced gerbes  in order to achieve a geometrical understanding of \emph{non-abelian} cohomology \cite{giraud}. However, already \emph{abelian} gerbes turned out to be interesting: Brylinski introduced the notion of a \emph{connection} on an abelian gerbe, and showed that these represent classes in a certain \emph{differential} cohomology theory, namely Deligne cohomology \cite{brylinski1}. Deligne cohomology in degree two  has before been related to two-dimensional conformal field theory  by Gaw\c edzki \cite{gawedzki3}. This relation is   established by means of the \emph{surface holonomy} of a connection on an abelian gerbe, which provides a term in the action functional of the field theory. 

Surface holonomy of  connections on  abelian gerbes is today well understood; see \cite{waldorf3,fuchs8} for reviews. While definitions of connections on \emph{non-abelian} gerbes have appeared \cite{breen1,aschieri}, it remained unclear what the surface holonomy of these connections is supposed to be,  how it is defined, and how it can be used.

In the present article we propose a general and systematic approach to connections on non-abelian gerbes, including notions of parallel transport and surface holonomy. Our approach is \emph{general} in the sense that it works for gerbes whose band is an arbitrary Lie 2-groupoid, and whose fibres are modelled by an arbitrary 2-category. Our approach is \emph{systematic} in the sense that it is solely based on axioms for parallel transport along surfaces, formulated in terms of gluing laws and smoothness conditions. The whole theory of connections on non-abelian gerbes is then derived as  a consequence.

In order to illustrate how this axiomatic formulation works we shall briefly review a corresponding formulation in a more familiar setting, namely the one of connections on fibre bundles; see   \cite{schreiber3}. It shows that for a Lie group $G$ the category of principal $G$-bundles with connection over a smooth manifold $X$ is equivalent to a category consisting of functors
\begin{equation}
\label{eq:functors}
F: \mathcal{P}_1(X) \to G\text{-}\mathrm{Tor}\text{.}
\end{equation}
These functors are defined on the path groupoid $\mathcal{P}_1(X)$ of the manifold $X$; its objects are the points of $X$, and its morphisms are (certain classes of)  paths in $X$. The functors \erf{eq:functors} take values in the category of $G$-torsors, i.e. smooth manifolds with a free and transitive $G$-action. 

The correspondence between principal $G$-bundles with connection and  functors \erf{eq:functors} is established by letting the functor $F$ assign to points the fibres of a given bundle, and to paths the corresponding parallel transport maps. The \emph{gluing laws} of parallel transport are precisely the axioms of a functor. The \emph{smoothness conditions} of parallel transport are more involved; they can be encoded in  the functors \erf{eq:functors} by requiring  \emph{smooth descend data} with respect to an open cover of $X$.
Functors  \erf{eq:functors} with smooth descend data are called \emph{transport functors with $G$-structure} -- they constitute an axiomatic formulation of the parallel transport of connections in $G$-bundles.

In Sections \ref{sec1} and  \ref{sec6_3} of the present article we generalize this axiomatic formulation to  connections on gerbes. Our formulation does not use any existing concept of a gerbe with connection  --- such concepts are an \emph{output} of our approach. It is based on \emph{2-functors} defined on the \emph{path 2-groupoid} $\mathcal{P}_2(X)$ of $X$, with values in some \quot{target} 2-category $T$,
\begin{equation}
\label{eq:2functors}
F: \mathcal{P}_2(X) \to T\text{.}
\end{equation}
In Section \ref{sec:path2groupoid} we review the path 2-groupoid: it is like the path groupoid but with additional 2-morphisms, which are essentially fixed-end homotopies between paths.

For example, if $T$ is the 2-category of algebras (over some fixed field), bimodules, and intertwiners, a 2-functor \erf{eq:2functors} provides for each point $x\in X$ an algebra $F(x)$, which is supposed to be the fibre of the gerbe at the point $x$. Further, it provides for each path $\gamma$ from $x$ to $y$ a $F(x)$-$F(y)$-bimodule $F(\gamma)$, which is supposed to be the parallel transport of the connection on that gerbe along the curve parameterized by  $\gamma$. Finally, it provides for each homotopy  $\Sigma$  between paths $\gamma$ and $\gamma'$ an intertwiner 
\begin{equation*}
F(\Sigma): F(\gamma) \to  F(\gamma')\text{,}
\end{equation*}
which is supposed to be the parallel transport of the connection on that gerbe along the surface parameterized by $\Sigma$. 
The axioms of the 2-functor \erf{eq:2functors}  describe how these parallel transport structures are compatible with the composition of paths and gluing of homotopies.
These axioms and all other 2-categorical structure we use can be looked up in \cite[Appendix A]{schreiber6}.

Apart from the evident generalization from functors to 2-functors, more work has to be  invested into the generalization of the smoothness conditions. Imposing smoothness conditions relies on a notion of \emph{local triviality}
for 2-functors defined on path 2-groupoids. A 2-functor
\begin{equation*}
F: \mathcal{P}_2(X) \to T
\end{equation*}
is considered to be trivializable, if it factors through a prescribed 2-functor $i: \mathrm{Gr} \to T$, with $\mathrm{Gr}$ a strict Lie 2-groupoid. The  Lie 2-groupoid $\mathrm{Gr}$ plays the role of the \quot{typical fibre}, and the 2-functor $i$ indicates how the typical fibre is realized in the target 2-category $T$. A local trivialization of the 2-functor $F$
is a cover of $X$ by open sets $U_{\alpha}$,  a collection of locally defined \quot{trivial} 2-functors $\mathrm{triv}_{\alpha}: \mathcal{P}_2(U_{\alpha}) \to \mathrm{Gr}$ and of equivalences
\begin{equation*}
\alxydim{}{t_{\alpha}: F|_{U_{\alpha}} \ar[r]^-{\cong} & i \circ \mathrm{triv}_{\alpha}}
\end{equation*}
between 2-functors defined on $U_{\alpha}$. Local trivializations lead to \emph{descend data}, generalizing the transition functions of a bundle.
The descent data of a 2-functor with a local trivialization consists of the 2-functors $\mathrm{triv}_{\alpha}$, of transformations
\begin{equation*}
g_{\alpha\beta}: i \circ \mathrm{triv}_{\alpha} \to i \circ \mathrm{triv}_{\beta}
\end{equation*}
between 2-functors over $U_{\alpha} \cap U_{\beta}$,
and of higher coherence data that we shall ignore for the purposes of this introduction.  The theory of local trivializations and descent data for 2-functors is developed in our paper \cite{schreiber6} and reviewed in Section \ref{sec:loctriv}.

The smoothness conditions we want to formulate are imposed with respect to descent data;  they are the content of Section \ref{sec6_3}.  First of all, we require that the 2-functors $\mathrm{triv}_{\alpha}$ are smooth. This makes sense since they take values in the \emph{Lie} 2-groupoid $\mathrm{Gr}$. For certain Lie 2-groupoids, a theory developed in our paper \cite{schreiber5} identifies the smooth functors $\mathrm{triv}_{\alpha}$ with certain 2-forms $B_{\alpha}$ on $U_{\alpha}$ --- the curving of the gerbe connection. In order to treat the transformations $g_{\alpha\beta}$, we apply an observation in abstract 2-category theory: the transformations $g_{\alpha\beta}$ can be regarded as a collection of functors 
\begin{equation*}
\holo(g_{\alpha\beta}): \mathcal{P}_1(U_{\alpha} \cap U_{\beta}) \to \Lambda T\text{,}
\end{equation*}
for $\Lambda T$ a certain category of diagrams in $T$.
The smoothness condition that we impose for the transformation $g_{\alpha\beta}$ is that the functors $\holo(g_{\alpha\beta})$ are transport functors. According to the before-mentioned correspondence between transport functors and fibre bundles with connection, we thus obtain a smooth fibre bundle $\holo(g_{\alpha\beta})$ with connection over two-fold overlaps $U_{\alpha} \cap U_{\beta}$ --- a significant feature of a gerbe.

Summarizing this overview, our axiomatic formulation of connections on gerbes consists of transport 2-functors: 2-functors $F: \mathcal{P}_2(X) \to T$ that are locally trivializable with respect to a typical fibre $i: \mathrm{Gr} \to T$, and have smooth descent data.

In Section \ref{sec4} of this article we test our axiomatic formulation by choosing examples of target 2-categories $T$ and 2-functors $i: \mathrm{Gr} \to T$. In these examples the Lie 2-groupoids are \quot{deloopings} of strict Lie 2-groups, $\mathrm{Gr}=\mathcal{B}\mathfrak{G}$; these Lie 2-groups $\mathfrak{G}$ play the same role for gerbes as Lie groups  for principal bundles. We find the following results:
\begin{enumerate}[(i)]
\item 
For a general Lie 2-group $\mathfrak{G}$ and the identity 2-functor 
\begin{equation*}
i = \id_{\mathcal{B}\mathfrak{G}}: \mathcal{B}\mathfrak{G} \to \mathcal{B}\mathfrak{G}
\end{equation*}
we prove (Theorem \ref{th:class}) that there is a bijection
\begin{equation*}
\hc 0 \mathrm{Trans}_{\mathcal{B}\mathfrak{G}}(X, \mathcal{B}\mathfrak{G})  \cong \hat H^2(X,\mathfrak{G}) 
\end{equation*}
between isomorphism classes of transport 2-functors  and the degree two differential non-abelian cohomology of  $X$ with coefficients in  $\mathfrak{G}$. These cohomology groups have been explored in \cite{breen1} and \cite{baez2}; they are an extension of Giraud's non-abelian cohomology by differential form data. Upon setting $\mathfrak{G}=\mathcal{B}S^1$ it reduces to  Deligne cohomology.

\item
The Lie 2-group $\mathcal{B}S^1$ has a monoidal functor $\mathcal{B}S^1 \to S^1\text{-}\mathrm{Tor}$ to the monoidal category of $S^1$-torsors, by sending the single objects of $\mathcal{B}S^1$ to $S^1$ considered as a torsor over itself. Delooping  yields a 2-functor 
\begin{equation*}
i: \mathcal{BB}S^1 \to \mathcal{B}(S^1\text{-}\mathrm{Tor})\text{.}
\end{equation*}
We prove (Theorem \ref{th:bgrb})  that there is an  equivalence of 2-categories
\begin{equation*}
\mathrm{Trans}_{\mathcal{BB}S^1}(X, \mathcal{B}(S^1\text{-}\mathrm{Tor})) \cong \bigset{10em}{$S^1$-bundle gerbes with connection over $X$}\text{.}
\end{equation*}
Bundle gerbes have been introduced by Murray \cite{murray}.
The equivalence arises  by realizing that the transport functor $\holo(g_{\alpha\beta})$ in the descent data  corresponds in the present situation to an $S^1$-bundle with connection.

\item
Let $H$ be a Lie group and let $\mathrm{AUT}(H)$ be the automorphism 2-group of $H$. It has a monoidal functor $\AUT(H) \to \hbitor$ to the monoidal category of $H$-bitorsors. Delooping yields a 2-functor
\begin{equation*}
i: \mathcal{B}\AUT(H) \to \mathcal{B}(\hbitor)\text{.}
\end{equation*}  
We prove (Theorem \ref{th:nagrb}) that there is an  equivalence of 2-categories
\begin{equation*}
\mathrm{Trans}_{\mathcal{B}\mathrm{AUT}(H)}(X,\mathcal{B}(H\text{-}\mathrm{BiTor})) \cong \bigset{13em}{Non-abelian $H$-bundle gerbes with connection over $X$}\text{.}
\end{equation*}
Non-abelian bundle gerbes are a generalization of $S^1$-bundle gerbes introduced in \cite{aschieri}, and the above equivalence arises by proving that the transport functor $\holo(g_{\alpha\beta})$ corresponds in the present situation to a \quot{principal $H$-bibundle with twisted connection}.

\end{enumerate}
The relations (i) to (iii) show that all these existing concepts of gerbes with connection fit into our axiomatic formulation.

Apart from these relations to existing gerbes with connection, transport 2-functors are able to  determine systematically new  concepts  in cases when only the target 2-category $T$ and the 2-group $\mathfrak{G}$ are given. We demonstrate this in Section \ref{sec4_4} with the examples of connections on  vector 2-bundles, string 2-bundles, and  principal 2-bundles.

Finally,  we  discuss in Section \ref{sec5} the notion of parallel transport along surfaces,  which is manifestly included in the concept of a transport 2-functor. We  introduce a notion of surface holonomy for transport 2-functors, defined for closed oriented surfaces with a marking, i.e. a certain presentation of its fundamental group. It is obtained by evaluating the transport 2-functor on a homotopy that realizes the single relation in this presentation.

The existing notion of surface holonomy for \emph{abelian} gerbes  takes values in  $S^1$ \cite{gawedzki3,murray}, while our notion of surface holonomy takes values in the 2-morphisms of the target 2-category $T$. In order to compare the two notions, we propose a \quot{reduction} procedure which can be applied in the case that typical fibre of the transport 2-functor is of the form $i: \mathcal{B}\mathfrak{G} \to T$, where $\mathfrak{G}$ is a Lie 2-group. The first part of this procedure is the definition of an abelian group $\mathfrak{G}\red$ which can be formed for any Lie 2-group $\mathfrak{G}$ (Definition \ref{def:redgroup}). Heuristically, it generalizes the abelianization of an ordinary Lie group. As the second part of the reduction procedure, we show (Proposition \ref{prop:reduction}) that the surface holonomy of every transport 2-functor with $\mathcal{B}\mathfrak{G}$-structure can consistently be reduced to a function with values in $\mathfrak{G}\red$.

Our main results in Section \ref{sec5} concern this reduced  surface holonomy of transport 2-functors with $\mathcal{B}\mathfrak{G}$-structure. We prove in Theorem \ref{th:rhol} a  rigidity result for reduced surface holonomy, namely that it  depends only on the isomorphism class of the transport 2-functor, and only on the equivalence class of the marking. The isomorphism invariance allows us to transfer the reduced surface holonomy from transport 2-functors through the equivalences (i), (ii), and (iii) described above. In particular, we equip non-abelian $\mathfrak{G}$-gerbes with a well-defined notion of a $\mathfrak{G}\red$-valued surface holonomy; such a concept was not known before.

Finally, we show that our new concept of reduced surface holonomy is compatible with the existing notion of $S^1$-valued surface holonomy of abelian gerbes. Namely, in the case  $\mathfrak{G}=\mathcal{B}S^1$ we  find  for the reduction  $(\mathcal{B}S^1)\red = S^1$, so that both concepts take values in the same set. We prove then
in Proposition \ref{propabsf} that the two concepts indeed coincide.  Thus, our new notion of (reduced) surface holonomy consistently extends the existing notion from abelian to non-abelian gerbes.

\paragraph{Acknowledgements.} 
The project described here has some of
its roots in ideas of John Baez and in his joint work with  US, and we are grateful for all discussions and suggestions.
We are also  grateful for opportunities to give talks about this project at an unfinished state, namely at the Fields Institute, at the VBAC 2007 meeting in Bad Honnef, at the MedILS in Split and at the NTNU in Trondheim. In addition, we thank the Hausdorff Research Institute for Mathematics in Bonn for kind hospitality and support during several visits.


\setsecnumdepth{2}

\section{Foundations of the Transport Functor Formalism}

\label{sec1}

The present paper is the last part of a project carried out in a sequence of papers \cite{schreiber3,schreiber5,schreiber6}. In these papers, we have prepared the foundations for transport 2-functors -- our axiomatic formulation of connection on non-abelian gerbes. The purpose of this section is to make the present paper self-contained; we collect and review the most important definitions and results from the previous papers.  

\subsection{The Path 2-Groupoid of a Smooth Manifold}

\label{sec:path2groupoid}

The basic idea of the path 2-groupoid is very simple: for a smooth manifold $X$, it is a strict 2-category whose objects are the points of $X$, whose 1-morphisms are smooth  paths in~$X$, and whose 2-morphisms are smooth homotopies between these paths. We recall some definitions from \cite{schreiber3,schreiber5}.

For points $x,y\in X$, a \emph{path} $\gamma:x \to y$ is a smooth map $\gamma:[0,1] \to X$ with $\gamma(0)=x$ and $\gamma(1)=y$. Since the composition $\gamma_2 \circ \gamma_1$ of two  paths $\gamma_1:x \to y$ and $\gamma_2:y \to z$  should again be a smooth map we require sitting instants for all paths: a number $0 < \varepsilon <\frac{1}{2}$ with
$\gamma(t)=\gamma(0)$ for $0 \leq t < \varepsilon$ and $\gamma(t)=\gamma(1)$ for $1-\varepsilon
< t \leq 1$. The set of these paths is denoted by $PX$. In order to make the composition associative and to make paths invertible, we  consider the following equivalence relation on $PX$: two paths $\gamma,\gamma':x \to y$ are called \emph{thin homotopy equi\-va\-lent}
if there exists a smooth map
$h:[0,1]^2 \to X$ 
such that 
\begin{enumerate}
\item[(1)]
$h$ is a homotopy from $\gamma$ to $\gamma'$ through paths $x \to y$ and has sitting instants at $\gamma$ and $\gamma'$.

\item[(2)]
the differential of $h$ has at most rank 1. 
\end{enumerate}
The set of equivalence classes is denoted by $P^1X$.
We remark that any path $\gamma$ is thin homotopy equivalent to any orientation-preserving reparameterization of $\gamma$.
The composition of paths induces a well-defined associative composition on $P^1X$ for which the constant paths $\id_x$ are identities and the reversed paths $\gamma^{-1}$ are inverses; see \cite[Section 2.1]{schreiber3} for more details.

A homotopy $h$ between two paths $\gamma_0$ and $\gamma_1$ like above but without condition (2) on the rank of its differential is called a \emph{bigon} in $X$ and denoted by $\Sigma: \gamma_0 \Rightarrow \gamma_1$. These bigons  form the 2-morphisms of the path 2-groupoid of $X$.
We denote the set
of  bigons in $X$ by $BX$. Bigons can be composed in two natural ways. For two bigons $\Sigma:\gamma_1
\Rightarrow \gamma_2$ and $\Sigma':\gamma_2 \Rightarrow \gamma_3$ 
we have a \emph{vertical composition} 
\begin{equation*}
\Sigma'\bullet\Sigma: \gamma_1 \Rightarrow \gamma_3\text{.}
\end{equation*}
If two bigons $\Sigma_1: \gamma_1 \Rightarrow \gamma_1'$ and $\Sigma_2:\gamma_2 \Rightarrow \gamma_2'$ are such  that $\gamma_1(1)=\gamma_2(0)$, we have a \emph{horizontal composition }
\begin{equation*}
\Sigma_2 \circ \Sigma_1: \gamma_2 \circ \gamma_1 \Rightarrow \gamma_2' \circ \gamma_1'\text{.}
\end{equation*}
Like in the case of paths, we  consider an equivalence relation on $BX$ in order to make the two compositions associative and to make bigons invertible:
two bigons $\Sigma:\gamma_0 \Rightarrow \gamma_1$ and $\Sigma':\gamma'_0 \Rightarrow \gamma'_1$
are called \emph{thin ho\-mo\-to\-py equivalent} if there exists a smooth
map $h: [0,1]^{3} \to X$ such that
\begin{enumerate}
\item[(1)] 
$h$ is a homotopy from $\Sigma$ to $\Sigma'$ through bigons and has sitting instants at $\Sigma$ and $\Sigma'$.

\item[(2)]
the induced  homotopies $\gamma_0 \Rightarrow \gamma_0'$ and $\gamma_1\Rightarrow \gamma_1'$ are thin.

\item[(3)]
the differential of $h$ has at most rank 2.
\end{enumerate}
Condition (1) assures that we have defined an equivalence
relation on $BX$, and  condition (2) asserts that two thin homotopy equivalent bigons $\Sigma:\gamma_0 \Rightarrow \gamma_1$ and $\Sigma'\maps\gamma_0' \Rightarrow \gamma_1'$ start and end on thin homotopy equivalent paths
$\gamma_0 \sim \gamma_0'$ and $\gamma_1 \sim \gamma_1'$. We denote the set of equivalence classes by $B^2X$.
The two compositions $\circ$ and $\bullet$ between bigons  induce a well-defined composition on $B^2X$. The \emph{path 2-groupoid $\mathcal{P}_2(X)$} is  the 2-category whose set of objects is $X$, whose set of 1-morphisms is $P^1X$ and whose set of 2-morphisms is $B^2X$. The path 2-groupoid is strict and all 1-morphisms are strictly invertible. We refer the reader to \cite[Section 2.1]{schreiber5} for a detailed discussion.

In this article we describe connections on gerbes by transport 2-functors --  certain (not necessarily strict) 2-functors
\begin{equation*}
F: \mathcal{P}_2(X) \to T\text{,}
\end{equation*}
for some 2-category $T$,  the \emph{target 2-category}. 
We note that 2-functors can be pulled back along smooth maps $f:M \to X$: such a map induces  a strict 2-functor $f_{*}\maps\mathcal{P}_2(M) \to \mathcal{P}_2(X)$, and we write
\begin{equation*}
f^{*}F := F \circ f_{*}\text{.}
\end{equation*}

If we drop condition (3) from the definition of thin homotopy equivalence between bigons we would still get a strict 2-groupoid, which we denote by $\Pi_2(X)$ and  which we call the \emph{fundamental 2-groupoid} of $X$. The projection defines a strict 2-functor $\mathcal{P}_2(X) \to \Pi_2(X)$. We say that a 2-functor $F\maps\mathcal{P}_2(M) \to T$ is \emph{flat} if it factors through the 2-functor $\mathcal{P}_2(M) \to \Pi_2(M)$. We  show in Section  \ref{sec5_1} that this abstract notion of flatness is equivalent to the vanishing of a certain curvature 3-form.

\subsection{Local Trivializations and Descent Data}

\label{sec1_2}
\label{sec:loctriv}

Let $T$ be a 2-category. 
A key feature of a transport 2-functor is that it is locally trivializable. Local trivializations
of a   2-functor $F: \mathcal{P}_2(M)
\to T$ are defined with respect to three attributes:
\begin{enumerate}
\item
A  strict 
2-groupoid $\mathrm{Gr}$, the \emph{structure 2-groupoid}. In Section \ref{sec3} we will require that $\mathrm{Gr}$ is a \emph{Lie} 2-groupoid, and formulate smoothness conditions with respect to its smooth structure. 

\item
A   2-functor $i:\mathrm{Gr} \to T$ that indicates how the structure 2-groupoid is realized in the target 2-category.  

\item 
A surjective
submersion $\pi:Y \to M$, which serves as an \quot{open cover} of the base manifold $M$.
\end{enumerate}
For a surjective submersion $\pi:Y \to M$ the fibre products $Y^{[k]}\nobr := \nobr Y \times_M ... \times_M Y$ are again smooth manifolds in such a way that the  projections
$\pi_{i_1...i_p}\maps Y^{[k]} \to Y^{[p]}$
(to the indexed factors) are smooth maps. An example is an open cover $\mathfrak{U}=\left \lbrace U_\alpha \right \rbrace$ of $M$, for which the disjoint union of all open sets $U_\alpha$ together with the projection to $M$ is a surjective submersion. In this example, the $k$-fold fibre product is the disjoint union of the $k$-fold intersections of the open sets $U_{\alpha}$. 

\begin{definition}
\label{sec2_def5}
A \emph{$\pi$-local $i$-trivialization} of a 2-functor 
$F\maps \mathcal{P}_2(M)
\to T$
is a pair $(\mathrm{triv},t)$ of a strict 2-functor 
\begin{equation*}
\mathrm{triv}: \mathcal{P}_2(Y) \to \mathrm{Gr}
\end{equation*}
and
a \pe
\begin{equation*}
\quadrat{\mathcal{P}_2(Y)}{\mathcal{P}_2(M)}{\mathrm{Gr}}{T\text{.}}{\pi_{*}}{\mathrm{triv}}{F}{i}{t}
\end{equation*}
\end{definition}

For the notion of a pseudonatural equivalence we refer to \cite[Appendix A]{schreiber6}. According to the conventions we fixed there, it \emph{includes} a weak inverse $\bar t$ together with modifications
\begin{equation}
\label{7}
i_t :  \bar t \circ t \Rightarrow \id_{\pi^{*}F}
\quad\text{ and }\quad
j_t :  \id_{\mathrm{triv}_i} \Rightarrow t \circ \bar t
\end{equation}
satisfying the so-called zigzag identities.

In the following we use the abbreviation $\mathrm{triv}_i := i \circ \mathrm{triv}$, and we write $\loctrivfunct{i}{\mathrm{Gr}}{2}$ for the 2-category of 2-functors $F: \mathcal{P}_2(M) \to T$ with $\pi$-local $i$-trivializations (together with all pseudonatural transformations and all modifications).
Next we come to the definition of a 2-category $\trans{i}{\pi}{2}$ of \emph{descent data} with respect to a surjective submersion $\pi:Y \to M$ and a structure 2-groupoid $i:\mathrm{Gr} \to T$. 
\begin{definition}
\label{def:descendobject}
A \uline{descent object}  is a  quadruple $(\mathrm{triv},g,\psi,f)$  consisting of a strict  2-functor
\begin{equation*}
\mathrm{triv}\maps\mathcal{P}_2(Y) \to \mathrm{Gr}\text{,}
\end{equation*}
a \pe\
\begin{equation*}
g\maps \pi_1^{*}\mathrm{triv}_i \to \pi_2^{*}\mathrm{triv}_i\text{,}
\end{equation*}
and invertible, coherent modifications 
\begin{equation*}
\psi: \id_{\mathrm{triv}_i} \Rightarrow \Delta^{*}g
\quand 
f\maps \pi_{23}^{*}g \circ \pi_{12}^{*}g \Rightarrow \pi_{13}^{*}g\text{.}
\end{equation*}
\end{definition}
\begin{comment}
In these diagrams, $r$, $l$ and $a$ are the right and left \unifier s and the associator of the 2-category $T$, $\Delta:Y \to Y^{[2]}$ is the diagonal map, and $\Delta_{112},\Delta_{122}:Y^{[2]} \to Y^{[3]}$ are the maps duplicating the first or the second factor, respectively. \end{comment}

The coherence conditions for the modifications $\psi$ and $f$ can be found in \cite[Definition 2.2.1]{schreiber6}.

Let us briefly rephrase the above definition in case that $Y$ is the union of open sets $U_{\alpha}$: first there are  strict 2-functors $\mathrm{triv}_{\alpha}:\mathcal{P}_2(U_{\alpha}) \to \mathrm{Gr}$. To compare the difference between $\mathrm{triv}_{\alpha}$ and $\mathrm{triv}_{\beta}$ on a two-fold intersection $U_{\alpha}\cap U_{\beta}$ there are \pe s $g_{\alpha\beta}: (\mathrm{triv}_{\alpha})_i \to (\mathrm{triv}_{\beta})_i$. If we assume for a moment that $g_{\alpha\beta}$ was the transition function of some fibre bundle, one would demand that $1=g_{\alpha\alpha}$ on every $U_{\alpha}$ and that $g_{\beta\gamma}g_{\alpha\beta}=g_{\alpha\gamma}$ on every three-fold intersection $U_{\alpha}\cap U_{\beta}\cap U_{\gamma}$. In the present situation, however, these equalities have been replaced by modifications: the first one by a modification $\psi_{\alpha}: \id_{(\mathrm{triv}_{\alpha})_i} \Rightarrow g_{\alpha\alpha}$ and the second one by a modification $f_{\alpha\beta\gamma}: g_{\beta\gamma} \circ g_{\alpha\beta} \Rightarrow g_{\alpha\gamma}$.

Next we describe how to extract a descend object from a local trivialization of a 2-functor following \cite[Section 2.3]{schreiber6}. Let $F:\mathcal{P}_2(M) \to
T$ be a  2-functor with a $\pi$-local $i$-trivialization $(\mathrm{triv},t)$.
Using the weak inverse $\bar t: \mathrm{triv}_i
\to \pi^{*}F$ of $t$ we define
\begin{equation*}
g:= \pi_2^{*}t \circ \pi_1^{*}\bar t : \pi_1^{*}\mathrm{triv_i} \to \pi_2^{*}\mathrm{triv}_i\text{.}
\end{equation*}
This composition is well-defined since $\pi_1^{*}\pi^{*}F=\pi_2^{*}\pi^{*}F$. Let $i_t$ and$j_t$ be the modifications \erf{7}. We obtain $\Delta^{*}g= t \circ \bar t$, so that the definition $\psi:= j_t$ yields the invertible modification
$\psi:\id_{\mathrm{triv}_i} \Rightarrow \Delta^{*}g$.
Similarly, one defines with $i_{t}$ the invertible modification $f$.
The quadruple $(\mathrm{triv},g,\psi,f)$ obtained like this is a descend object in the sense of Definition \ref{def:descendobject}; see \cite[Lemma 2.3.1]{schreiber6}.

Next suppose $(\mathrm{triv},g,\psi,f)$ and $(\mathrm{triv}',g',\psi',f')$ are descent objects. A \emph{descent 1-morphism} $(\mathrm{triv},g,\psi ,f) \to (\mathrm{triv}',g',\psi ',f')$
  is a pair $(h,\varepsilon)$  consisting of a 
  \pt
\begin{equation*}
h: \mathrm{triv}_i \to \mathrm{triv}_i'
\end{equation*}
and  an invertible    modification
\begin{equation*}
\epsilon: \pi_2^{*}h \circ g \Rightarrow g' \circ \pi_1^{*}h
\end{equation*}
satisfying two natural coherence conditions; see \cite[Definition 2.2.2]{schreiber6}.
\begin{comment}
These are the commutativity of the diagrams
\begin{equation}
\label{5}
\alxydim{@C=2cm}{\pi_{23}^{*}g' \circ (\pi_{2}^{*}h \circ \pi_{12}^{*}g) \ar@{=>}[r]^{a} \ar@{=>}[d]_{\id \circ \pi_{12}^{*}\epsilon} & (\pi_{23}^{*}g' \circ \pi_{2}^{*}h) \circ \pi_{12}^{*}g \ar@{=>}[d]^{\pi_{23}^{*}\epsilon^{-1} \circ \id}   \ar@{=>}[d]\\\pi_{23}^{*}g' \circ (\pi_{12}^{*}g' \circ \pi_1^{*}h)  \ar@{=>}[d]_{a^{-1}} & (\pi_{3}^{*}h \circ \pi_{23}^{*}g) \circ \pi_{12}^{*}g \ar@{=>}[d]^{a}   \ar@{=>}[d] \\(\pi_{23}^{*}g' \circ \pi_{12}^{*}g') \circ \pi_1^{*}h\ar@{=>}[d]_{f' \circ \id} &\pi_{3}^{*}h \circ (\pi_{23}^{*}g \circ \pi_{12}^{*}g)\ar@{=>}[d]^{\id\circ f} \\ \pi_{13}^{*}g'\circ \pi_1^{*}h \ar@{=>}[r]_{\pi_{13}^{*}\epsilon} & \pi_{3}^{*}h \circ \pi_{13}^{*}g\text{.}}
\end{equation}
and
\begin{equation}
\label{4}
\alxydim{}{\id_{\mathrm{triv}'_i} \circ h \ar@{=>}[r]^-{l_{h}} \ar@{=>}[d]_{\psi' \circ \id_h} & h \ar@{=>}[r]^-{r_{h}^{-1}} & h \circ \id_{\mathrm{triv}_i} \ar@{=>}[d]^{\id_h \circ \psi} \\ \Delta^{*}g' \circ h \ar@{=>}[rr]_{\Delta^{*}\epsilon} && h \circ \Delta^{*}g\text{.}}
\end{equation}
\end{comment}
Finally, we suppose that $(h_1,\varepsilon_1)$ and $(h_2,\varepsilon_2)$ are descent 1-morphisms from a descent object $(\mathrm{triv},g,\psi,f)$ to another descent object $(\mathrm{triv}',g',\psi',f')$. A \emph{descent 2-morphism} $(h_1,\varepsilon_1) \Rightarrow (h_2,\varepsilon_2)$ is a modification 
\begin{equation*}
E: h_1 \Rightarrow h_2
\end{equation*}
satisfying another coherence condition; see \cite[Definition 2.2.3]{schreiber6}.
\begin{comment}
Namely, the diagram
\begin{equation}
\label{6}
\alxydim{@=1.2cm}{g' \circ \pi_1^{*}h_1 \ar@{=>}[d]_{\id \circ \pi_1^{*}E}  \ar@{=>}[r]^{\varepsilon_{1}} & \pi_2^{*}h_1 \circ g \ar@{=>}[d]^{\pi_2^{*}E \circ \id } \\ g' \circ \pi_1^{*}h_2 \ar@{=>}[r]_{\varepsilon_2} & \pi_2^{*}h_2 \circ g\text{.}}
\end{equation}
\end{comment}

Descent objects, 1-morphisms and 2-morphisms form a  2-category $\trans{i}{\pi}{2}$, called the \emph{descent 2-category}. In concrete examples of the target 2-category $T$ these structures have natural interpretations in terms of smooth maps and differential forms,  as we show in Section \ref{sec4}. The extraction of a descent object from a local trivialization  outlined above extends to a 2-functor
\begin{equation}
\label{eq:expi}
\ex{\pi}:\loctrivfunct{i}{\mathrm{Gr}}{2} \to \trans{i}{\mathrm{Gr}}{2}\text{,}
\end{equation}
which we have described in \cite[Section 2.3]{schreiber6}.

\label{sec:locglob}

In order to avoid the dependence to the fixed surjective submersion $\pi:Y \to M$, we have shown in \cite[Section 4.2]{schreiber6} that the two 2-categories $\loctrivfunct{i}{\mathrm{Gr}}{2}$ and $\trans{i}{\mathrm{Gr}}{2}$ form a direct system for refinements of surjective submersions over $M$. The corresponding direct limits are 2-categories
\begin{equation*}
\mathrm{Triv}^2(i)_{M} := \lim_{\overrightarrow{\pi}} \mathrm{Triv}^2_{\pi}(i)
\quand
\mathfrak{Des}^2(i)_{M} := \lim_{\overrightarrow{\pi}} \mathfrak{Des}^2_{\pi}(i)\text{.}
\end{equation*}
For instance, an object in the direct limit is a pair of a surjective submersion $\pi$ and an object in the corresponding  2-category $\mathrm{Triv}^2_{\pi}(i)$ or $\mathfrak{Des}^2_{\pi}(i)$. 1-morphisms and 2-morphisms are defined over common refinements.
The 2-functor $\ex{\pi}$ from \erf{eq:expi} induces an equivalence
\begin{equation*}
\mathrm{Triv}^2(i)_{M} \cong \mathfrak{Des}^2(i)_{M}
\end{equation*}
between these two direct limit 2-categories \cite[Proposition 4.2.1]{schreiber6}.

Finally, we want to get rid of the chosen trivializations that are attached to the objects of $\mathrm{Triv}^2(i)_{M}$. We denote by $\mathrm{Funct}_i(\mathcal{P}_2(M),T)$ the 2-category of locally $i$-trivializable 2-functors, i.e. 2-functors which \emph{admit} a $\pi$-local $i$-trivialization, for \emph{some} surjective submersion $\pi$. We have shown \cite[Theorem 4.3.1]{schreiber6}:
\begin{theorem}
\label{th:locglobequiv}
There is an equivalence
\begin{equation*}
\mathrm{Funct}_i(\mathcal{P}_2(M),T) \cong \mathfrak{Des}^2(i)_{M} 
\end{equation*}
between 2-categories of locally $i$-trivializable 2-functors and their descend data.
\end{theorem}

In Section \ref{sec6_3} we  select a sub-2-category of $\mathfrak{Des}^2(i)_{M}$ consisting of \emph{smooth descend data}. The corresponding sub-2-category of $\mathrm{Funct}_i(\mathcal{P}_2(M),T)$ is the one we are aiming at -- the 2-category of transport 2-functors. 

\subsection{Smooth 2-Functors}

\label{sec3}
\label{sec3_1}

This section and the forthcoming Section \ref{sec:smoothnessconditions} prepare two tools we need in Section \ref{sec:smoothdescent}  in order to specify the sub-2-category of smooth descend data. The first tool is the concept  of \emph{smooth 2-functors}.

The general idea behind \quot{smooth functors} is to consider them internal to smooth manifolds. That is, the sets of objects and morphisms of the involved categories are smooth manifolds, and a smooth functor consists of a smooth map between the objects and a smooth map between the morphisms. Categories internal to smooth manifolds are called \emph{Lie categories}, internal groupoids are called \emph{Lie groupoids}. The same idea applies to 2-functors between \emph{2}-categories.

In the context of the present paper, we want to consider smooth  2-functors defined on the path 2-groupoid $\mathcal{P}_2(X)$ of a smooth manifold $X$,  respectively. However, $\mathcal{P}_2(X)$ is \emph{not} internal to smooth manifolds, not even infinite-dimensional ones. Instead, we consider it internal to a larger category of generalized manifolds, so-called \emph{diffeological spaces} \cite{souriau1}. Diffeological spaces and diffeological maps form a category $D^{\infty}$ that enlarges the category $C^{\infty}$ of smooth manifolds by means of a full and faithful functor
$C^{\infty} \to D^{\infty}$.
For an introduction to diffeological spaces we refer the reader to  \cite{baez6} or  \cite[Appendix A.2]{schreiber3}.

Diffeological spaces admit many constructions that are not possible in the category of smooth manifolds. We need three of them. Firstly, if $X$ and $Y$ are diffeological spaces, the set $D^{\infty}(X,Y)$ of smooth maps from $X$ to $Y$ is again a diffeological space. In particular, the set of smooth maps between smooth manifolds is a diffeological space. Secondly, every subset of a diffeological space is a diffeological space. Thirdly, the quotient of every diffeological space by any equivalence relation is a diffeological space. These constructions are relevant because they show that the set $P^1X$ of thin homotopy classes of paths in $X$  as well as the set $B^{2}X$ of thin homotopy classes of bigons in $X$  are diffeological spaces.
We conclude that the path 2-groupoid $\mathcal{P}_2(X)$ of  a smooth manifold $X$ is internal to diffeological spaces, and we have a corresponding 2-category $\mathrm{Funct}^{\infty}(\mathcal{P}_2(X),S)$ of smooth  2-functors with values in some Lie 2-category $S$.

\subsection{Transport Functors}

\label{sec3_5}
\label{sec:smoothnessconditions}

The second tool we need for Section \ref{sec6_3} is the concept of a transport functor. 
Transport functors are an axiomatic formulation of  connections in fibre bundles -- they are the one-dimensional analogue of transport 2-functors, the axiomatic formulation of connections on non-abelian gerbes we are aiming at in the present article. We have introduced and discussed transport functors in \cite{schreiber3}.

From a general perspective, the definition of a \quot{transport $n$-functor} is supposed to rely on a recursive principle in the sense that it uses transport $(n-1)$-functors.  This is one reason to recall the definition of a transport functor. The other reason is to highlight the analogy between the two definitions, which might be helpful to notice: 

\begin{enumerate}[(a)]
\item 
Instead of the path 2-groupoid $\mathcal{P}_2(X)$, we are looking at the \emph{path groupoid} $\mathcal{P}_1(X)$, obtained by just taking objects and 1-morphisms of $\mathcal{P}_2(X)$. A transport functor is a certain functor
\begin{equation*}
F: \mathcal{P}_1(X) \to T\text{,}
\end{equation*}
for some target category $T$: it assigns objects in $T$ -- the \quot{fibres} -- to the points of $X$, and morphisms in $T$ -- the \quot{parallel transport maps} -- to paths in $X$. 

\item
In order to say which functors are transport functors we need a Lie groupoid $\mathrm{Gr}$ and a functor $i\maps \mathrm{Gr} \to T$.
A \emph{local $i$-trivialization} of $F$ is a surjective submersion $\pi:Y \to X$, a functor $\mathrm{triv}: \mathcal{P}_1(Y) \to \mathrm{Gr}$, and a natural equivalence
\begin{equation*}
\alxydim{}{\mathcal{P}_1(Y) \ar[r]^{\pi_{*}} \ar[d]_{\mathrm{triv}} & \mathcal{P}_1(X) \ar@{=>}[dl]|*+{t} \ar[d]^{F} \\ \mathrm{Gr} \ar[r]_{i} & T\text{.} }
\end{equation*}

\item
Associated to a local trivialization is a \emph{descent object}: it is a pair $(\mathrm{triv},g)$ consisting of the functor $\mathrm{triv}: \mathcal{P}_1(Y) \to \mathrm{Gr}$ and of a natural equivalence
\begin{equation*}
g: \pi_1^{*}\mathrm{triv}_i \to \pi_2^{*}\mathrm{triv}_i
\end{equation*}
satisfying a cocycle condition. 

\end{enumerate}
The final step in the definition of a transport functor is the characterization of \emph{smooth descent data}. 
\begin{enumerate}
\item[(d)]
A descent object $(\mathrm{triv},g)$ is called \emph{smooth}, if the functor \begin{equation*}
\mathrm{triv}: \mathcal{P}_1(X)\to \mathrm{Gr}
\end{equation*}
is smooth, i.e. internal to diffeological spaces, and if the components map 
\begin{equation*}
g: Y^{[2]} \to \mathrm{Mor}(T)
\end{equation*}
of the natural equivalence $g$ is the composition of a \emph{smooth} map $\tilde g: Y^{[2]} \to \mathrm{Mor}(\mathrm{Gr})$ with $i: \mathrm{Gr} \to T$. 

\end{enumerate}
In view of the analogy between  (a) - (c) and  Sections \ref{sec:path2groupoid} and \ref{sec:loctriv}, (d) is the analogue of the forthcoming Section \ref{sec:smoothdescent}. Summarizing, we have:

\begin{definition}[{{\cite[Definition 3.6]{schreiber3}}}]
A \emph{transport functor on $X$ with values in $T$ and with $\mathrm{Gr}$-structure} is a locally $i$-trivializable functor
\begin{equation*}
F: \mathcal{P}_1(X) \to T
\end{equation*}
with smooth descent data. 
\end{definition}

Transport functors form a category which we denote by $\transportX{i}{1}{\mathrm{Gr}}{T}$.
The main result of our paper \cite{schreiber3} is that transport functors are an axiomatic formulation of connection on fibre bundles. 

In order to illustrate that, and since we need this result later several times, we provide the following example.  let $G$ be a Lie group, and let $\gbunX$ be the category of principal $G$-bundles with connection over $X$. Further, we denote by $\mathcal{B}G$ the Lie groupoid with one object and morphisms $G$, by $G\text{-}\mathrm{Tor}$ the category of $G$-torsors, and by $i: \mathcal{B}G \to G\text{-}\mathrm{Tor}$ the functor that sends the single object of $\mathcal{B}G$ to $G$, regarded as a $G$-torsor over itself. Then, we have:

\begin{theorem}[{{\cite[Theorem 5.8]{schreiber3}}}]
\label{th3}
\label{th2}
Let $X$ be a smooth manifold. The assignment
\begin{equation*}
\gbunX \to \transportX{i}{1}{\mathcal{B}G}{G\text{-}\mathrm{Tor}}: (P,\omega) \mapsto F_{P,\omega}
\end{equation*}
defined by
\begin{equation*}
F_{P,\omega}(x) := P_x \quand
F_{P,\omega}(\gamma) := \tau_{\gamma}\text{,}
\end{equation*}
where $x\in X$, $\gamma\in PX$, and $\tau_\gamma$ is the parallel transport of $\omega$ along $\gamma$, establishes a surjective equivalence of categories.
\end{theorem}

\setsecnumdepth{3}

\section{Transport 2-Functors}

\label{sec6_3}

In this section we introduce the central definition of this paper: transport 2-functors. For this purpose, we define in Section \ref{sec:smoothdescent} a 2-category of smooth descent data, based on the notions of smooth 2-functors and transport functors. In Section \ref{sec:def} we define transport 2-functors as those 2-functors that correspond to smooth descent data under the equivalence of Theorem \ref{th:locglobequiv}.  Section \ref{sec:features} describes some basic  properties of transport 2-functors, and in Section \ref{sec5_1} we construct an explicit example.

\subsection{Smooth Descent Data}

\label{sec:smoothdescent}

In this section we select a sub-2-category $\transsmooth{i}{2}{\mathrm{Gr}}$ of \emph{smooth descent data} in the 2-category $\trans{i}{\mathrm{Gr}}{2}$ of descent data described in Section \ref{sec:loctriv}.  If $(\mathrm{triv},g,\psi,f)$ is a descent object, we demand that the strict 2-functor $\mathrm{triv}: \mathcal{P}_2(Y) \to \mathrm{Gr}$ has to be smooth in the sense of Section \ref{sec3_1}, i.e. internal to diffeological spaces. Imposing smoothness conditions for the \pt\ $g$ and the modifications $\psi$ and $f$ is more subtle since they do not take  values  in the Lie 2-category $\mathrm{Gr}$ but in the 2-category $T$ which is not assumed to be a Lie 2-category.

Briefly, we proceed in the following two steps. We explain first how the  \pt
\begin{equation*}
g: \pi_1^{*}\mathrm{triv}_i \to \pi_2^{*}\mathrm{triv}_i
\end{equation*}
can be viewed as a certain functor $\holo(g)$ defined on $\mathcal{P}_1(Y^{[2]})$. Secondly,  we impose the  condition that $\mathscr{F}(g)$ is a transport functor. A little motivation might be the observation that $\holo(g)$  corresponds then (at least in some cases, by Theorem \ref{th3}) to a principal bundle with connection  over $Y^{[2]}$ -- one of the well-known ingredients of a (bundle) gerbe, see Sections \ref{sec3_2} and \ref{sec3_3}.

Let us first explain in general how a \pt\ between two 2-functors can be viewed as a functor. We consider  2-functors $F$ and $G$  between 2-categories $S$ and $T$.
Since a \pt\ $\rho:F\to G$ assigns 1-morphisms in $T$ to objects in $S$ and 2-morphisms in $T$ to 1-morphisms in $S$, the general idea is to construct a category $S_{0,1}$ consisting of objects and 1-morphisms of $S$ and a category $\Lambda T$ consisting of 1-morphisms and 2-morphisms of $T$ such that $\rho$ yields a functor
\begin{equation*}
\holo(\rho): S_{0,1} \to \Lambda T\text{.}
\end{equation*}
We assume that $S$ is strict, so that forgetting its 2-morphisms produces a well-defined category $S_{0,1}$. The construction of  the category $\Lambda T$ is more involved. 

If $T$ is strict, the objects of $\Lambda T$ are the 1-morphisms of $T$. A morphism between objects $f\maps X_f \to Y_f$ and $g:X_g \to Y_g$ is a pair of 1-morphisms $x:X_f \to X_g$ and $y:Y_f \to Y_g$ and a 2-morphism
\begin{equation}
\label{36}
\alxydim{}{X_f \ar[r]^{x} \ar[d]_{f} & X_g \ar@{=>}[dl]|{\varphi} \ar[d]^{g} \\ Y_f \ar[r]_{y} & Y_g\text{.}}
\end{equation}
This gives indeed a category $\Lambda T$, whose composition is defined by putting the diagrams \erf{36} next to each other. Clearly, any strict 2-functor $f:T' \to T$ induces a functor $\Lambda f\maps\Lambda T' \to \Lambda T$. For a more detailed discussion of these constructions we refer the reader to Section 4.2 of \cite{schreiber5}.

Now let $\rho:F \to G$ be a \pt\ between two strict 2-functors from $S$ to $T$. Sending an object $X$ in $S$ to the 1-morphism $\rho(X)$ and sending a 1-morphism $f$ in $S$ to the 2-morphism $\rho(X)$ now yields a functor \begin{equation*}
\holo(\rho): S_{0,1} \to \Lambda T\text{.}
\end{equation*}
It respects the composition due to axiom (T1) for $\rho$ and the identities due to \cite[Lemma A.7]{schreiber6}. Moreover, a modification $\mathcal{A}:\rho_1\Rightarrow \rho_2$ defines a natural transformation $\holo(\mathcal{A})\maps \holo(\rho_1) \Rightarrow \holo(\rho_2)$, so that the result is a functor
\begin{equation}
\label{8}
\holo: \mathrm{Hom}(F, G) \to \mathrm{Funct}(S_{0,1},\Lambda T)
\end{equation} 
between the category of \pt s between $F$ and $G$ and the category of functors from $S_{0,1}$ to $\Lambda T$,
for $S$ and $T$ strict 2-categories and $F$ and $G$ strict 2-functors.

In case that the 2-category $T$ is not strict, the construction of $\Lambda T$ suffers from the fact that the composition is not longer associative.  The situation becomes treatable if one requires the objects $X_{f},Y_f$ and $X_g,Y_g$ and the 1-morphisms $x$ and $y$ in \erf{36} to be contained the image of a strict 2-category $T^{\mathrm{str}}$ under some 2-functor $i:T^{\mathrm{str}} \to T$. \begin{figure}[h]
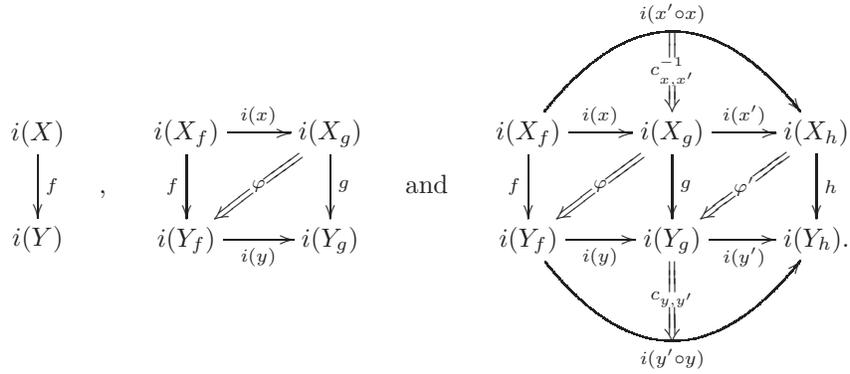

\begin{equation*}
\alxydim{}{i(X) \ar[d]^{f} \\ i(Y)}
\quad\text{, }\quad
\alxydim{}{i(X_f) \ar[r]^{i(x)} \ar[d]_{f} & i(X_g) \ar@{=>}[dl]|{\varphi} \ar[d]^{g} \\ i(Y_f) \ar[r]_{i(y)} & i(Y_g)}
\quad\text{ and }\quad
\alxydim{}{i(X_f) \ar@/^3.2pc/[rr]^{i(x' \circ x)}="1" \ar@{=>}"1";[r]|-{c_{x,x'}^{-1}} \ar[r]^{i(x)} \ar[d]_{f} & i(X_g) \ar@{=>}[dl]|{\varphi} \ar[d]^{g} \ar[r]^{i(x')} & i(X_h) \ar[d]^{h} \ar@{=>}[dl]|{\varphi'} \\ i(Y_f) \ar@/_3.2pc/[rr]_{i(y'\circ y)}="2" \ar@{=>}[r];"2"|{c_{y,y'}} \ar[r]_{i(y)} & i(Y_g) \ar[r]_{i(y')} & i(Y_h)\text{.}}
\end{equation*}
\caption{Objects, morphisms and the composition of the category $\Lambda_iT$ (the diagram on the right hand side ignores the associators and the bracketing of 1-morphisms). Here, $c$ is the compositor of the 2-functor $i$.}
\label{fig1}
\end{figure} The result is a category $\Lambda_iT$, in which the associativity of the composition is restored by axiom (F3) on the compositor of the 2-functor $i$. We omit a more formal definition and refer the reader to Figure \ref{fig1} for an illustration. 
For any 2-functor $f: T \to T'$, a functor
\begin{equation*}
\Lambda F: \Lambda_i T \to \Lambda_{F \circ i} T'
\end{equation*}
is induced by applying $f$ to all involved objects, 1-morphisms, and 2-morphisms.
We may now consider strict 2-functors $F$ and $G$ from $S$ to $T^{\mathrm{str}}$. Then, the functor \erf{8} generalizes straightforwardly to a functor
\begin{equation*}
\holo: \mathrm{Hom}(i \circ F, i \circ G) \to \mathrm{Funct}(S_{0,1},\Lambda_i T)
\end{equation*}
between the category of \pt s\ between $i\circ F$ and $i\circ G$ and the category of functors from $S_{0,1}$ to $\Lambda_i T$. The following lemma follows directly from the definitions. 

\begin{lemma}
\label{lem:F}
The functor $\mathscr{F}$ has the following properties:
\begin{enumerate}[(i)]
\item 
\label{lem:F9}
It is natural with respect to strict 2-functors $f:S' \to S$ in the sense that the diagram
\begin{equation*}
\alxydim{@C=1.5cm}{\mathrm{Hom}(i \circ F,i\circ G) \ar[d]_{f^{*}} \ar[r]^-{\holo} &  \mathrm{Funct}(S_{0,1},\Lambda_i T) \ar[d]^{f^{*}} \\ \mathrm{Hom}(i \circ F \circ f,i\circ G \circ f) \ar[r]_-{\holo} & \mathrm{Funct}(S'_{0,1},\Lambda_i T)}
\end{equation*}
is strictly commutative.

\item
\label{lem:F50}
It preserves the composition of pseudonatural transformations in the sense that if $F,G,H\maps S \to T^{\mathrm{str}}$ are three strict 2-functors, the diagram
\begin{equation*}
\alxydim{@C=1.1cm}{\mathrm{Hom}(i\! \circ\! G,i\!\circ\! H)\! \times \!\mathrm{Hom}(i \!\circ \!F,i\!\circ\! G) \ar[d]_{\circ} \ar[r]^-{\holo \times \holo} & \mathrm{Funct}(S_{0,1},\Lambda_i T) \!\times\! \mathrm{Funct}(S_{0,1},\Lambda_i T) \ar[d]^{\otimes} \\ \mathrm{Hom}(i \circ H,i \circ F) \ar[r]_-{\holo} & \mathrm{Funct}(S_{0,1},\Lambda_i T)}
\end{equation*}
is commutative.

\end{enumerate}
\end{lemma}

In Lemma \ref{lem:F} \erf{lem:F50}
the symbol $\otimes$ has the following meaning. The composition of morphisms in $\Lambda_iT$ was defined by putting the diagrams \erf{36} next to each other as shown in Figure \ref{fig1}. But one can also put the diagrams of appropriate morphisms on top of each other, provided that the arrow on the bottom of the upper one coincides with the arrow on the top of the lower one. This is indeed the case for the morphisms in the image of composable \pt s under $\holo \times \holo$, so that the diagram in (ii) makes sense. In a more formal context, the tensor product $\otimes$ can be discussed in the formalism of \emph{weak double categories}, but we will not stress this point.

In the following discussion the strict 2-category $S$ is the path  2-groupoid of some smooth manifold, $S=\mathcal{P}_2(X)$. Notice that  $S_{0,1}=\mathcal{P}_1(X)$ is then the path groupoid of $X$. The 2-category $T$ is the target 2-category, and the strict 2-category $T^{\mathrm{str}}$ is the Lie 2-groupoid $\mathrm{Gr}$.  

Let $(\mathrm{triv},g,\psi,f)$ be a descent object in the  descent 2-category $\trans{i}{\mathrm{Gr}}{2}$.  The \pt\ $g:\pi_1^{*}\mathrm{triv}_i \to \pi_2^{*}\mathrm{triv}_i$ induces a functor
\begin{equation*}
\holo(g): \mathcal{P}_1(Y^{[2]}) \to \Lambda_i T\text{.}
\end{equation*}
In order to impose the condition that $\mathscr{F}(g)$ is a transport functor, we will use the functor
\begin{equation*}
\Lambda i: \Lambda \mathrm{Gr} \to \Lambda_i T
\end{equation*}
as its structure Lie groupoid. 
Further, the modification $\psi:\id_{\mathrm{triv}_i} \to \Delta^{*}g$ induces via Lemma \ref{lem:F} \erf{lem:F9} a natural transformation
\begin{equation*}
\holo(\psi):\holo(\id_{\mathrm{triv}_i}) \Rightarrow \Delta^{*}\holo(g)\text{.}
\end{equation*}
Finally, the modification $f$ induces via Lemma \ref{lem:F} \erf{lem:F9} and  \erf{lem:F50} a natural transformation
\begin{equation*}
\holo(f): \pi_{23}^{*}\holo(g) \otimes \pi_{12}^{*}\holo(g) \Rightarrow \pi_{13}^{*}\holo(g)\text{.}
\end{equation*}

\begin{definition}
A descent object $(\mathrm{triv},g,\psi,f)$ is called \emph{smooth} if the following  conditions are satisfied:
\begin{enumerate}[(i)]
\item
the 2-functor $\mathrm{triv}\maps\mathcal{P}_2(Y) \to \mathrm{Gr}$ is smooth.

\item
the functor $\holo(g)$ is a transport functor with $\Lambda \mathrm{Gr}$-structure.

\item
the natural transformations $\mathscr{F}(\psi)$ and $\mathscr{F}(f)$ are morphisms between transport functors. 
\end{enumerate}

\end{definition}

In the same way we qualify smooth descent 1-morphisms and descent 2-morphisms. A descent 1-morphism 
\begin{equation*}
(h,\varepsilon):(\mathrm{triv},g,\psi,f) \to (\mathrm{triv}',g',\psi',f')
\end{equation*}
is converted into a functor
\begin{equation*}
\holo(h): \mathcal{P}_1(Y) \to \Lambda_iT
\end{equation*}
and a natural transformation
\begin{equation*}
\holo(\varepsilon): \pi_2^{*}\holo(h) \otimes \holo(g) \Rightarrow \holo(g') \otimes \pi_1^{*}\holo(h)\text{.}
\end{equation*}
We say that $(h,\epsilon)$ is \emph{smooth}, if  $\holo(h)$ is a  transport functor
with $\Lambda\mathrm{Gr}$-structure and   $\holo(\epsilon)$
is a 1-morphism between transport functors.
A descent 2-morphism $E: (h,\varepsilon) \Rightarrow (h',\varepsilon')$ is converted into a natural transformation
\begin{equation*}
\holo(E): \holo(h) \Rightarrow \holo(h')\text{,}
\end{equation*}
and we say that $E$ is \emph{smooth}, if  $\holo(E)$
is a 1-morphism between transport functors. 
We claim two obvious properties of smooth descent data:
\begin{enumerate}[(i)]

\item 
Compositions of smooth descent 1-morphisms and smooth descent 2-morphisms are again smooth, so that smooth descent data forms a sub-2-category of $\trans{i}{\mathrm{Gr}}{2}$, which we denote by $\transsmooth{i}{2}{\mathrm{Gr}}$.

\item
Pullbacks of smooth descend objects, 1-morphisms and 2-morphisms along refinements of surjective submersions are again smooth, so that the direct limit
\begin{equation*}
\transsmoothpi{i}{2}{\mathrm{Gr}}{}_M := \lim_{\overrightarrow{\pi}} \transsmooth{i}{2}{\mathrm{Gr}}
\end{equation*}
is a well-defined sub-2-category of $\mathfrak{Des}^2(i)_M$.

\end{enumerate}
In Section \ref{sec4} we show that the 2-category $\transsmoothpi{i}{2}{\mathrm{Gr}}{}_M$ of smooth descent data becomes nice and familiar upon choosing concrete examples for the structure 2-groupoid $i:\mathrm{Gr} \to T$.

\subsection{Transport 2-Functors}

\label{sec:def}

Now we come to the central definition of this paper.

\begin{definition}
Let $M$ be a smooth manifold, $\mathrm{Gr}$  be a strict Lie 2-groupoid,  $T$  be a 2-category and  $i:\mathrm{Gr} \to T$  be a 2-functor. 
A \emph{transport 2-functor on $M$ with values in $T$ and with $\mathrm{Gr}$-structure} is a 2-functor
\begin{equation*}
\mathrm{tra}:\mathcal{P}_2(M) \to T
\end{equation*}
such that there exists a surjective submersion $\pi:Y \to M$ and a $\pi$-local $i$-trivialization $(\mathrm{triv},t)$ whose  descent object $\ex{\pi}(\mathrm{tra},\mathrm{triv},t)$ is smooth. 
\end{definition}

A 1-morphism between transport 2-functors $\mathrm{tra}$ and $\mathrm{tra}'$ is a \pt\ 
 $A\maps \mathrm{tra} \to \mathrm{tra}'$
such that there exists a surjective submersion $\pi$ together with $\pi$-local $i$-trivializations of $\mathrm{tra}$ and $\mathrm{tra}'$ for which the descent 1-morphism $\ex{\pi}(A)$ is smooth. 2-morphisms are defined in the same way. 
Transport 2-functors $\mathrm{tra}:\mathcal{P}_2(M) \to T$ with $\mathrm{Gr}$-structure, 1-morphisms, and 2-morphisms form a sub-2-category of the 2-category of locally $i$-trivializable 2-functor $\mathrm{Funct}_i(\mathcal{P}_2(M),T)$, and we  denote this sub-2-category by $\transport{}{2}{\mathrm{Gr}}{T}$. We emphasize that being a transport 2-functor is a property, not additional structure. In particular, no surjective submersion or open cover  is contained in the structure of a transport 2-functor: they are manifestly \emph{globally defined} objects. 

We want to establish an equivalence between transport 2-functors and their smooth descent data. In order to achieve this equivalence we have to make a slight assumption on the 2-functor $i$. We call a 2-functor $i:\mathrm{Gr} \to T$ \emph{full and faithful}, if it induces an equivalence on Hom-categories. In particular, $i$ is full and faithful if it is an equivalence of 2-categories, which is in fact true in all examples we are going to discuss.

\begin{theorem}
\label{th:locglobeqsmooth}
\label{th4}
Let $M$ be a smooth manifold, and let $i \maps \mathrm{Gr} \to T$ be a full and faithful 2-functor. Then, the equivalence of Theorem \ref{th:locglobequiv} restricts to an equivalence
\begin{equation*}
\transport{}{2}{\mathrm{Gr}}{T} \cong \mathfrak{Des}^2(i)_M^{\infty}
\end{equation*}
between transport 2-functors and their smooth descent data.
\end{theorem}

Theorem \ref{th:locglobeqsmooth} is proved by the following two lemmata. As an intermediate step we introduce -- for a surjective submersion $\pi$ -- the sub-2-category $\loctrivfunctsmooth{i}{}{\pi}{2}$ of $\loctrivfunct{i}{}{2}$ as the preimage of $\transsmooth{i}{2}{}$ under the 2-functor $\ex{\pi}$. 

\begin{lemma}
\label{lem9}
The 2-functor $\ex{\pi}$ restricts to an equivalence of 2-categories
\begin{equation*}
\loctrivfunctsmooth{i}{}{\pi}{2} \cong \mathfrak{Des}^2_{\pi}(i)^{\infty}\text{.}
\end{equation*}
\end{lemma}

\begin{proof}
It is clear that $\ex\pi$ restricts properly. Recall from \cite[Section 3]{schreiber6}  that inverse to $\ex\pi$ is a \quot{reconstruction} 2-functor $\mathrm{Rec}_{\pi}$. In order to prove that the image of the restriction of $\mathrm{Rec}_{\pi}$ is contained in $\loctrivfunctsmooth{i}{}{\pi}{2}$ we have to check that $\ex{\pi} \circ \mathrm{Rec}_{\pi}$ restricts to an endo-2-functor of $\transsmooth{i}{2}{}$. In the proof of \cite[Lemma 4.1.2]{schreiber6} we have explicitly computed this 2-functor, and by inspection of the corresponding expressions one recognizes its image as smooth descent data.

The second part of the proof is to show that the components of two pseudonatural equivalences  $\rho:\ex{\pi} \circ \mathrm{Rec}_{\pi} \to \id$ and $\eta: \id \to \mathrm{Rec}_{\pi} \circ \ex{\pi}$ that establish the equivalence of \cite[Proposition 4.1.1]{schreiber6} are in $\mathfrak{Des}^2_{\pi}(i)^{\infty}$ and $\loctrivfunctsmooth{i}{}{\pi}{2}$, respectively. For the transformation $\rho$, this is again by inspection of the formulae in the proof of  \cite[Lemma 4.1.2]{schreiber6}. For the transformation $\eta$, we suppose $F$ is a 2-functor with a $\pi$-local $i$-trivialization $(t,\mathrm{triv})$ with smooth descent data $(\mathrm{triv},g,\psi,f)$. We have to prove that the descent 1-morphism $\ex\pi(\eta(F,t,\mathrm{triv}))$ is smooth. Indeed, according to the definition of $\eta$ given in the proof of \cite[Lemma 4.1.3]{schreiber6} it is given by the \pt\ $g$  and a modification composed from the modifications $f$ and $\psi$. The descent object is by assumption smooth, and so is $\eta(F)$. The same argument shows that the component $\eta(A)$ of a \pt\ $A:F \to F'$ with smooth descent data is smooth. 
\end{proof}

Next we go to the direct limit
\begin{equation*}
\loctrivfunctsmooth{i}{}{}{2}_M := \lim_{\overrightarrow{\pi}} \loctrivfunctsmooth{i}{}{\pi}{2}\text{.}
\end{equation*}
The equivalence of Lemma \ref{lem9} induces an equivalence
\begin{equation*}
\mathrm{Ex} : \loctrivfunctsmooth{i}{}{}{2}_M \to \mathfrak{Des}^2_{\pi}(i)^{\infty}
\end{equation*}
in the direct limit. Next we show that the 2-categories $\loctrivfunctsmooth{i}{}{}{2}_M$ and $\transport{}{2}{\mathrm{Gr}}{T}$  are equivalent. We have an evident 2-functor
\begin{equation*}
v^{\infty}:\loctrivfunctsmooth{i}{}{}{2}_M \to \transportX{}{2}{\mathrm{Gr}}{T}
\end{equation*}
induced by forgetting the chosen trivialization. 

\begin{lemma}
Under the assumption that the 2-functor $i$ is full and faithful,
the 2-functor $v^{\infty}$
is an equivalence of 2-categories. 
\end{lemma} 

\begin{proof}
It is clear that an inverse functor $w^{\infty}$ takes a given transport 2-functor and picks some smooth local trivialization for some surjective submersion $\pi:Y\to M$. It follows immediately that $v^{\infty} \circ w^{\infty}=\id$. It remains to construct a pseudonatural equivalence $\id\cong w^{\infty} \circ v^{\infty}$, i.e. a 1-isomorphism 
\begin{equation*}
A: (\mathrm{tra},\pi,\mathrm{triv},t) \to (\mathrm{tra},\pi',\mathrm{triv}',t')
\end{equation*}
in $\loctrivfunctsmooth{i}{}{}{2}_M$, where the original $\pi$-local trivialization $(\mathrm{triv},t)$ has been forgotten and replaced by a new $\pi'$-local trivialization $(\mathrm{triv}',t')$. But since the 1-morphisms in $\loctrivfunctsmooth{i}{}{}{2}_M$ are just \pt\ between the 2-functors ignoring the trivializations,  we only have to prove that the identity \pt
\begin{equation*}
A:=\id_{\mathrm{tra}}: \mathrm{tra} \to \mathrm{tra}
\end{equation*}
of a transport 2-functor $\mathrm{tra}$ has smooth descent data $(h,\varepsilon)$ with respect to \textit{any} two  trivializations $(\pi,\mathrm{triv},t)$ and $(\pi',\mathrm{triv}',t')$. 

The first step is to choose a refinement $\zeta: Z \to Y \times_M Y'$ of the common refinement of the to surjective submersions. One can choose $Z$ such that is has contractible connected components. If $c: Z \times[0,1] \to Z$ is such a contraction, it defines for each point $z\in Z$ a path $c_z: z \to z_k$ that moves $z$ to the distinguished point $z_k$ to which the component of $Z$ that contains $z$ is contracted. It further defines for each path $\gamma:z_1 \to z_2$ a bigon $c_{\gamma}:\gamma \Rightarrow c_{z_2}^{-1} \circ c_{z_1}$. Axiom (T2) for the pseudonatural transformation 
\begin{equation*}
h:=t' \circ \bar t: \mathrm{triv}_i \to \mathrm{triv}_i'
\end{equation*}
applied to the bigon $c_{\gamma}$ yields the commutative diagram
\begin{equation*}
\alxydim{@C=2.4cm@R=1.3cm}{h(z_2) \circ \mathrm{triv}_i(\gamma) \ar@{=>}[r]^-{h(\gamma)} \ar@{=>}[d]_{\id\circ \mathrm{triv}_i(c_{\gamma})} & \mathrm{triv}_i'(\gamma) \circ h(z_1) \ar@{=>}[d]^{\mathrm{triv}_i'(c_{\gamma}) \circ \id} \\ h(z_2) \circ \mathrm{triv}_i(c_{z_2}^{-1} \circ c_{z_1}) \ar@{=>}[r]_-{h(c_{z_2}^{-1}\circ c_{z_1})} & \mathrm{triv}_i'(c_{z_2}^{-1}\circ c_{z_1}) \circ h(z_1)\text{.}}
\end{equation*}
Notice that the 1-morphisms $h(z_j): \mathrm{triv}_i(z_j) \to \mathrm{triv}_i'(z_j)$ have by assumption preimages $\kappa_j: \mathrm{triv}(z_j) \to \mathrm{triv}'(z_j)$ under $i$ in $\mathrm{Gr}$, and that the 2-morphism $h(c_{z_2}^{-1}\circ c_{z_1})$ also has a preimage $\Gamma$ in $\mathrm{Gr}$. Thus,
\begin{equation*}
h(\gamma) = i \left ( (\mathrm{triv}'(c_{\gamma}) \circ \id)^{-1} \bullet \Gamma \bullet (\id \circ \mathrm{triv}(c_{\gamma})) \right )\text{.}
\end{equation*}
This is nothing but the Wilson line $\mathcal{W}^{\holo(h),\Lambda i}_{z_1,z_2}$ of the functor $\holo(h)$ and it is smooth since $\mathrm{triv}$ and $\mathrm{triv}'$ are smooth 2-functors. Hence, by Theorem 3.12 in \cite{schreiber3}, $\holo(h)$ is a transport functor with $\Lambda \mathrm{Gr}$-structure. 

It remains to prove that the modification  $\varepsilon: \pi_2^{*}h \circ g \Rightarrow g' \circ \pi_1^{*}h$
induces a morphism $\holo(\varepsilon)$ of transport functors. This simply follows from the general fact that under the assumption that the functor $i: \mathrm{Gr} \to T$ is full, every natural transformation $\eta$ between transport functors with $\mathrm{Gr}$-structure is a morphism of transport functors. We have not stated that explicitly in \cite{schreiber3} but it can easily be deduced from the naturality conditions on trivializations $t$ and $t'$ and on $\eta$, evaluated for paths with a fixed starting point. \end{proof}

With Theorem \ref{th4} we have established an equivalence between globally defined transport 2-functors and locally defined smooth descent data. In Section \ref{sec4} we will identify smooth descent data with various models of gerbes with connections. Under this identifications, Theorem \ref{th4} describes the relation between these gerbes with connections and their parallel transport.

\subsection{Some Features of Transport 2-Functors}

\label{sec:features}

In this section we provide several features of transport 2-functors.

\subsubsection{Operations on Transport 2-Functors}

\label{operations}

It is straightforward to see that transport 2-functors allow a list of natural operations.
\begin{enumerate}[(i)]
\item 
\emph{Pullbacks}: Let $f:M \to N$ be a smooth map. The pullback $f^{*}\mathrm{tra}$ of any transport 2-functor on $N$ is a transport 2-functor on $M$.

\item
\emph{Tensor products}:
Let $\otimes: T \times T \to T$ be a monoidal structure on a 2-category $T$. For transport 2-functors $\mathrm{tra}_1,\mathrm{tra}_2:\mathcal{P}_2(M) \to T$ with $\mathrm{Gr}$-structure, the pointwise tensor product $\mathrm{tra}_1 \otimes \mathrm{tra}_2: \mathcal{P}_2(M) \to T$ is again a transport 2-functor with $\mathrm{Gr}$-structure, and makes the 2-category $\transport{}{2}{\mathrm{Gr}}{T}$ a monoidal 2-category.

\item
\emph{Change of the target 2-category}:
Let $T$ and $T'$ be two target 2-categories equipped with 2-functors $i:\mathrm{Gr} \to T$ and $i': \mathrm{Gr} \to T'$,  and let $F:T \to T'$ be a 2-functor together with a \pe
\begin{equation*}
\rho: F \circ i \to i'\text{.}
\end{equation*}
If $\mathrm{tra}:\mathcal{P}_2(M) \to T$ is a transport 2-functor with $\mathrm{Gr}$-structure, $F \circ \mathrm{tra}$ is also a transport 2-functor with $\mathrm{Gr}$-structure. In particular, this is the case for $i':= F \circ i$ and $\rho=\id$.

\item
\emph{Change of the structure 2-groupoid}: Let $\mathrm{tra}:\mathcal{P}_2(M) \to T$ be a transport 2-functor with $\mathrm{Gr}$-structure, for a 2-functor $i:\mathrm{Gr} \to T$ which is a composition
\begin{equation*}
\alxydim{}{\mathrm{Gr} \ar[r]^{F} & \mathrm{Gr}' \ar[r]^{i'} & T}
\end{equation*}
in which $F$ is a smooth 2-functor.
Then, $\mathrm{tra}$ is also a transport 2-functor with $\mathrm{Gr}'$-structure, since for any local $i$-trivialization $(\mathrm{triv},t)$ of $\mathrm{tra}$ we have a local $i'$-trivialization $(F \circ \mathrm{triv},t)$.  

\end{enumerate}

\subsubsection{Structure Lie 2-Groups}

We describe some examples of Lie 2-groupoids and outline the role of the corresponding transport 2-functors. First we recall the following generalization of a Lie group.

\begin{definition}
\label{def4}
A \emph{Lie 2-group} is a strict monoidal Lie category $(\mathfrak{G},\boxtimes,\trivlin)$ together with a smooth functor $inv:\mathfrak{G} \to \mathfrak{G}$ such that \begin{equation*}
X \boxtimes inv(X) = \trivlin = inv(X) \boxtimes X
\quad\text{ and }\quad
f \boxtimes inv(f) = \id_{\trivlin} = inv(f) \boxtimes f
\end{equation*}
for all objects $X$ and all morphisms $f$ in $\mathfrak{G}$.
\end{definition}

The strict monoidal category $(\mathfrak{G},\boxtimes,\trivlin)$ defines a strict 2-category $\mathcal{B}\mathfrak{G}$ with a single object \cite[Example A.2]{schreiber6}. 
The additional functor $inv$ assures that $\mathcal{B}\mathfrak{G}$ is a strict 2-\emph{groupoid}. All our examples in Section \ref{sec4} discuss transport 2-functors with $\mathcal{B}\mathfrak{G}$-structure, for $\mathfrak{G}$ a Lie 2-group.

Lie 2-groups can be obtained from the following structure.

\begin{definition}
\label{def:crossedmodule}
A \emph{smooth crossed module} is a quadruple $(G,H,t,\alpha)$ of Lie groups $G$ and $H$,  of a Lie group homomorphism  $t: H \to G$, and of a smooth left action $\alpha: G \times H \to H$ by Lie group homomorphisms such that
\begin{itemize}
\item[a)]
$t(\alpha(g,h))=gt(h)g^{-1}$ for all $g\in G$ and $h\in H$.
\item[b)] 
$\alpha(t(h),x) = hxh^{-1}$ for all $h,x\in H$.
\end{itemize}
\end{definition}

The construction of a Lie 2-group $\mathfrak{G}=\mathfrak{G}(G,H,t,\alpha)$ from a given smooth crossed module $(G,H,t,\alpha)$ can be found in the Appendix of \cite{schreiber5}. We shall explicitly  describe the corresponding  Lie 2-groupoid $\mathcal{B}\mathfrak{G}$. It has one object denoted  $\ast$ . A 1-morphism is a group element $g\in G$, the identity 1-morphism is the neutral element, and the composition of 1-morphisms is the multiplication, $g_2 \circ g_1 := g_2g_1$.  
The 2-morphisms are pairs $(g,h)\in G \times H$, considered as 2-morphisms
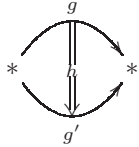
\begin{equation*}
\bigon{\ast}{\ast}{g}{g'}{h}
\end{equation*}
with $g':=t(h)g$.
The vertical composition is
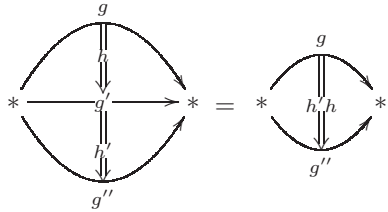
\begin{equation*}
\alxydim{@C=2cm}{\ast \ar[r]|{g'}="2" \ar@/^2.5pc/[r]^{g}="1" \ar@/_2.5pc/[r]_{g''}="3" \ar@{=>}"1";"2"|{h}\ar@{=>}"2";"3"|{h'} & \ast}
 = \bigon{\ast}{\ast}{g}{g''}{h'h}
\end{equation*}
with $g'=t(h)g$ and $g''=t(h')g'=t(h'h)g$, and the horizontal composition is
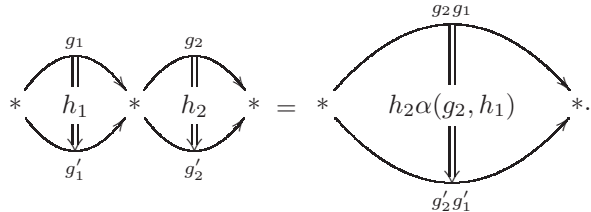
\begin{equation*}
\alxydim{@C=1.2cm}{\ast \ar@/^1.5pc/[r]^{g_{1}}="1" \ar@/_1.5pc/[r]_{g'_{1}}="2" \ar@{=>}"1";"2"|*+{h_{1}} & \ast \ar@/^1.5pc/[r]^{g_2}="3" \ar@/_1.5pc/[r]_{g'_2}="4" \ar@{=>}"3";"4"|*+{h_2} & \ast} = \alxydim{@C=3cm}{\ast \ar@/^2.5pc/[r]^{g_2g_1}="1" \ar@/_2.5pc/[r]_{g_2'g_1'}="2" \ar@{=>}"1";"2"|*+{h_2\alpha(g_2,h_1)} & \ast\text{.}}
\end{equation*}
Summarizing, one can go from smooth crossed modules to Lie 2-groups, and then to Lie 2-groupoids. 

\begin{example}\
\label{ex:2groups}
\begin{enumerate}[(i)]
\item
Let $A$ be an abelian Lie group. A smooth crossed module is defined by  $G=\lbrace 1 \rbrace$  and $H:= A$. This fixes the maps to $t(a):=1$ and $\alpha(1,a):=a$. Notice that axiom b) is  satisfied because $A$ is abelian. The associated Lie 2-group is denoted by $\mathcal{B} A$. Transport 2-functors with $\mathcal{BB}A$-structure play the role of \emph{abelian gerbes with connection}; see Section \ref{sec3_2}. 

\item
Let $G$ be a Lie group. A smooth crossed module is defined by  $H := G$, $t=\id$ and $\alpha(g,h)\df ghg^{-1}$. The associated Lie 2-group is  denoted by  $\mathcal{E}G$. This notation is devoted to the fact that the geometric realization of the nerve of the category $\mathcal{E}G$ yields the universal $G$-bundle $EG$. Transport 2-functors with $\mathcal{BE}G$-structure arise as the \emph{curvature} of transport functors; see Section \ref{sec5_1}.

\item
\label{ex:2groups:aut}
Let $H$ be a connected Lie group, so that the group of Lie group automorphisms of $H$ is again a Lie group $G:=\mathrm{Aut}(H)$. The definitions  $t(h)(x):=hxh^{-1}$ and $\alpha(\varphi,h):= \varphi(h)$ yield a smooth crossed module whose associated Lie 2-group is denoted by $\mathrm{AUT}(H)$, it is called the  \emph{automorphism 2-group} of $H$. Transport 2-functors with $\mathcal{B}\mathrm{AUT}(H)$-structure play the role \emph{non-abelian gerbes with connection}; see Section \ref{sec3_3}. 

\item
Let 
\begin{equation*}
\alxydim{}{1 \ar[r] & N \ar[r]^-{t} & H \ar[r]^{p} & G \ar[r] & 1}
\end{equation*}
be an  exact sequence of Lie groups, not necessarily central. There is a canonical action $\alpha$ of $H$ on $N$ defined by requiring
\begin{equation*}
t(\alpha(h,n)) = ht(n)h^{-1}\text{.}
\end{equation*}
This defines a smooth crossed module, whose associated Lie 2-group we  denote by $\mathfrak{N}$.  Transport 2-functors with $\mathcal{B}\mathfrak{N}$-structure correspond to (non-abelian) \emph{lifting gerbes}. They generalize the abelian lifting gerbes \cite{brylinski1,murray} for central extensions to arbitrary short exact sequences of Lie groups.  
\end{enumerate}
\end{example}

\subsubsection{Transgression to Loop Spaces}

Transport 2-functors on a smooth manifold $M$ induce interesting structure on the loop space $LM$. This comes from the fact that there is a canonical smooth functor
\begin{equation*}
\ell: \mathcal{P}_1(LM) \to \Lambda \mathcal{P}_2(M)
\end{equation*}
expressing  that a point in $LM$ is just a particular path in $M$, and that a path in $LM$ is just a particular bigon in $M$ \cite[Section 4.2]{schreiber5}. If $\mathrm{tra}: \mathcal{P}_2(M) \to T$ is a transport 2-functor, then the composition of $\ell$ with 
\begin{equation*}
\Lambda\mathrm{tra}: \Lambda \mathcal{P}_2(M) \to \Lambda_{\mathrm{tra}}T
\end{equation*}
yields a functor
\begin{equation*}
\mathscr{T}_{\mathrm{tra}} := \Lambda\mathrm{tra} \circ \ell: \mathcal{P}_1(LM) \to \Lambda_{\mathrm{tra}}T
\end{equation*}
that we call the \emph{transgression of $\mathrm{tra}$ to the loop space}. In order abbreviate the discussion of the functor $\mathscr{T}_{\mathrm{tra}}$ we make three simplifying assumptions:
\begin{enumerate}[(i)]

\item
We restrict our attention to the based loop  space $\Omega M$ (for some fixed base point) and identify $\mathscr{T}_{\mathrm{tra}}$ with its pullback along the embedding $\iota: \Omega M \to LM$.

\item 
We assume that there exists a surjective submersion $\pi:Y \to M$ for which $\mathrm{tra}$ admits smooth local trivializations  and for which $\Omega\pi:\Omega Y \to \Omega M$ is also a surjective submersion. 

\item
We assume that the target 2-category $T$ is strict, so that $\Lambda T$ is the target category of the functor $\mathscr{T}_{\mathrm{tra}}$.

\end{enumerate}

\begin{proposition}
\label{prop:trans}
Let $\mathrm{tra}: \mathcal{P}_2(M) \to T$ be a transport 2-functor with $\mathrm{Gr}$-structure  satisfying (ii) and (iii).
Then,
\begin{equation*}
\mathscr{T}_{\mathrm{tra}}: \mathcal{P}_1(\Omega M) \to \Lambda T
\end{equation*}
is a transport functor with $\Lambda\mathrm{Gr}$-structure.
\end{proposition}

\begin{proof}
Let $t: \pi^{*}\mathrm{tra} \to \mathrm{triv}_i$ be a $\pi$-local $i$-trivialization of $\mathrm{tra}$ for $\pi$ a surjective submersion satisfying (ii). A local trivialization $\tilde t$ of $\mathscr{T}_{\mathrm{tra}}$ is given by 
\begin{equation*}
\alxydim{}{\mathcal{P}_1(\Omega Y) \ar[d]_{\ell} \ar[r]^{(\Omega \pi)_{*}} & \mathcal{P}_1(\Omega M) \ar[d]^{\ell} \\ \Lambda \mathcal{P}_2(Y) \ar[d]_{\Lambda\mathrm{triv}} \ar[r]|{\pi_{*}} & \Lambda \mathcal{P}_2(M) \ar@{=>}[dl]|{\Lambda t} \ar[d]^{\Lambda\mathrm{tra}} \\ \Lambda \mathrm{Gr} \ar[r]_{\Lambda i} & \Lambda T}
\end{equation*}
in which the upper subdiagram is commutative on the nose. If $g: \pi_1^{*}\mathrm{triv}_i \to \pi_2^{*}\mathrm{triv}_i$ is the \pt\ in the smooth descent object $\ex{\pi}(\mathrm{tra},t,\mathrm{triv})$, and $\tilde g$ is the natural transformation in the descent object $\ex{\pi}(\mathscr{T}_{\mathrm{tra}}, \tilde t, \ell^{*}\Lambda \mathrm{triv})$ associated to the above trivialization, we find $\tilde g = \ell^{*}\Lambda g$. 
Since $\holo(g)$ is a transport 2-functor with $\Lambda\mathrm{Gr}$-structure, it has smooth Wilson lines \cite{schreiber3}: for a fixed point $\alpha\in Y^{[2]}$ there exists a smooth natural transformation $g': \pi_1^{*}\ell^{*}\Lambda\mathrm{triv} \to \pi_2^{*}\ell^{*}\Lambda\mathrm{triv}$ with $g=i(g')$. This shows that $\tilde g$ factors through a smooth natural transformation $\ell^{*}\Lambda g'$, so that $\mathscr{T}_{\mathrm{tra}}$ is a transport functor.
\end{proof}

Having in mind that transport functors correspond to fibre bundles with connection, Proposition \ref{prop:trans} shows that transport 2-functors on a manifold $M$ naturally induce fibre bundles with connection on the loop space $\Omega M$. In general, these are so-called groupoid bundles \cite{moerdijk, schreiber3}, whose structure groupoid is $\Lambda\mathrm{Gr}$. However, in the abelian case, i.e. $\mathrm{Gr}=\mathcal{BB}A$ for an abelian Lie group $A$, we have $\Lambda\mathrm{Gr} \cong \mathcal{B}A$ (see Lemma \ref{lem6} below), so that the transgression $\mathscr{T}_{\mathrm{tra}}$ is -- via Theorem \ref{th3} -- a principal $A$-bundle with connection over $\Omega M$. This fits well into Brylinski's picture of transgression \cite{brylinski1}.

\subsubsection{Curving and Curvature}

Suppose $\mathrm{tra}:\mathcal{P}_2(M) \to T$ is a transport 2-functor with $\mathcal{B}\mathfrak{G}$-structure, for $\mathfrak{G}$ some Lie 2-group coming from a smooth crossed module $(G,H,t,\alpha)$. Since such 2-functors are supposed to describe (non-abelian) gerbes with connection, we want to identify a \emph{3-form curvature}. Just as for (non-abelian) principal bundles, this curvature is only locally defined.

First we need the following fact: if $\mathfrak{g}$ and $\mathfrak{h}$ denote the Lie algebras of the Lie groups $G$ and $H$, respectively, there is a bijection
\begin{equation}
\label{bij}
\bigset{3.2cm}{Smooth 2-functors\\$F: \mathcal{P}_2(X) \rightarrow \mathcal{B}\mathfrak{G}$} \cong 
\bigset{5.8cm}{Pairs $(A,B) \in \Omega^1(X,\mathfrak{g}) \times \Omega^2(X,\mathfrak{h})$ satisfying $t_{*}(B) = \mathrm{d}A + [A \wedge A]$}\text{,}
\end{equation}
where $t_{*}: \mathfrak{h} \to \mathfrak{g}$ is the differential of $t$.
This bijection is the lowest level of an equivalence of 2-categories that we will review in more detail in Section \ref{sec4_1}; see Theorem \ref{th1}.

\begin{definition}
\label{def:curv}
Let $\mathrm{tra}: \mathcal{P}_2(M) \to T$ be a transport functor with $\mathcal{B}\mathfrak{G}$-structure over $M$, let $\pi\maps Y \to M$ be a surjective submersion, and let $(t,\mathrm{triv})$ be a $\pi$-local trivialization with smooth descent data.
\begin{enumerate}[(i)]

\item
The differential forms $A \in \Omega^1(Y,\mathfrak{g})$ and $B \in \Omega^2(Y,\mathfrak{h})$ that correspond to the smooth 2-functor $\mathrm{triv}$ under the above bijection, are called the \emph{1-curving} and the \emph{2-curving} of $\mathrm{tra}$. 
\item
The 3-form
\begin{equation*}
\mathrm{curv}(\mathrm{tra}) = \mathrm{d}B + \alpha_{*}(A \wedge B)\in\Omega^3(Y,\mathfrak{h})\text{,}
\end{equation*}
where $\alpha_{*}: \mathfrak{g} \times \mathfrak{h} \to \mathfrak{h}$ is the differential of the action $\alpha: G \times H \to H$ of the crossed module, is called the \emph{curvature} of $\mathrm{tra}$.

\end{enumerate}

\end{definition}

We recall that we  called a 2-functor $\mathrm{tra}:\mathcal{P}_2(M) \to T$ \emph{flat} if it factors through the projection $\mathcal{P}_2(M) \to \Pi_2(M)$ of thin homotopy classes of bigons to homotopy classes. The next proposition shows that this notion of flatness is equivalent to the vanishing of the curvature. 

\begin{proposition}
\label{prop3}
Suppose that the 2-functor $i:\mathcal{B}\mathfrak{G} \to T$ is injective on 2-morphisms.
A transport 2-functor $\mathrm{tra}:\mathcal{P}_2(M) \to T$ with $\mathcal{B}\mathfrak{G}$-structure is flat if and only if its local curvature 3-form $\mathrm{curv}(\mathrm{tra})\in\Omega^3(Y,\mathfrak{h})$ with respect to any smooth local trivialization vanishes.
\end{proposition}

\begin{proof}
We proceed in two parts. (a): $\mathrm{curv}(\mathrm{tra})$ vanishes if and only if $\mathrm{triv}$ is a flat 2-functor, and (b): $\mathrm{tra}$ is flat if and only if $\mathrm{triv}$ is flat. The claim (a) follows from Lemma A.11 in \cite{schreiber5}. To see (b) consider two bigons $\Sigma_1: \gamma \Rightarrow \gamma'$  and $\Sigma_2: \gamma \Rightarrow \gamma'$ in $Y$ which are smoothly homotopic so that they define the same element in $\Pi_2(Y)$. Suppose $\mathrm{tra}$ is flat and let $\Sigma := \Sigma_2^{-1} \bullet \Sigma_1$. Axiom (T2) for the trivialization $t$ is then
\begin{equation*}
\alxydim{@C=1.8cm@R=1.6cm}{t(y) \circ \pi^{*}\mathrm{tra}(\gamma) \ar@{=>}[r]^-{t(\gamma)} \ar@{=>}[d]_{\id_{t(y)}\circ \pi^{*}\mathrm{tra}(\Sigma)} & \mathrm{triv}_i(\gamma) \circ t(x) \ar@{=>}[d]^{\mathrm{triv}_i(\Sigma) \circ \id_{t(x)}} \\ t(y) \circ \pi^{*}\mathrm{tra}(\gamma) \ar@{=>}[r]_-{t(\gamma)} & \mathrm{triv}_i(\gamma) \circ t(x)}
\end{equation*}
and since $\pi^{*}\mathrm{tra}(\Sigma)=\id$ by assumption it follows that $\mathrm{triv}_i(\Sigma)=\id$, i.e. $\mathrm{triv}$ is flat. Conversely, assume that $\mathrm{triv}$ is flat. The latter diagram shows that  $\mathrm{tra}$ is then flat on all bigons in the image of $\pi_{*}$. This is actually enough: let $h: [0,1]^3 \to M$  be a smooth homotopy between two bigons $\Sigma_1$ and $\Sigma_2$ which are not in the image of $\pi_{*}$. Like explained in Appendix A.3 of \cite{schreiber5} the cube $[0,1]^3$ can be decomposed into small cubes such that $h$ restricts to smooth homotopies between small bigons that bound these cubes. The decomposition can be chosen so small that each of these bigons is contained in the image of $\pi_{*}$, so that $\mathrm{tra}$ assigns the same value to the source and the target bigon of each small cube. By 2-functorality of $\mathrm{tra}$, this infers $\mathrm{tra}(\Sigma_1) = \mathrm{tra}(\Sigma_2)$.
\end{proof}

\subsection{Curvature 2-Functors}

\label{sec5_1}
\label{sec3_4}

In this section we provide a class of examples of transport 2-functors coming from transport functors, i.e. fibre bundle with connections. If $P$ is a principal $G$-bundle with connection $\omega$ over $M$, one can compare the parallel transport maps along two paths $\gamma_1,\gamma_2:x \to y$ by an automorphism of $P_y$, namely the holonomy around the loop $\gamma_2 \circ \gamma_1^{-1}$,
\begin{equation*}
\tau_{\gamma_2} = \mathrm{Hol}_{\omega}(\gamma_2 \circ \gamma_1^{-1}) \circ \tau_{\gamma_1}\text{.}
\end{equation*}
If the paths $\gamma_1$ and $\gamma_2$ are the source and the target of a bigon $\Sigma:\gamma_1\Rightarrow \gamma_2$, this holonomy is  related to the curvature of $\nabla$. So, a principal $G$-bundle with connection does not only assign fibres $P_x$ to points $x\in M$ and parallel transport maps $\tau_{\gamma}$ to paths, it also assigns a curvature-related quantity to bigons $\Sigma$.

Under the equivalence between principal $G$-bundles with connection and transport functors on $X$ with $\mathcal{B}G$-structure (Theorem \ref{th3}), the principal bundle $(P,\omega)$ corresponds to the transport functor
\begin{equation*}
\mathrm{tra}_P: \mathcal{P}_1(M) \to G\text{-}\mathrm{Tor}
\end{equation*}
that assigns the fibres $P_x$ to points $x\in M$ and the parallel transport maps $\tau_{\gamma}$ to paths $\gamma$. Adding an assignment for bigons  yields a \quot{curvature 2-functor}
\begin{equation*}
K(\mathrm{tra}_P): \mathcal{P}_2(M) \to \ttc{G\text{-}\mathrm{Tor}}
\end{equation*}
where $\ttc{G\text{-}\mathrm{Tor}}$ is the category $G\text{-}\mathrm{Tor}$ regarded as a strict 2-category with a unique 2-morphism between each pair of 1-morphisms. The uniqueness of the 2-morphisms expresses the fact that the curvature is  determined by the parallel transport.

More generally, let us  start with a  transport functor $\mathrm{tra}: \mathcal{P}_1(M) \to T$ with $\mathcal{B}G$-structure for some Lie group $G$ and some functor $i:\mathcal{B}G \to T$. 
\begin{comment}
We recall from \cite{schreiber3} that this means that there exists a surjective submersion $\pi:Y \to M$, a functor $\mathrm{triv}:\mathcal{P}_1(Y) \to \mathcal{B}G$ and a natural equivalence
\begin{equation*}
t: \pi^{*}\mathrm{tra} \to \mathrm{triv}_i
\end{equation*}
such that its descent data is smooth: the functor $\mathrm{triv}$ is smooth, and the natural transformation $g:\pi_1^{*}\mathrm{triv}_i \to \pi_2^{*}\mathrm{triv}_i$ factors through a smooth natural transformation $\tilde g: \pi_1^{*}\mathrm{triv} \to \pi_2^{*}\mathrm{triv}$, i.e. $g(\alpha) = i(\tilde g(\alpha))$ for every $\alpha\in Y^{[2]}$.
\end{comment}
The \emph{curvature 2-functor} of $\mathrm{tra}$ is the strict 2-functor
\begin{equation*}
K(\mathrm{tra}): \mathcal{P}_2(M) \to \ttc{T}
\end{equation*}
which is on objects and 1-morphisms equal to $\mathrm{tra}$ and  on 2-morphisms  determined by the fact that $\ttc{T}$ has a exactly one 2-morphism between each pair of 1-morphisms. In the same way, we obtain a strict 2-functor
\begin{equation*}
K(i): \ttc{\mathcal{B}G} \to \ttc{T}
\end{equation*}
We observe that the Lie 2-groupoids $\ttc{\mathcal{B}G}$ and $\mathcal{BE}G$ (see Section \ref{sec:features}) are canonically isomorphic under the assignment
\begin{equation*}
\bigon{\ast}{\ast}{g_1}{g_2}{*}\quad\mapsto\quad\bigon{\ast}{\ast\text{,}}{g_1}{g_2}{g_2g_1^{-1}}
\end{equation*}
so that we obtain a 2-functor $\mathcal{BE}G \to \ttc{\mathcal{B}G} \to \ttc{T}$. Now we are in the position to introduce our explicit example of a transport 2-functor.

\begin{theorem}
\label{th:lem4}
The curvature 2-functor $K(\mathrm{tra})$ is a  transport 2-functor with $\mathcal{BE}G$-structure.
\end{theorem}

\begin{proof}
We construct a local trivialization of $K(\mathrm{tra})$ starting with a local trivialization $(\mathrm{triv},t)$ of $\mathrm{tra}$ with respect to some surjective submersion $\pi:Y \to M$.  Let $\mathrm{dtriv}\maps\mathcal{P}_2(Y) \to \mathcal{BE}G$ be the derivative 2-functor associated to $\mathrm{triv}$ \cite{schreiber5}: on objects and 1-morphisms it is given by $\mathrm{triv}$, and it sends every bigon $\Sigma:\gamma_1\Rightarrow \gamma_2$ in $Y$ to the unique 2-morphism in $\mathcal{BE}G$ between the images of $\gamma_1$ and $\gamma_2$ under $\mathrm{triv}$. A \pe
\begin{equation*}
K(t): \pi^{*}K(\mathrm{tra}) \to K(i) \circ \mathrm{dtriv}
\end{equation*}
is defined as follows. Its component at a point $a\in Y$ is the 1-morphism
\begin{equation*}
K(t)(a):=t(a): \mathrm{tra}(\pi(a)) \to i(*)
\end{equation*}
in $T$. Its component $t(\gamma)$ at a path $\gamma:a \to b$ is the unique 2-morphism in $\ttc{T}$. Notice that since $t$ is a natural transformation, we have a commutative diagram
\begin{equation*}
\alxydim{@C=1.5cm}{\mathrm{tra}(\pi(a)) \ar[r]^{\mathrm{tra}(\pi(\gamma))} \ar[d]_{t(a)} & \mathrm{tra}(\pi(b)) \ar[d]^{t(b)} \\ i(*) \ar[r]_{\mathrm{triv}_i(\gamma)} & i(*)}
\end{equation*}
meaning that $t(\gamma)=\id$. This defines the \pt\ $t$ as required. 

Now we assume that the descent data $(\mathrm{triv},g_t)$ associated to the local trivialization $(\mathrm{triv},t)$ is smooth, and show that then also the descent object $(\mathrm{dtriv},g_{K(t)},\psi,f)$ is smooth. As observed in \cite{schreiber5}, the derivative 2-functor $\mathrm{dtriv}$ is smooth if and only if $\mathrm{triv}$ is smooth. To extract the remaining descent data according to the procedure described in Section \ref{sec:loctriv}. It turns out that the only non-trivial descent datum is the \pt\ 
\begin{equation*}
g_{K(t)}: \pi_1^{*}\mathrm{dtriv}_{K(i)} \to \pi_2^{*}\mathrm{dtriv}_{K(i)}\text{.}
\end{equation*}
Its component at a point $\alpha\in Y^{[2]}$ is given by $g_{K(t)}(\alpha) := g_t(\alpha)$, and its component at some path $\Theta:\alpha \to \alpha'$ is  the identity. 

The last step is to show that 
\begin{equation*}
\holo(g_{K(t)}): \mathcal{P}_1(Y^{[2]}) \to \Lambda_{K(i)}\ttc{T}
\end{equation*}
is a transport functor with $\Lambda\mathcal{BE}G$-structure.
To do so we have to find a local trivialization with smooth descent data. This is here particulary simple: the functor $\holo(g_{K(t)})$ is \emph{globally} trivial in the sense that it factors through the functor
\begin{equation*}
\Lambda K(i):\Lambda\mathcal{BE}G \to \Lambda_{K(i)}\ttc{T}\text{.}
\end{equation*}
To see this we use the smoothness condition on the natural transformation $g_t$, namely that it factors through a smooth natural transformation $\tilde g_t$. We obtain a smooth pseudonatural transformation $\tilde g_{K(t)}: \pi_1^{*}\mathrm{dtriv} \to \pi_2^{*}\mathrm{dtriv}$ such that $g_{K(t)} = K(i) (\tilde g_{K(t)})$. This finally gives us
\begin{equation*}
\holo(g_{K(t)}) = \Lambda K(i) \circ \holo(\tilde g_{K(t)})
\end{equation*}  
meaning that $\holo(g_{K(t)})$ is a transport functor with $\Lambda\mathcal{BE}G$-structure.
\end{proof}

Since the value of the curvature 2-functor $K(\mathrm{tra})$ on bigons does not depend on the bigon itself but only on its source and target path, we have the following.

\begin{proposition}
\label{prop5}
The curvature 2-functor $K(\mathrm{tra})$ of any transport functor  is flat.
\end{proposition}

This proposition gains a very nice interpretation when we relate the curvature of a connection $\omega$ in a principal $G$-bundle $p:P\to M$ to the  curvature 2-functor $K(\mathrm{tra}_P)$ associated to the corresponding transport functor $\mathrm{tra}_P$. First we show:

\begin{lemma}
\label{lem2}
The curvature 2-functor $K(\mathrm{tra}_P): \mathcal{P}_2(M) \to \ttc{G\text{-}\mathrm{Tor}}$ has a canonical smooth $p$-local trivialization $(p,t,\mathrm{triv})$. The classical curvature $\mathrm{curv}(\omega) \in \Omega^2(P,\mathfrak{g})$ is the 2-curving of  $K(\mathrm{tra}_P)$ with respect to the trivialization $(p,t,\mathrm{triv})$.
\end{lemma}

\begin{proof}
As described in detail in \cite[Section 5.1]{schreiber3}, $\mathrm{tra}_P$ admits local trivializations with respect to the surjective submersion $p:P \to M$ and with smooth descent data $(\mathrm{triv}',g)$ such that the connection 1-form $\omega \in \Omega^1(P,\mathfrak{g})$ of the bundle $P$ corresponds to the smooth functor $\mathrm{triv}':\mathcal{P}_1(P) \to \mathcal{B}G$ under the bijection of \cite[Proposition 4.7]{schreiber3}. Then, by \cite[Lemma 3.5]{schreiber5}, the 2-form $B'$ associated to $\mathrm{dtriv}'$ under the bijection \erf{bij}, which is by Definition \ref{def:curv} (i) the 2-curving of $K(\mathrm{tra}_P)$, is given by
\begin{equation*}
B'=\mathrm{d}\omega + [\omega \wedge \omega]\text{.}
\end{equation*}
The latter is by definition the curvature of the connection $\omega$. 
\end{proof}

The announced  interpretation of Proposition \ref{prop5} now is  as follows: using Lemma \ref{lem2} one can  calculate the curvature $\mathrm{curv}(K(\mathrm{tra}_P))$ of the curvature 2-functor of $\mathrm{tra}_P$. The calculation involves the second Bianchi identity for the connection $\omega$ on the principal $G$-bundle $P$, and the result is
\begin{equation*}
\mathrm{curv}(K(\mathrm{tra}_P)) = 0\text{,}
\end{equation*}
which is according to Proposition \ref{prop3} an independent proof of Proposition \ref{prop5}. In other words, Proposition \ref{prop5} is equivalent to the \emph{second Bianchi identity} for connections on fibre bundles.

\section{Transport 2-Functors are Non-Abelian Gerbes}

\label{sec4}

In this section we show that transport 2-functors reproduce -- systematically, by choosing appropriate target 2-categories and  structure 2-groups -- four known concepts of gerbes with connections: Deligne cocycles and Breen-Messing cocycles (Section \ref{sec4_1}),  abelian bundle gerbes (Section \ref{sec3_2}), and  non-abelian bundle gerbes  (Section \ref{sec3_3}). In Section \ref{sec4_4} we establish a further relation between transport 2-functors and 2-vector bundles with connection. Conversely, all these structures are examples of transport 2-functors.

\subsection{Differential Non-Abelian Cohomology}

\label{sec4_1}

In this section we  set up a classifying theory for transport 2-functors $\mathcal{P}_2(M) \to T$ with structure Lie 2-groupoid $i: \mathcal{B}\mathfrak{G} \to T$, for $\mathfrak{G}$ a Lie 2-group coming from a smooth crossed module $(G,H,t,\alpha)$, $T$ an arbitrary 2-category, and $i$ an equivalence of categories. The latter condition together with Theorem \ref{th4} implies that we have equivalences  
\begin{equation}
\label{mainseq}
\transport{}{2}{\mathcal{B}\mathfrak{G}}{T}\cong \transsmoothpi{i}{2}{\mathrm{Gr}}{}_M \cong \transsmoothpi{\id_{\mathcal{B}\mathfrak{G}}}{2}{\mathrm{Gr}}{}_M  \end{equation}
of 2-categories. 

Our strategy is to translate the structure of the 2-category $\transsmoothpi{\id_{\mathcal{B}\mathfrak{G}}}{2}{\mathrm{Gr}}{}_M$ into  an equivalent 2-category of \quot{non-abelian differential cocycles} made up of smooth functions and differential forms, with respect to open covers of $M$.
Isomorphism classes of non-abelian differential cocycles form a set which we define as the \emph{non-abelian differential cohomology} of $M$; it classifies transport 2-functors on $M$ with $\mathcal{B}\mathfrak{G}$-structure up to isomorphism.

\subsubsection{Smooth Functors and Differential Forms}

For the translation of the 2-category $\transsmoothpi{\id_{\mathcal{B}\mathfrak{G}}}{2}{\mathrm{Gr}}{}_M$ into  smooth functions and differential forms we recall a general result about  the 2-category $\mathrm{Funct}^{\infty}(\mathcal{P}_2(X),\mathcal{B}\mathfrak{G})$ of smooth 2-functors, smooth \pt s, and smooth modifications defined on a smooth manifold $X$, with values in the Lie 2-groupoid $\mathcal{B}\mathfrak{G}$. Following \cite[Section 2.2]{schreiber5}, it corresponds to the following structure expressed in terms of smooth functions and differential forms:

\begin{enumerate}[(i)]
\item 
A smooth 2-functor $F:\mathcal{P}_2(X) \to \mathcal{B}\mathfrak{G}$ induces a pair of differential forms: a 1-form $A\in\Omega^1(X,\mathfrak{g})$ with values in the Lie algebra of $G$, and a 2-form $B\in\Omega^2(X,\mathfrak{h})$ with values in the Lie algebra of $H$, such that
\begin{equation}
\label{42}
\mathrm{d}A + [A
\wedge A] = t_{*} \circ B\text{.}
\end{equation}

\item
A smooth pseudonatural transformation $\rho:F \to F'$ gives rise to a 1-form $\varphi\in\Omega^1(X,\mathfrak{h})$ and a smooth map $g:X \to G$, such that
\begin{eqnarray}
A' + t_{*} \circ \varphi &=& \mathrm{Ad}_g(A) - g^{*}\bar\theta   
\label{57}\\\label{55}
B' + \alpha_{*}(A' \wedge \varphi) + \mathrm{d}\varphi + [\varphi \wedge \varphi]&=& (\alpha_g)_{*} \circ B  \text{.}
\end{eqnarray}
The identity $\id:F \to F$ has $\varphi=0$ and $g=1$. If $\rho_1$ and $\rho_2$ are composable \pt s, the 1-form of their composition $\rho_2\circ\rho_1$ is $(\alpha_{g_{2}})_{*} \circ\varphi_1 +  \varphi_2$, and the smooth map is $g_2g_1:X \to G$.

\item
A smooth modification $\mathcal{A}: \rho \Rightarrow \rho'$ gives rise to a smooth map $a:X \to H$, such that
\begin{equation}
\label{58}
g' = (t \circ a) \cdot g
\quad\text{ and }\quad
\varphi' +(r_{a}^{-1} \circ \alpha_{a})_{*}(A') =  \mathrm{Ad}_a (\varphi) -a^{*}\bar\theta\text{.}
\end{equation}
The identity modification $\id_{\rho}$ has $a=1$. If two modifications $\mathcal{A}_1$ and $\mathcal{A}_2$ are vertically composable, $\mathcal{A}_2 \bullet \mathcal{A}_1$ has the map $a_2a_1$. If two modifications $\mathcal{A}_1:\rho_1 \Rightarrow \rho_1'$ and $\mathcal{A}_2:\rho_2 \Rightarrow \rho_2'$ are horizontally composable, $\mathcal{A}_2 \circ \mathcal{A}_1$ has the map $a_2\alpha(g_2,a_1)$.

\end{enumerate}
The structure (i), (ii) and (iii)  forms a strict 2-category $\diffco{\mathfrak{G}}{2}{X}$, which we call the  \emph{2-category of $\mathfrak{G}$-connections on $X$} \cite{schreiber5}.

\begin{theorem}[{{\cite[Theorem 2.20]{schreiber5}}}]
\label{th1}
The correspondences described above furnish a strict 2-functor
\begin{equation*}
\alxydim{@C=1.5cm}{\mathrm{Funct}^{\infty}(\mathcal{P}_2(X),\mathcal{B}
\mathfrak{G}) \ar[r]^-{\mathcal{D}} &  \diffco{\mathfrak{G}}{2}{X}\text{,}}
\end{equation*}
which is an isomorphism of 2-categories.
\end{theorem}

We remark that  we have already used this isomorphism on the level of objects as a bijection \erf{bij}.

For a general Lie 2-group $\mathfrak{G}$ the correspondence between a smooth 2-functor $F: \mathcal{P}_2(X) \to \mathcal{B}\mathfrak{G}$ and the pair $(A,B)$ with $A\in \Omega^1(X,\mathfrak{g})$ and $B\in \Omega^2(X,\mathfrak{h})$ is established by an iterated integration and described in detail in \cite[Section 2.3.1]{schreiber5}. For the Lie 2-group $\mathfrak{G}=\mathcal{B}S^1$ it reduces to the following relation.

\begin{lemma}
\label{sfdfbsone}
Let $B \in \Omega^2(X)$ be a 2-form. Then, the smooth 2-functor $F: \mathcal{P}_2(X) \to \mathcal{BB}S^1$ that corresponds to $B$ under the isomorphism of Theorem \ref{th1} is given by
\begin{equation*}
F(\Sigma) = \exp \left (- \int_{[0,1]^2} \Sigma^{*}B \right )\text{,}
\end{equation*}
for all bigons $\Sigma \in BX$. 
\end{lemma}

\begin{proof}
We go through the construction in \cite[Section 2.3.1]{schreiber5} and reduce everything to the case $\mathfrak{G}=\mathcal{B}S^1$. The 1-form $\mathcal{A}_{\Sigma}$ in \cite[Equation 2.26]{schreiber5} is
\begin{equation*}
\mathcal{A}_{\Sigma} = -\int_{[0,1]} \Sigma^{*}B \in \Omega^1([0,1])\text{,}
\end{equation*}
with the integration performed over the second factor of $[0,1]^2$. It extends (due to the sitting instants of $\Sigma$) to a 1-form on $\R$, and so corresponds via \cite[Proposition 4.7]{schreiber3}  to a smooth functor $F_{\mathcal{A}_{\Sigma}}\maps\mathcal{P}_1(\R) \to \mathcal{B}S^1$, which is just a smooth map $f_{\Sigma}: \R \times \R \to S^1$. The correspondence between the 1-form $\mathcal{A}_{\Sigma}$ and the function $f_{\Sigma}$ is by \cite[Lemma 4.1]{schreiber3}:
\begin{equation*}
f_{\Sigma}(t_0,t_1) = \exp \left ( \int_{t_0}^{t_1} -\mathcal{A}_{\Sigma} \right )\text{.}
\end{equation*} 
According to the prescription, \cite[Equation 2.27]{schreiber5} and \cite[Proposition 2.17]{schreiber5} define
\begin{equation*}
F(\Sigma) = f_{\Sigma}(0,1)^{-1} = \exp \left (  \int_{0}^{1}  \mathcal{A}_{\Sigma} \right ) = \exp \left ( -\int_{[0,1]^2} \Sigma^{*}B \right )\text{.}\vspace{-1em}
\end{equation*}
\end{proof}

\subsubsection{Non-abelian Differential Cocycles}

We want to translate the structure of the  2-category $\transsmoothpi{\id_{\mathcal{B}\mathfrak{G}}}{2}{\mathrm{Gr}}{\pi}$ of smooth descent data with respect to a surjective submersion $\pi:Y \to M$ into smooth functions and differential forms, using Theorem \ref{th1} as a dictionary.

We recall that an object in $\transsmooth{\id_{\mathcal{B}\mathfrak{G}}}{2}{}$ is a collection $(\mathrm{triv},g,\psi,f)$ containing a smooth 2-functor $\mathrm{triv}\maps\mathcal{P}_2(Y) \to \mathcal{B}\mathfrak{G}$, a \pt\
$g: \pi_1^{*}\mathrm{triv} \to \pi_2^{*}\mathrm{triv}$
whose associated functor $\holo(g_{0}):\mathcal{P}_1(Y^{[2]}) \to \mathcal{B}\mathfrak{G}$ is a transport functor, and  modifications $\psi$ and $f$ whose associated natural transformations $\mathscr{F}(\psi)$ and $\mathscr{F}(f)$ are morphisms between transport functors. 

We begin with looking at a sub-2-category $\mathcal{U} \subset \transsmooth{\id_{\mathcal{B}\mathfrak{G}}}{2}{}$ in which all pseudonatural transformations and modifications are smooth, i.e. correspond to trivial transport functors and morphisms between trivial transport functors. An object $(\mathrm{triv},g_{\infty},\psi_{\infty},f_{\infty})$  in $\mathcal{U}$ corresponds under Theorem \ref{th1} to the following structure:
\begin{enumerate}
\item[(a)]
an object $(A,B) := \fo(\mathrm{triv})$  in $\diffco{\mathfrak{G}}{2}{Y}$, i.e. differential forms $A\in\Omega^1(Y,\mathfrak{g})$ and $B\in\Omega^2(Y,\mathfrak{h})$ satisfying relation \erf{42}.

\item[(b)]
a 1-morphism 
\begin{equation*}
(g,\varphi) := \fo(g_{\infty}):  \pi_1^{*}(A,B) \to \pi_2^{*}(A,B)
\end{equation*}
in $\diffco{\mathfrak{G}}{2}{Y^{[2]}}$, i.e. a smooth function $g:Y^{[2]} \to G$ and a 1-form $\varphi\in\Omega^1(Y^{[2]},\mathfrak{h})$ satisfying the relations \erf{57} and \erf{55}. 

\item[(c)]
a 2-morphism 
\begin{equation*}
f := \fo(f_{\infty}): \pi_{23}^{*}(g,\varphi) \circ \pi_{12}^{*}(g,\varphi)  \Rightarrow \pi_{13}^{*}(g,\varphi)
\end{equation*}
in $\diffco{\mathfrak{G}}{2}{Y^{[3]}}$ and a 2-morphism
\begin{equation*}
 \psi := \fo(\psi_{\infty}): \id_{(A,B)} \Rightarrow \Delta^{*}(g,\varphi)
\end{equation*} 
in $\diffco{\mathfrak{G}}{2}{Y}$; these are smooth functions $f:Y^{[3]} \to H$ and $\psi: Y \to H$ satisfying relations \erf{58}.
\end{enumerate}
The coherence conditions for $f_{\infty}$ and $\psi_{\infty}$ \cite[Definition 2.2.1]{schreiber6} imply the following equations of smooth functions, expressed as pasting diagrams: 
\begin{equation}
\label{eq:diffob1}
\alxydim{@C=1.3cm@R=1.9cm}{\pi_2^{*}(A,B) \ar[r]^{\pi_{23}^{*}(g,\varphi)} \ar@{<-}[d]_{\pi_{12}^{*}(g,\varphi)}
& \pi_{3}^{*}(A,B) \ar[d]^{\pi_{34}^{*}(g,\varphi)} \\ \pi_1^{*}(A,B)  \ar[ur]|{\pi_{13}^{*}(g,\varphi)}="2" \ar@{=>}[u];"2"|{\pi_{123}^{*}f} \ar[r]_{\pi_{14}^{*}(g,\varphi)}="1"
&  \ar@{<=}"1";"2"|{\pi_{134}^{*}f} \pi_{4}^{*}(A,B)}
=
\alxydim{@C=1.3cm@R=1.9cm}{\pi_2^{*}(A,B) \ar[dr]|{\pi_{24}^{*}(g,\varphi)}="2" \ar[r]^{\pi_{23}^{*}(g,\varphi)} \ar@{<-}[d]_{\pi_{12}^{*}(g,\varphi)}
& \pi_{3}^{*}(A,B) \ar[d]^{\pi_{34}^{*}(g,\varphi)}  \ar@{<=}"2";[]|{\pi_{234}^{*}f}
 \\  \pi_1^{*}(A,B)  \ar[r]_{\pi_{14}^{*}(g,\varphi)}="1"
\ar@{<=}"1";"2"|{\pi_{124}^{*}f}& \pi_{4}^{*}(A,B)}
\end{equation}
in the 2-category $\diffco{\mathfrak{G}}{2}{Y^{[4]}}$ and
\begin{equation}
\label{eq:diffob2}
\alxydim{@R=1.1cm@C=-0.2cm}{& \pi_2^{*}(A,B) \ar[dr]|{\Delta_{22}^{*}(g,\varphi)}="3" \ar@/^3pc/[dr]^{\id_{\pi_2^{*}(A,B)}}="2" \ar@{=>}"2";"3"|-{\pi_2^{*}\psi} & \\ \pi_1^{*}(A,B) \ar[ur]^{(g,\varphi)} \ar[rr]_{(g,\varphi)}="1" && \pi_2^{*}(A,B) \ar@{=>}[ul];"1"|>>>>>{\Delta_{122}^{*}f} }
\hspace{-0.8cm}= \id_{(g,\varphi)} = \hspace{-0.8cm}
\alxydim{@R=1.1cm@C=-0.2cm}{& \pi_1^{*}(A,B) \ar[dr]^{(g,\varphi)} & \\ \pi_1^{*}(A,B) \ar[ur]|{\Delta_{11}^{*}(g,\varphi)}="3"
\ar@/^3pc/[ur]^{\id_{\pi_1^{*}(A,B)}}="2" \ar@{=>}"2";"3"|-{\pi_1^{*}\psi} \ar[rr]_{(g,\varphi)}="1" && \pi_2^{*}(A,B) \ar@{=>}[ul];"1"|>>>>>{\Delta_{112}^{*}f} }
\end{equation}
in the 2-category $\diffco{\mathfrak{G}}{2}{Y^{[2]}}$.
We call the structure (a), (b), (c) subject to the two conditions
\erf{eq:diffob1} and \erf{eq:diffob1} a \emph{differential $\mathfrak{G}$-cocycle} for the surjective submersion $\pi:Y \to X$.

We proceed similarly with a  1-morphism $(h_{\infty},\varepsilon_{\infty})$ between  objects $(\mathrm{triv},g_{\infty},\psi_{\infty},f_{\infty})$ and $(\mathrm{triv}',g'_{\infty},\psi'_{\infty},f'_{\infty})$ in $\mathcal{U}$, and obtain:
\begin{enumerate}
\item [(d)]
a 1-morphism
$(h,\phi) := \mathcal{D}(h_{\infty}): (A,B) \to (A',B')$
in $\diffco{\mathfrak{G}}{2}{Y}$, i.e. a smooth function $h\maps Y \to G$ and a 1-form $\phi\in\Omega^1(Y,\mathfrak{h})$ satisfying relations \erf{57} and \erf{55},

\item[(e)]
 a 2-morphism 
$\varepsilon := \mathcal{D}(\varepsilon): \pi_2^{*}(h,\phi) \circ (g,\varphi) \Rightarrow (g',\varphi') \circ \pi_1^{*}(h,\phi)$
in $\diffco{\mathfrak{G}}{2}{Y^{[2]}}$, i.e. a smooth function $\varepsilon: Y^{[2]} \to H$ satisfying \erf{58}. 
\end{enumerate}
The coherence conditions of Definition \cite[Definition 2.2.2]{schreiber6} result in the identities 
\begin{equation}
\label{eq:diff1morph1}
\alxydim{@C=1.2cm@R=1.2cm}{(A,B) \ar@/^2.5pc/[r]^{\id}="1" \ar[r]|{\Delta^{*}(g,\varphi)}="2"
\ar@{=>}"1";"2"|{\psi} \ar[d]_{(h,\phi)} & (A,B) \ar[d]^{(h,\phi)}
\ar@{=>}[dl]|{\Delta^{*}\varepsilon} \\ (A',B') \ar[r]_{\Delta^{*}(g',\varphi')} & (A',B')}
=
\alxydim{@C=1.2cm@R=1.2cm}{(A,B) \ar[r]^{\id} \ar[d]_{(h,\phi)} & (A,B) \ar[d]^{(h,\phi)}
\ar@{=>}[dl]|{\id_{(h,\phi)}} \\ (A',B') \ar@/_2.5pc/[r]_{\Delta^{*}(g',\varphi')}="2"   \ar[r]|{\id}="1" & (A',B') \ar@{=>}"1";"2"|{\psi'}}
\end{equation}
in the 2-category  $\diffco{\mathfrak{G}}{2}{Y}$ and
\begin{equation}
\label{eq:diff1morph2}
\alxydim{@R=1.2cm@C=-0.3cm}{&\pi_2^{*}(A,B) \ar@/^1pc/[dr]^-{\pi_{23}^{*}(g,\varphi)}&\\ \pi_1^{*}(A,B) \ar@/^1pc/[ur]^-{\pi_{12}^{*}(g,\varphi)} \ar[rr]^{\pi_{13}^{*}(g,\varphi)}="1" \ar@{=>}[ur];"1"|{^{}f} \ar[d]_{\pi_1^{*}(h,\phi)}
&&\pi_3^{*}(A,B) \ar[d]^{\pi_3^{*}(h,\phi)} \ar@{=>}[dll]|{\pi_{13}^{*}\varepsilon}
\\ \pi_1^{*}(A',B') \ar[rr]_{\pi_{13}^{*}(g',\varphi')} && \pi_3^{*}(A',B')'}
=
\alxydim{@R=1.2cm@C=0.8cm}{\pi_1^{*}(A,B)\ar[d]_{\pi_1^{*}(h,\phi)} \ar[r]^{\pi_{12}^{*}(g,\varphi)}
& \pi_2^{*}(A,B) \ar[d]|{\pi_2^{*}(h,\phi)} \ar@{=>}[dl]|{\pi_{12}^{*}\varepsilon} \ar[r]^{\pi_{23}^{*}(g,\varphi)}
&\pi_3^{*}(A,B) \ar[d]^{\pi_3^{*}(h,\phi)} \ar@{=>}[dl]|{\pi_{23}^{*}\varepsilon} \\\pi_1^{*}(A',B')  \ar@/_3.3pc/[rr]_{\pi_{13}^{*}(g',\varphi')}="2"
\ar@{=>}[r];"2"|{f'} \ar[r]_{\pi_{12}^{*}(g',\varphi')}
& \pi_2^{*}(A',B')\ar[r]_{\pi_{23}^{*}(g',\varphi')} & \pi_3^{*}(A',B')}
\end{equation}
in the 2-category  $\diffco{\mathfrak{G}}{2}{Y^{[3]}}$. The structure (d), (e) subject to the conditions \erf{eq:diff1morph1} and \erf{eq:diff1morph2} is called a \emph{1-morphism between differential $\mathfrak{G}$-cocycles} for the surjective submersion $\pi$.

Finally, a  2-morphism induces   a 2-morphism $E: (h,\phi) \Rightarrow (h',\phi')$ in $\diffco{\mathfrak{G}}{2}{Y}$, i.e. a smooth function $E:Y \to H$ that satisfies \erf{58}, and the coherence condition of \cite[Definition 2.2.3]{schreiber6} infers the identity 
\begin{equation*}
\hspace{-0.4cm}
\alxydim{@C=0.7cm@R=1.2cm}{\pi_1^{*}(A,B) \ar[r]^{(g,\varphi)} \ar@/_2.5pc/[d]_{\pi_1^{*}(h',\phi')}="2"
\ar[d]^{\pi_1^{*}(h,\phi)}="1" \ar@{=>}"1";"2"|-{\pi_1^{*}E} &  \pi_2^{*}(A,B)
\ar@{=>}[dl]|{\varepsilon}
\ar[d]^{\pi_2^{*}(h,\phi)} \\ 
\pi_1^{*}(A',B')\ar[r]_{(g',\varphi')} &\pi_2^{*}(A',B') }
=
\alxydim{@C=0.7cm@R=1.2cm}{\pi_1^{*}(A,B) \ar[r]^{(g,\varphi)} 
\ar[d]_{\pi_1^{*}(h',\phi')}& \pi_2^{*}(A,B)
\ar@/^2.5pc/[d]^{\pi_2^{*}(h,\phi)}="2"
\ar@{=>}[dl]|{\varepsilon'}
\ar[d]_{\pi_2^{*}h_2}="1" \ar@{=>}"2";"1"|{\pi_2^{*}E}   \\ \pi_1^{*}(A',B')
\ar[r]_{(g',\varphi')} & \pi_2^{*}(A',B')}
\end{equation*}
in the 2-category $\diffco{\mathfrak{G}}{2}{Y^{[2]}}$. Such data is called a \emph{2-morphism between differential $\mathfrak{G}$-cocycles}.

We have now collected, in a systematical way, objects, 1-morphisms, and 2-morphisms of a 2-category, the  \emph{2-category of degree two differential $\mathfrak{G}$-cocycles}, which we denote by $\diffco{\mathfrak{G}}{2}{\pi}$. By construction we have a 2-functor
\begin{equation*}
\alxydim{@C=1.5cm}{ \transsmoothpi{\id_{\mathcal{B}\mathfrak{G}}}{2}{\mathrm{Gr}}{\pi}\supseteq \mathcal{U} \ar[r]^-{\mathcal{D}_{\pi}} & \diffco{\mathfrak{G}}{2}{\pi} \text{,}} 
\end{equation*}
and Theorem \ref{th1} implies immediately that it is an isomorphism of 2-categories.

Next we argue that for certain surjective submersions, the sub-2-category $\mathcal{U}$ is in fact equivalent to $\transsmoothpi{\id_{\mathcal{B}\mathfrak{G}}}{2}{\mathrm{Gr}}{\pi}$. These are submersions for which $Y$ and $Y^{[2]}$ have contractible connected components; we will call these \emph{two-contractible}. Any good open cover gives rise to a two-contractible surjective submersion; in particular, any surjective submersion can be refined by a two-contractible one.

\begin{proposition}
\label{prop1}
Let $\mathfrak{G}$ be a Lie 2-group and let $\pi:Y \to M$ be a two-contractible  surjective submersion. Then, the inverse of the 2-functor $\mathcal{D}_{\pi}$ induces an equivalence of categories:
\begin{equation*}
\diffco{\mathfrak{G}}{2}{\pi} \cong \transsmooth{\id_{\mathcal{B}\mathfrak{G}}}{2}{}\text{.} \end{equation*}
\end{proposition}

\begin{proof}
By \cite[Corollary 3.13]{schreiber3} transport functors over contractible manifolds are naturally equivalent to smooth functors. This applies to the transport functors $\mathscr{F}(g)$  contained in a descent object
$(\mathrm{triv},g,\psi,f)$, and to the transport functor $\mathscr{F}(h)$ contained in a descent 1-morphism $(h,\varepsilon)$. Such natural equivalences induce a descent 1-morphism $(\mathrm{triv},g,\psi,f) \cong (\mathrm{triv},g_{\infty},\psi',f')$ and a descent 2-morphism $(h,\varepsilon) \cong (h_{\infty},\varepsilon')$ with $g_{\infty}$ and $h_{\infty}$ smooth. Since the embedding of smooth functors into transport functors is full, all the modifications $\psi'$, $f'$, $\varepsilon'$ are then automatically smooth. This shows that $\transsmooth{\id_{\mathcal{B}\mathfrak{G}}}{2}{}$ is equivalent to its subcategory $\mathcal{U}$. \end{proof}

\subsubsection{Explicit Local Data}

\label{sec:localdata}

In order to make the structure of the 2-category $\diffco{\mathfrak{G}}{2}{\pi}$ of degree two differential $\mathfrak{G}$-cocycles more transparent, we shall spell out all details in case that the surjective submersion $\pi$ comes from an  open cover $\mathfrak{V}=\left \lbrace V_i \right \rbrace$ of $M$. In particular, we express the diagrammatic equations  above in terms of actual equations, using the relation between the Lie 2-groupoid $\mathcal{B}\mathfrak{G}$ and the crossed module $(G,H,t,\alpha)$ established in  Section \ref{sec:features}.

Firstly, a differential $\mathfrak{G}$-cocycle  $((A,B),(g,\varphi),\psi,f)$ has the following smooth functions and differential forms:
\begin{enumerate}
\item[(a)] 
On every open set $V_i$, 
\begin{equation*}
\psi_i: V_i \to H
\quad\text{, }\quad
A_i \in \Omega^1(V_i,\mathfrak{g})
\quad\text{ and }\quad
B_i \in \Omega^2(V_i,\mathfrak{h})\text{.}
\end{equation*}

\item[(b)]
On every two-fold intersection $V_i \cap V_j$, 
\begin{equation*}
g_{ij}:V_i \cap V_j \to G
\quad\text{ and }\quad
\varphi_{ij} \in \Omega^1(V_i \cap V_j,\mathfrak{h})\text{.}
\end{equation*}
\item[(c)]
On every three-fold intersection $V_i \cap V_j \cap V_k$,
\begin{equation*}
f_{ijk}:V_i \cap V_j \cap V_k \to H\text{.}
\end{equation*}
\end{enumerate}  
The cocycle conditions are the following:
\begin{enumerate}
\item 
Over every open set $V_i$,
\begin{eqnarray}
\label{31}
\mathrm{d}A_i + [A_i \wedge A_i] &=& t_{*}(B_i)
\\
\label{31a}
g_{ii}&=&t(\psi_i)
\\
\varphi_{ii}&=&-(r^{-1}_{\psi_i} \circ\alpha_{\psi_i})_{*}(A_i) - \psi_i^{*}\bar\theta\text{.}\nonumber
\end{eqnarray}
\item
Over every two-fold intersection $V_i \cap V_j$,
\begin{eqnarray}
\nonumber
A_j &=&\mathrm{Ad}_{g_{ij}}(A_{i}) - g_{ij}^*\bar\theta- t_{*}(\varphi_{ij})
\\
\nonumber
B_j  &=& (\alpha_{g_{ij}})_{*}(B_i)- \alpha_{*}(A_j \wedge \varphi_{ij}) - \mathrm{d}\varphi_{ij} - [\varphi_{ij}\wedge \varphi_{ij}]
\\
\label{31b}
1&=&f_{ijj}  \psi_j = f_{iij} \, \alpha_{g_{ij}}(\psi_i)\text{.}
\end{eqnarray}
\item
Over every three-fold intersection $V_i \cap V_j \cap V_k$,
\begin{eqnarray*}
g_{ik} &=& t(f_{ijk})  g_{jk}g_{ij}
\\
\mathrm{Ad}_{f_{ijk}}(\varphi_{ik})&=& (\alpha_{g_{jk}})_{*}(\varphi_{ij}) + \varphi_{jk} + (r_{f_{ijk}}^{-1} \circ \alpha_{f_{ijk}})_{*}(A_k) + f_{ijk}^{*}\bar\theta\text{.}
\end{eqnarray*}

\item
Over every four-fold intersection $V_i \cap V_j \cap V_k \cap V_l$,
\begin{equation}
\label{31c}
f_{ikl} \alpha(g_{kl},f_{ijk})=f_{ijl}f_{jkl}\text{.}
\end{equation}
\end{enumerate} 
Additionally, the curvature of the differential cocycle is according to Definition \ref{def:curv} given by
\begin{equation*}
H_i := \mathrm{d}B_i + \alpha_{*}(A_i \wedge B_i) \in \Omega^3(V_i,\mathfrak{h})\text{.}
\end{equation*}

Secondly, a 1-morphism $((h,\varepsilon),\phi)$ between differential cocycles $((A,B),(g,\varphi),\psi,f)$ and $((A',B'),(g',\varphi'),\psi',f')$ has the following structure:
\begin{enumerate}
\item[(a)]
On every open set $V_i$,
\begin{equation*}
h_{i}:V_{i} \to G
\quad\text{ and }\quad
\phi_{i}\in\Omega^1(V,\mathfrak{h})\text{.}
\end{equation*}
\item[(b)]
On every two-fold intersection $V_i \cap V_j$,
\begin{equation*}
\epsilon_{ij}:V_i \cap V_j \to H\text{.}
\end{equation*}
\end{enumerate}
The following conditions have to be satisfied:
\begin{enumerate}
\item 
Over every open set $V_i$,
\begin{eqnarray}
\label{32}
B_i' &=& (\alpha_{h_i})_{*}(B_i) - \alpha_{*}(A_i' \wedge \phi_i) - \mathrm{d}\phi_i - [\phi_i \wedge \phi_i]
\\
\label{32a}
A'_i &=& \mathrm{Ad}_{h_i}(A_i) - t_{*}(\phi_i) - h_i^{*}\bar\theta
\\
\label{32b}
\psi_i'&=& \epsilon_{ii}\alpha(h_i,\psi_i)\text{.}
\end{eqnarray}

\item
Over every two-fold intersection $V_i \cap V_j$,
\begin{eqnarray}
\label{32c}
g_{ij}' &=& t(\epsilon_{ij})h_jg_{ij}h_i^{-1}
\\
\label{32d}
\varphi_{ij}' &=& \mathrm{Ad}_{\epsilon_{ij}}((\alpha_{h_j})_{*}(\varphi_{ij}) + \phi_j) 
- (\alpha_{g_{ij}'})_{*}(\phi_i) - (r_{\varepsilon_{ij}}^{-1} \circ \alpha_{\epsilon_{ij}})_{*}(A'_{j}) - \epsilon_{ij}^{*}\bar\theta
\end{eqnarray}
\item
Over every three-fold
 intersection $V_i \cap V_j \cap V_k$,
\begin{equation}
\label{32f}
f_{ijk}' = \epsilon_{ik}\alpha(h_k,f_{ijk})\alpha(g'_{ik},\epsilon_{ij}^{-1})\epsilon_{jk}^{-1}\text{.}
\end{equation}
\end{enumerate}
Finally, a 2-morphism  $E$ between differential cocycles has, for any open set $V_i$, a smooth function $E_i: V_{i} \to H$  such that on every open set $V_i$
\begin{eqnarray*}
h'_i &=& t(E_i) h_i
\\
\phi'_i &=& \mathrm{Ad}_{E_i} (\phi_i)  - (r^{-1}_{E_i} \circ \alpha_{E_i})_{*}(A'_i)-E_i^{*}\bar\theta
\end{eqnarray*}
and, on every 2-fold intersection $V_i \cap V_j$,
\begin{equation*}
\epsilon_{ij}' = \alpha(g_{ij}',E_i) \epsilon_{ij}E_j^{-1}\text{.}
\end{equation*}

Concluding this discussion we note the following normalization lemma, which we need for the discussion of surface holonomy.
It   can be seen as a 
variant of the maybe familiar normalization result for \v Cech
cocycles.

\begin{lemma}
\label{equiavelncenormalized}
Every differential $\mathfrak{G}$-cocycle with respect to an open cover $\left \lbrace V_i \right \rbrace_{i\in I}$ is 1-isomorphic to a differential $\mathfrak{G}$-cocycle  with 
\begin{equation*} 
\psi_i=1
\quomma
g_{ii}=1
\quomma
f_{iij}=f_{ijj}=1
\quand
f_{iji}=y_{ij}^{-1}\alpha(x_{ij},y_{ij})
\end{equation*}  
for all $i,j\in I$, and  elements $x_{ij}\in G$ and $y_{ij}\in H$. In particular, if $\mathfrak{G}=\mathcal{B}A$ for an abelian Lie group $A$, then $f_{iji}=1$. 
\end{lemma}

\begin{proof}
Suppose $((A,B),(g,\varphi),\psi,f)$ is a differential $\mathfrak{G}$-cocycle with respect to $\left \lbrace V_i \right \rbrace$. It is easy to see that one can always pass to a1-isomorphic cocycle with $\psi_i=1$, $g_{ii}=1$, and $f_{iij}=f_{ijj}=1$. 
\begin{comment}
We look at  the triple $(\phi_i,h_i,\varepsilon_{ij})$ with the differential forms $\phi_i := 0 \in \Omega^1(V_i)$, smooth maps $h_{i} := 1$ on $V_i$, and smooth maps
\begin{equation*}
\varepsilon_{ij} := \begin{cases}\psi_i^{-1} & i=j \\
1 & i \neq j \\
\end{cases}
\end{equation*}
defined on $V_i \cap V_j$. Then we \emph{define} a new differential $\mathfrak{G}$-cocycle $((A',B'),(g',\varphi'),\psi',f')$ via equations \erf{32} to \erf{32f}. We have $\psi'_i = \epsilon_{ii}\alpha(h_i,\psi_i)=\varepsilon_{ii}\psi_i^{-1}=1$. Then,  $g'_{ii}=1$  follows from \erf{31a}, and $f_{iij}'=f_{ijj}'=1$ follows from \erf{31b}. \end{comment}
In order to achieve the remaining condition for $f_{iji}$, we choose a total order on the index set $I$ of the open cover, and look at  smooth maps
\begin{equation*}
\varepsilon_{ij} := \begin{cases}1 & i\leq j \\
f_{iji} & i>j \\
\end{cases}
\end{equation*}
defined on $V_i \cap V_j$. Then we define yet another differential $\mathfrak{G}$-cocycle $((A',B'),(g',\varphi'),\psi',f')$ via  equations \erf{32} to \erf{32f}, and it is straightforward to see that this one has the claimed form.
\begin{comment}
We still have $\psi_i''=1$ and hence still $g_{ii}''=f_{iij}''=f_{ijj}=1$. From \erf{31c} we conclude $\alpha(g_{ij},f_{iji})=f_{jij}$.
Further, we get from \erf{32f}
\begin{equation*}
f_{iji}'' = f'_{iji}\epsilon_{ij}^{-1}\epsilon_{ji}^{-1}\text{.}
\end{equation*}
For $i>j$, this gives $f''_{iji}= f'_{iji}\epsilon_{ij}^{-1}\epsilon_{ji}^{-1} = f_{iji}'f'_{iji}^{-1}=1$, so that $f_{iji}=y_{ij}^{-1}\alpha(x_{ij},y_{ij})$ for $x_{ij}=y_{ij}=1$. For $i < j$, this gives $f_{iji}'' = f'_{iji}f'_{jij}^{-1}=f_{iji}'\alpha(g_{ij},f_{iji}^{-1})$, so that $f_{iji}=y_{ij}^{-1}\alpha(x_{ij},y_{ij})$ for $x_{ij}=g_{ij}$ and $y_{ij}=f_{iji}^{-1}$.
\end{comment}
\end{proof}

\subsubsection{Non-abelian Differential Cohomology}

The 2-categories $\diffco{\mathfrak{G}}{2}{\pi}$ of degree two differential $\mathfrak{G}$-cocycles form a direct system for surjective submersions over $M$, and so do the sets $\hc 0 (\diffco{\mathfrak{G}}{2}{\pi})$ of isomorphism classes of objects.  

\begin{definition}
\label{def:nonco}
The direct limit
\begin{equation*}
\hat H^2(M,\mathfrak{G}) := \lim_{\overrightarrow{\pi}} \hc 0(\diffco{\mathfrak{G}}{2}{\pi})
\end{equation*} 
is called the degree two \emph{differential non-abelian cohomology} of $M$ with values in $\mathfrak{G}$.
\end{definition}

The terminology \quot{non-abelian differential cohomology} is motivated by the following observation. If one drops all differential forms from the local data described in Section \ref{sec:localdata}, and only keeps the smooth functions, the corresponding limit coincides with  the \emph{non-abelian cohomology} $H^2(M,\mathfrak{G})$, as it appears for instance in \cite{giraud,breen2,bartels,wockel1}. Thus, Definition \ref{def:nonco} extends ordinary non-abelian cohomology by differential form data. 

We remark that by Lemma \ref{equiavelncenormalized} and the fact that every surjective submersion can be refined by an open cover, one can define non-abelian differential cohomology by only  using cocycles that are normalized in the sense of Lemma \ref{equiavelncenormalized}. 

Since every surjective submersion can be refined to a two-contractible one (for example, by a good open cover), it also suffices to take the limit over the two-contractible surjective submersions. Then, Proposition \ref{prop1} and Theorem \ref{th4} imply via \erf{mainseq} the following classification result.

\begin{theorem}
\label{th:class}
Let $\mathfrak{G}$ be a Lie 2-group, and let $i: \mathcal{B}\mathfrak{G} \to T$ be an equivalence of 2-categories. Then, there is a bijection
\begin{equation*}
\hc 0 \transport{}{2}{\mathrm{Gr}}{T} \cong \hat H^2(M,\mathfrak{G})\text{.}
\end{equation*}
In other words, transport 2-functors on $M$ with $\mathcal{B}\mathfrak{G}$-structure are classified up to isomorphism by the non-abelian differential cohomology of $M$ with values in $\mathfrak{G}$.
\end{theorem}

Let us specify two particular examples of differential non-abelian cohomology which have appeared before in the literature:
\begin{enumerate}
\item 
The Lie 2-group $\mathfrak{G} = \mathcal{B}S^1$. We leave it as an easy exercise to the reader to check  that our differential non-abelian cohomology is precisely degree two Deligne cohomology,
\begin{equation*}
\hat H^2(M,\mathcal{B}S^1) = H^2(M,\mathcal{D}(2))\text{.}
\end{equation*}
Deligne cohomology \cite{brylinski1} is a well-known local description of abelian gerbes with connection. 

\item
The Lie 2-group $\mathfrak{G}=\mathrm{AUT}(H)$ for $H$ some ordinary Lie group $H$. We also leave it to the reader to check that our differential cocycles are precisely  the local description of connections in  non-abelian gerbes given by Breen and Messing \cite{breen1} (see Remark \ref{rem2} below). Thus, we have an equality
\begin{equation*}
\hat H^2(M,\mathrm{AUT}(H)) = \bigset{5.4cm}{Equivalence classes of local data\\of Breen-Messing $H$-gerbes\\with connection over $M$}  \text{.}
\end{equation*}
\end{enumerate}
Summarizing, the  theory of transport 2-functors covers both Deligne cocycles and Breen-Messing cocycles.

\begin{remark}
\label{rem2}
We remark that condition \erf{31} between the 1-curving $A$ and the 2-curving $B$  of a differential $\mathfrak{G}$-cocycle is present neither  in Breen-Messing gerbes \cite{breen1} nor in the non-abelian bundle gerbes of  \cite{aschieri}, which we discuss in Section \ref{sec3_3}. Breen and Messing call the local 2-form
\begin{equation*}
t_{*}(B_i) - \mathrm{d}A_i - [A_i \wedge A_i]
\end{equation*}
which is here zero by \erf{31}, the \emph{fake curvature} of the gerbe.
In this terminology, transport 2-functors only cover Breen-Messing gerbes with vanishing fake curvature. 

The crucial point is here that unlike transport 2-functors, neither  Breen-Messing gerbes nor   non-abelian bundle gerbes have concepts of holonomy or parallel transport. And indeed, equation \erf{31} comes from an important consistency condition on this parallel transport, namely from the target-source matching condition for the transport 2-functor, which makes it possible to decompose parallel transport into pieces.  So we understand equation \erf{31} as an \emph{integrability condition} which is necessarily for a consistent parallel transport. In other words, vanishing fake curvature is not specifying a subclass of connections, it is a \emph{necessary condition} every connection has to satisfy.

\end{remark}

\subsection{Abelian Bundle Gerbes
with Connection}

\label{sec3_2}

In this section we consider the target 2-category $T=\mathcal{B}(S^1\text{-}\mathrm{Tor})$,  the monoidal category of $S^1$-torsors viewed as a 2-category with a single object \cite[Example A.2]{schreiber6}. Associated to this 2-category is the 2-functor 
\begin{equation*}
i_{S^1}: \mathcal{BB}S^1 \to \mathcal{B}(S^1\text{-}\mathrm{Tor})
\end{equation*}
that sends the single 1-morphism of $\mathcal{BB}S^1$ to the circle, viewed as an $S^1$-torsor over itself. The main result of this section is:
\begin{theorem}
\label{th:bgrb}
There is an equivalence
\begin{equation*}
\mathrm{Trans}_{\mathcal{BB}S^1}(M,\mathcal{B}(S^1\text{-}\mathrm{Tor})) \cong \mathfrak{BGrb}^{\nabla}(M) 
\end{equation*}
between the 2-category of transport 2-functors on $M$ with values in $\mathcal{B}(S^1\text{-}\mathrm{Tor})$ and with $\mathcal{BB}S^1$-structure, and the 2-category of  bundle gerbes with connection over $M$.
\end{theorem}

Let us first recall the definition of  bundle gerbes following \cite{murray,murray2,stevenson1,waldorf1}. For a surjective submersion $\pi:Y \to M$, we first define the following 2-category $\mathfrak{BGrb}^{\nabla}(\pi)$:

\begin{enumerate}
\item
An object is a tuple $(B,L,\omega,\mu)$ consisting of  a 2-form $B\in \Omega^{2}(Y)$, a principal $S^{1}$-bundle $L$ with connection $\omega$ over $Y^{[2]}$ of curvature $\mathrm{curv}(\omega) = \pi_1^{*}B - \pi_2^{*}B$, and an associative, connection-preserving bundle isomorphism 
\begin{equation*}
\mu: \pi_{23}^{*}L \otimes \pi_{12}^{*}L \to \pi_{13}^{*}L
\end{equation*}
over $Y^{[3]}$. 

\item
A 1-morphism $(B,L,\omega,\mu) \to (B',L',\omega',\mu')$, also known as   \emph{stable isomorphism}, is a principal $S^{1}$-bundle $A$ with connection $\varsigma$ over $Y$ of curvature $\mathrm{curv}(\varsigma)=B - B'$ together with a connection-preserving bundle isomorphism 
\begin{equation*}
\alpha:  \pi_2^{*}A \otimes L \to L' \otimes \pi_1^{*}A
\end{equation*}
over $Y^{[2]}$, such that the diagram
\begin{equation}
\label{43}
\alxydim{@C=2cm}{\pi_3^{*}A \otimes \pi_{23}^{*}L \otimes \pi_{12}^{*}L \ar[r]^-{\id\otimes\mu} \ar[d]_{\pi_{23}^{*}\alpha \otimes \id} & \pi_3^{*}A \otimes  \pi_{13}^{*}L  \ar[dd]^{\pi_{13}^{*}\alpha} \\ \pi_{23}^{*}L' \otimes \pi_{2}^{*}A \otimes L \ar[d]_{\id\otimes \pi_{12}^{*}\alpha} & \\ \pi_{23}^{*}L' \otimes \pi_{12}^{*}L'\otimes\pi_1^{*}A   \ar[r]_-{\mu'\otimes\id} & \pi_{13}^{*}L' \otimes \pi_{1}^{*}A}
\end{equation}
of bundle isomorphisms over $Y^{[3]}$ is commutative. 

\item
A 2-morphism $(A,\varsigma,\alpha) \Rightarrow (A',\varsigma',\alpha')$ is a connection-preserving bundle isomorphism $\varphi: A \to A'$ over $Y$, such that the diagram
\begin{equation}
\label{45}
\alxydim{@R=1.3cm@C=1.3cm}{\pi_2^{*}A \otimes L \ar[d]_{\pi_2^{*}\varphi \otimes \id_L} \ar[r]^{\alpha} & L' \otimes \pi_1^{*}A \ar[d]^{\id_{L'} \otimes \pi_1^{*}\varphi} \\ \pi_2^{*}A' \otimes L \ar[r]_{\alpha'} & L' \otimes \pi_1^{*}A'}
\end{equation}
of bundle isomorphisms over $Y^{[2]}$ is commutative.
\end{enumerate}
The full 2-category of bundle gerbes with connection  over $M$ is then defined as the direct limit
\begin{equation*}
\mathfrak{BGrb}^{\nabla}(M) := \lim_{\overrightarrow{\pi}} \mathfrak{BGrb}^{\nabla}(\pi)\text{.}
\end{equation*}

The strategy to prove Theorem \ref{th:bgrb} is to first establish the equivalence on a local level: 

\begin{proposition}
\label{prop:bgrb}
For every surjective submersion $\pi:Y \to M$ there is a  surjective equivalence of 2-categories:
\begin{equation*}
\transsmooth{i_{S^1}}{2}{\Sigma\Sigmna U(1)} \cong \mathfrak{BGrb}^{\nabla}(\pi)\text{.}
\end{equation*}
\end{proposition}

A 2-functor $\transsmooth{i_{S^1}}{2}{\Sigma\Sigmna U(1)} \to \mathfrak{BGrb}^{\nabla}(\pi)$ realizing the claimed equivalence is defined in the following way. For a smooth descent object $(\mathrm{triv},g,\psi,f)$,  the smooth 2-functor $\mathrm{triv}:\mathcal{P}_2(Y) \to \mathcal{B}\mathcal{B}S^1$ defines by Theorem \ref{th1}  a 2-form $B\in\Omega^2(Y)$, this is the first ingredient of the bundle gerbe. The \pt\ $g$ yields a transport functor 
\begin{equation*}
\holo(g):\mathcal{P}_1(Y^{[2]})
\to \Lambda_{i_{S^1}}\mathcal{B}(S^1\text{-}\mathrm{Tor})
\end{equation*}
with $\Lambda\mathcal{B\mathcal{B}}S^1$-structure. 
Despite of the heavy notation, the following lemma  translates this functor into familiar language. 
\begin{lemma}
\label{lem6}
For $X$ a smooth manifold, there is a canonical surjective equivalence of monoidal categories
\begin{equation*}
\sbunX \cong \transportX{2}{1}{_{\Lambda \mathcal{BB}S^1}}{\Lambda_{i_{S^1}}\mathcal{B}(S^1\text{-}\mathrm{Tor})}
\end{equation*}
between principal $S^1$-bundles with connection and transport functors with $\Lambda\mathcal{BB}S^1$-structure.
\end{lemma}

\begin{proof}
First of all, we have evidently
$\Lambda\mathcal{B}\mathcal{B}S^1 = \mathcal{B}S^1$. Second, there is a canonical equivalence of categories
\begin{equation}
\label{44}
\Lambda_{i_{S^1}}\mathcal{B}(S^1\text{-}\mathrm{Tor}) \cong S^1\text{-}\mathrm{Tor}\text{.}
\end{equation}
This comes from the fact that an object is in both categories an $S^1$-torsor. A morphism between $S^1$-torsors $V$ and $W$ in $\Lambda_{i_{S^1}}\mathcal{B}(S^1\text{-}\mathrm{Tor})$ is by definition a 2-morphism  
\begin{equation*}
\alxydim{@=1.2cm}{\ast \ar[r]^{S^1} \ar[d]_{V} & \ast \ar[d]^{W} \ar@{=>}[dl]|*+{f} \\ \ast \ar[r]_{S^1} & \ast }
\end{equation*}
in $\mathcal{B}(S^1\text{-}\mathrm{Tor})$, and this is in turn an $S^1$-equivariant map
\begin{equation*}
f: W \otimes S^1 \to S^1 \otimes V\text{.}
\end{equation*}
It can be identified canonically with an $S^1$-equivariant map $f^{-1}: V \to W$, i.e. a morphism in $S^1\text{-}\mathrm{Tor}$. It is straightforward to see that \erf{44} is even a monoidal equivalence. Then, the claim is proved by Theorem \ref{th3}.
\end{proof}

 Now, via Lemma \ref{lem6}  the transport functor $\holo(g)$ is  a principal $S^1$-bundle $L$ with connection $\omega$ over $Y^{[2]}$. This circle bundle will be the second ingredient of the bundle gerbe.
\begin{lemma}
\label{lem8}
The curvature of the connection $\nabla$ on the circle bundle $L$ satisfies \begin{equation*}
\mathrm{curv}(\omega)=\pi_1^{*}B - \pi_2^{*}B\text{.}
\end{equation*}
\end{lemma}

\begin{proof}
Let $U_{\alpha}$ be open sets covering $Y^{[2]}$, and let $(\widetilde{\mathrm{triv}},\tilde t)$ be a local $i_{S^1}$-trivialization of the transport functor $\holo(g)$ consisting of smooth functors $\widetilde{\mathrm{triv}}_{\alpha}: \mathcal{P}_1(U_{\alpha}) \to \mathcal{B}S^1$ and natural transformations
\begin{equation*}
\tilde t_{\alpha}:  \holo(g)|_{U_{\alpha}} \to (\widetilde{\mathrm{triv}}_{\alpha})_{i_{S^1}}\text{.}
\end{equation*}
We observe that the functors $\widetilde{\mathrm{triv}}_{\alpha}$ and the natural transformation $\tilde t_{\alpha}$ lie in the image of the functor $\holo$, such that there exist smooth \pt s $\rho_{\alpha}: \pi_1^{*}\mathrm{triv}|_{U_{\alpha}} \to \pi_2^{*}\mathrm{triv}|_{U_{\alpha}}$ and modifications $t_{\alpha}: g|_{U_{\alpha}} \Rightarrow \rho_{\alpha}$ with 
\begin{equation*}
\widetilde{\mathrm{triv}}_{\alpha} = \holo(\rho_{\alpha})
\quad\text{ and }\quad
\tilde t_{\alpha} = \holo(t_{\alpha})\text{.}
\end{equation*}
As found in \cite{schreiber5} and reviewed in Section \ref{sec3_1} of the present article, associated to the smooth \pt\ $\rho_{\alpha}$ is a 1-form $\varphi_{\alpha}\in \Omega^1(U_{\alpha})$,
and equation \erf{55} infers in the present situation
\begin{equation*}
\pi_{1}^{*}B - \pi_2^{*}B =  \mathrm{d}\varphi_{\alpha}\text{.}
\end{equation*}
It remains to trace back the relation between $\varphi_{\alpha}$ and the curvature of the connection $\omega$ on circle bundle $L$. Namely, if $A_{\alpha}$ is the 1-form corresponding to the smooth functor $\widetilde{\mathrm{triv}}_{\alpha}$, we have
\begin{equation*}
A_{\alpha} = \varphi_{\alpha}
\quad\text{ and }\quad
\mathrm{d}A_{\alpha} = \mathrm{curv}(\omega)\text{.}
\end{equation*}
This shows the claim. 
\end{proof}

In order to complete the construction of the bundle gerbe, we note that the modification $f\maps \pi_{23}^{*}g \circ \pi_{12}^{*}g \Rightarrow \pi_{13}^{*}g$ induces an isomorphism 
\begin{equation*}
\holo(f): \pi_{23}^{*}\holo(g) \otimes \pi_{12}^{*}\holo(g) \to \pi_{13}^{*}\holo(g)
\end{equation*}
of transport functors; again by Lemma \ref{lem6} this defines an isomorphism
\begin{equation*}
\mu : \pi_{12}^{*}L \otimes \pi_{23}^{*}L \to \pi_{13}^{*}L 
\end{equation*}
of circle bundles with connection, which is the last ingredient of the bundle gerbe. The coherence axiom of \cite[Definition 2.2.1]{schreiber6}
implies the associativity condition on $\mu$. This shows that $(B,L,\nabla,\mu)$ is a bundle gerbe with connection. The descent datum $\psi$ has been forgotten by this construction.

Using Lemma \ref{lem6} in the same way as just demonstrated it is easy to assign bundle gerbe 1-morphisms to descent 1-morphisms and bundle gerbe 2-morphisms to descent 2-morphisms. Here the  coherence conditions of \cite[Definitions 2.2.2 and 2.2.3]{schreiber6}  translate  into the commutative diagrams \erf{43} and \erf{45}. The composition  for 1-morphisms and 2-morphisms  of bundle gerbes (which we have not carried out above) is precisely reproduced by the composition laws of the descent 2-category $\transsmooth{i_{S^1}}{2}{}$. This defines the 2-functor of Proposition \ref{prop:bgrb}. 

It is  evident that this 2-functor is an equivalence of 2-categories, since all manipulations we have made are equivalences according to Lemma \ref{lem6} and Theorem \ref{th1}. It remains to remark that the descent datum $\psi$ can be reproduced in a canonical way from a given bundle gerbe using the existence of dual circle bundles, see Lemma 1 in \cite{waldorf1}.
This finishes the proof of Proposition \ref{prop:bgrb}.

It is clear that the equivalence of Proposition \ref{prop:bgrb} is compatible with the refinement of surjective submersions, so that it given an equivalence in the direct limits:
\begin{equation*}
\mathfrak{Des}^2(i_{S^1})^{\infty}_M
 \cong \mathfrak{BGrb}^{\nabla}(M)\text{.}
\end{equation*}
With Theorem \ref{th4} this proves Theorem \ref{th:bgrb}. 

\subsection{Non-Abelian Bundle Gerbes with Connection}

\label{sec3_3}

In this section we generalize the equivalence of the previous section to so-called non-abelian bundle gerbes, as defined in \cite{aschieri}. The difference between abelian and non-abelian bundle gerbes can roughly be summarized as follows:
\begin{enumerate}

\item 
The Lie group $S^1$ is replaced by some connected Lie group $H$.

\item
Principal $S^1$-bundles are replaced by \emph{principal $H$-bibundles}, bundles with commuting left and right principal $H$-action. Morphisms between $H$-bibundles are bi-equivariant bundle morphisms.

\item
Connections on $S^1$-bundles are replaced by so-called \emph{twisted connections} on bibundles. 
\end{enumerate}
Non-abelian $H$-bundle gerbes with connection over $M$ form a 2-category which we denote by $H\text{-}\mathfrak{BGrb}^{\nabla}(M)$; some details are reviewed below. There is a sub-2-category $H\text{-}\mathfrak{BGrb}^{\nabla_{\!f\!f}}(M)$ of fake-flat non-abelian $H$-bundle gerbes, related to Remark \ref{rem2}.

On the side of transport 2-functors, we write $\hbitor$ for the category
  whose objects are smooth manifolds with commuting, free and transitive, smooth left and right
  $H$-actions, and whose morphisms are
  smooth bi-equivariant maps. Using the product over
  $H$ this is naturally a (non-strict) monoidal category, and we   write
  $\mathcal{B}(\hbitor)$ for the corresponding one-object (non-strict)
  2-category. Let $\mathrm{AUT}(H)$ be the Lie 2-group associated to $H$ via Example \ref{ex:2groups} \erf{ex:2groups:aut}. A 2-functor
\begin{equation}
\label{48}
i:\mathcal{B}\mathrm{AUT}(H) \to \mathcal{B}(\hbitor)
\end{equation}
is  defined as follows.
It sends a 1-morphism $\varphi \in \mathrm{Aut}(H)$ to the $H$-bitorsor $_\varphi H$
which is the group $H$ on which an element $h$ acts from the right by multiplication and from the left by multiplication with $\varphi(h)$. The compositors of $i$ are given by the canonical identifications
\begin{equation*}
c_{g_1,g_2}:{}_{g_1} H \;\times_H \;{}_{g_2} H
  \to
  {}_{g_2 g_1} H\text{,}
\end{equation*}
and the unitor is the identity. The 2-functor $i$
further sends a 2-morphism $h:\varphi_1 \Rightarrow \varphi_2$ to the bi-equivariant map
\begin{equation*}
_{\varphi_1}H \to _{\varphi_2}H:x \mapsto hx\text{.}
\end{equation*}
While the bi-equivariance with respect to the right action is obvious, the one with respect to the left action follows from the condition $\varphi_2(x)=h\varphi_1(x)h^{-1}$ we have for the 2-morphisms in $\mathrm{AUT}(H)$ for all $x\in H$.

The main result of this section is:

\begin{theorem}
\label{th:nagrb}
There is a canonical injective map
\begin{equation*}
\hc 0 (H\text{-}\mathfrak{BGrb}^{\nabla_{\!f\!f}}(M)) \to \hc 0 (\transport{}{2}{\mathcal{B}\mathrm{AUT}(H)}{\mathcal{B}(\hbitor)})
\end{equation*}
from isomorphism classes of fake-flat non-abelian $H$-bundle gerbes to transport 2-functors with values in $\mathcal{B}(H\text{-}\mathrm{BiTor})$ and $\mathcal{B}\AUT(H)$-structure. 
\end{theorem}

That we do not achieve a bijection or an  equivalence of 2-categories is due to the fact that the category of bibundles with twisted connections -- as defined in \cite{aschieri} -- oversees a class of morphisms, which in principle should be allowed
and arise naturally from the theory of transport functors.

Anyway, Theorem \ref{th:nagrb} provides non-trivial examples of transport 2-functors. The injective map of Theorem \ref{th:nagrb} is constructed as 
\begin{equation*}
H\text{-}\mathfrak{BGrb}^{\nabla_{\!f\!f}}(M) \to \transsmoothpi{i}{2}{\mathrm{Gr}}{}_M\cong \transport{}{2}{\mathcal{B}\mathrm{AUT}(H)}{\mathcal{B}(\hbitor)}\text{,}
\end{equation*}
where the first 2-functor is constructed in Proposition \ref{prop:locequivnagrb} and the equivalence is the one of Theorem \ref{th4}. 
We first review necessary material.

\subsubsection{Twisted Connections on Bibundles}

\label{bitorsors}

Let $P$ be a principal $H$-bibundle over $X$.
We  denote the left and right actions by an element $h\in H$ on  $P$ by $l_h$ and $r_h$, respectively. Measuring the difference between the left and the right action in the sense of  $l_h(p) = r_{g(h)}(p)$ furnishes a smooth map
\begin{equation}
\label{47}
g: P \to \mathrm{Aut}(H)\text{.}
\end{equation}
In the following we denote by $\mathfrak{aut}(H)$ the Lie algebra of $\mathrm{Aut}(H)$. Further, we denote by $t: H \to \mathrm{Aut}(H)$ the assignment of inner automorphisms and by $\alpha: \mathrm{Aut}(H) \times H \to H$ the evaluation. 

\begin{definition}[\cite{aschieri}]
Let $p : P \to X$ be a principal $H$-bibundle, and let 
  $A \in \Omega^1(X,\mathfrak{aut}(H))$ be a 1-form on the base space. An \emph{$A$-twisted connection on $P$} is
  a  1-form $\phi \in \Omega^1(P,\mathfrak{h})$ satisfying
\begin{equation}
\label{53}
    \phi_{\rho h}\left(
      \frac{\mathrm{d}}{\mathrm{d}t}(\rho h)
    \right)
    =
    \mathrm{Ad}_h^{-1}
    \left(
      \phi_\rho\left(\frac{\mathrm{d} \rho}{\mathrm{d} t}\right)
    \right)
    -
    (r_h \circ \alpha_h)_{*}\circ (p^*A)
    +
    \theta_h\left(
      \frac{\mathrm{d} h}{\mathrm{d}t}
    \right)
\end{equation}
for all smooth curves $\rho : [0,1] \to P$ and $h : [0,1] \to H$. A morphism $f: P \to P'$ respects $A$-twisted connections $\phi$ on $P$ and $\phi'$ on $P'$ if $f^{*}\phi'=\phi$.  
\end{definition}

We write $\mathfrak{Bibun}_H^{\nabla}(X,A)$ for the category of principal $H$-bibundles with $A$-twisted connection over $X$, and $\mathfrak{Bibun}_{H}^{\nabla}(X)$ for the union of these categories over all 1-forms $A$.

\begin{remark}
For $A = 0$ an $A$-twisted connection on $P$ is the same as an ordinary
connection on $P$ regarded  as a right principal
bundle. 
\begin{comment}
One can give an analogous definition of a twisted left connection. Then, a twisted right connection
gives rise to a twisted left connection, for a different
twist, and vice versa. This is discussed in detail in
\cite{aschieri}, but will be a manifest consequence of
the reformulation which we give later.
\end{comment}
\end{remark}

\begin{lemma}[{{\cite{aschieri}}}]
Let $A\in \Omega^1(X,\mathfrak{aut}(H))$ be a 1-form and let
 $p : P \to X$ be a principal $H$-bibundle. For any $A$-twisted connection
$\phi$ on $P$ there exists a unique 1-form $A_{\phi} \in \Omega^1(X,\mathfrak{aut}(H))$ satisfying
\begin{equation*}
    p^* A_{\phi} =
    \mathrm{Ad}_{g} (p^* A )
    - g^* \bar\theta
    -
    t_{*} \circ \phi
    \text{,}
\end{equation*}
where $g$ is the map from \erf{47}.
\end{lemma}

A twisted connection in a principal bibundle $P$ gives rise to parallel transport maps
\begin{equation*}
\tau_{\gamma}: P_x \to P_y
\end{equation*}
between the fibres of $P$
over points $x,y$ associated to any path $\gamma:x \to y$. It is obtained in the same way as in an ordinary principal bundle but using equation \erf{53} instead of the usual one. As a result of the twist, the maps $\tau_{\gamma}$ are not bi-equivariant; they satisfy
\begin{equation}
\label{54}
\tau_{\gamma}(l_{F_{\phi}(\gamma)(h)}(p)) = l_{h^{-1}}(\tau_{\gamma}(p))
\quad\text{ and }\quad
\tau_{\gamma}(r_h(p)) = r_{F(\gamma)^{-1}(h^{-1})}\tau_{\gamma}(p)
\end{equation}
where $F,F_{\phi}: PX \to \mathrm{Aut}(H)$ are the parallel transport maps associated to the 1-forms $A$ and $A_{\phi}$; see \cite[Proposition 4.7]{schreiber3}. These complicated relations have a very easy interpretation, as we will see in the next section.

Finally, an $A$-twisted connection $\phi$ on a principal $H$-bundle $P$ has a curvature: this is the 2-form
\begin{equation*}
\mathrm{curv}(\phi) := \mathrm{d}\phi + [\phi \wedge \phi] + \alpha_{*}(A \wedge \phi) \in \Omega^2(P,\mathfrak{h})\text{.}
\end{equation*}

For two principal $H$-bibundles $P$ and $P'$ over $X$ one can fibrewise take the tensor product of $P$ and $P'$ yielding a new principal $H$-bibundle $P \times_H P$ over $X$. If the two bibundles are equipped with twisted connections, the bibundle $P \times_H P'$ inherits a twisted connection only  if the two twists
satisfy an appropriate matching condition. Suppose the principal $H$-bibundle $P$ is equipped with an $A$-twisted connection $\phi$, and $P'$ is equipped with an $A'$-twisted connection $\phi'$, and suppose that the matching condition
\begin{equation}
\label{49}
A'_{\phi'} = A
\end{equation}
is satisfied. Then, the tensor product bibundle $P \times_H P'$ carries an $A'$-twisted connection $\phi_{\mathrm{tot}}\in\Omega^1(P \times_H P,\mathfrak{h})$
characterized uniquely by the condition that
\begin{equation*}
    \mathrm{pr}^{*} \phi_{\mathrm{tot}}
    =
    (g \circ p')_* \circ p^*\phi
    +
    p'^* \phi'\text{,}
\end{equation*}
where $\mathrm{pr}: P \times_X P' \to P \times_H P'$ is the projection to the tensor product and $p$ and $p'$ are the projections to the two factors.
This tensor product,  defined only between appropriate pairs of bibundles with twisted connections, turns $\mathfrak{Bibun}_{H}^{\nabla}(X)$
into a \quot{monoidoidal} category. 

A better point of view is to see it as a 2-category: the objects are 
the twists, i.e.  1-forms $A \in \Omega^1(X,\mathfrak{aut}(H))$,
  a 1-morphism $A \to A'$ is a principal $H$-bibundle $P$ with $A'$-twisted connection $\phi$ such that $A'_{\phi'}=A$,  and a 2-morphism $(P,\phi) \Rightarrow (P',\phi')$ is just a morphism of principal $H$-bibundles that respects the $A'$-twisted connections.

\subsubsection{Transport Functors for Bibundles with Twisted Connections}

\label{sec4_3_2}

We are now going to relate the category $\mathfrak{Bibun}_H^{\nabla}(X)$ of principal $H$-bibundles with twisted connections over $X$ to a category of transport functors. Using the terminology of Section \ref{sec:smoothdescent}, we have:

\begin{proposition}
\label{prop4}
There is a surjective and faithful functor
\begin{equation*}
\mathfrak{Bibun}_H^{\nabla}(X) \to \transportX{}{1}{\Lambda\mathcal{B}\mathrm{AUT}(H)}{\Lambda_i\mathcal{B}(\hbitor)}\text{.}
\end{equation*}
\end{proposition}

\begin{proof}
Given a principal $H$-bibundle $P$ with $A$-twisted connection, we define the associated transport functor by
\begin{equation*}
\mathrm{tra}_P\quad:\quad\alxydim{}{x \ar[r]^{\gamma} & y}\quad\mapsto\quad\alxydim{@C=1.5cm@R=1.3cm}{i(*) \ar[d]_{P_x} \ar[r]^{i(F(\gamma))} & i(*) \ar@{=>}[dl]|{\tau_{\gamma}^{-1}} \ar[d]^{P_y} \\ i(*) \ar[r]_{i(F_{\phi}(\gamma))} & i(*)\text{.}}
\end{equation*}
Here $F,F_{\phi}:PX \to \mathrm{Aut}(H)$ are the maps defined by $A$ and $A_{\phi}$ that we have already used in the previous section. The definition contains the claim that the parallel transport map $\tau_{\gamma}$ gives a bi-equivariant map
\begin{equation*}
\tau_{\gamma}^{-1}: P_y \;\times_H \;_{F(\gamma)}H \to \;_{F_{\phi}(\gamma)}H \;\times_H\; P_x\text{;}
\end{equation*}
it is indeed easy to check that this is precisely the meaning of equations \erf{54}. A morphism $f:P \to P'$ between bibundles with $A$-twisted connections induces a natural transformation $\eta_f: \mathrm{tra}_P \to \mathrm{tra}_{P'}$ between the associated  functors, whose component at a point $x$ is the bi-equivariant map $f_x:P_x \Rightarrow P_x'$. This is a particular morphism in $\Lambda_i\mathcal{B}(\hbitor)$ for which the horizontal 1-morphisms are identities. Here it becomes clear that the assignments
\begin{equation*}
(P,\phi) \mapsto \mathrm{tra}_P\quad\text{ and }\quad f\mapsto \eta_f
\end{equation*}
define a functor which is faithful but not full. 

It remains to check that the functor $\mathrm{tra}_P$ is a transport 2-functor. We leave it as an exercise for the reader to construct a local trivialization $(t,\mathrm{triv})$ of $\mathrm{tra}_P$ with smooth descent data.
Hint: use an ordinary local trivialization of the bibundle $P$ and follow the proof of Proposition 5.2 in \cite{schreiber3}.
\end{proof}

Both categories in Proposition \ref{prop4} have the feature that they have tensor products between appropriate objects. Concerning the bibundles with twisted connections, we have described this in terms of the matching condition \erf{49} on the twists. Concerning the category of transport functors, this tensor product is inherited from the one on $\Lambda_i\mathcal{B}(\hbitor)$, which has been discussed in Section \ref{sec:smoothdescent}.

\begin{lemma}
\label{lem7}
The matching condition \erf{49} corresponds to the required condition for tensor products in $\Lambda_i\mathcal{B}(\hbitor)$ under the functor of Proposition \ref{prop4}. Moreover, the functor respects tensor products whenever they are well-defined.
\end{lemma}

\begin{proof}
Suppose that the matching condition $A'_{\phi'}=A$ holds, so that principal $H$-bibundles $P$ and $P'$ with connections $\phi$ and $\phi$ have a tensor product. It follows that the map $F_{\phi'}$  which labels the horizontal 1-morphisms at the bottom of the images of $\mathrm{tra}_{P'}$ is equal to the map  $F$ which labels the ones at the top of the images of $\mathrm{tra}_{P}$; this is the required condition for the existence of the tensor product $\mathrm{tra}_{P'} \otimes \mathrm{tra}_{P}$. That the tensor products are respected follows from the definition of the twisted connection $\phi_{\mathrm{tot}}$ on the tensor product bibundle.
\end{proof}

An alternative formulation of Lemma \ref{lem7} would be that the functor of Proposition \ref{prop4} respects the \quot{monoidoidal structures}, or, that it is a double functor between (weak) double categories.

\subsubsection{Non-Abelian Bundle Gerbes as Transport 2-Functors}

Let $\pi:Y \to M$ be a surjective submersion. We recall:

\begin{definition}[{{\cite{aschieri}}}]
\label{def:nagrb}
A \emph{non-abelian $H$-bundle gerbe with connection} with respect to $\pi$ is a 2-form $B\in\Omega^2(Y,\mathfrak{h})$, a 1-form $A \in \Omega^1(Y,\mathfrak{aut}(H))$, a principal $H$-bibundle $p\maps P \to Y^{[2]}$ with $A$-twisted connection $\phi$ such that
\begin{equation}
\label{56}
\mathrm{curv}(\phi) = (\pi_1 \circ p)^{*}B - (\alpha_{g})_{*} \circ (\pi_2 \circ p)^{*}B\text{,}
\end{equation}
and an associative morphism 
\begin{equation*}
\mu: \pi_{23}^{*}P \;\times_H \; \pi_{12}^{*}P \to \pi_{13}^{*}P
\end{equation*}
of bibundles over $Y^{[3]}$ that respects the twisted connections.
\end{definition}

In \erf{56}, $g$ is the smooth map \erf{47} and $\alpha: \mathrm{Aut}(H) \times H \to H$ is the evaluation.   The definitions of  1-morphisms and  2-morphisms generalize in a straightforward way the definitions from the abelian case.  This defines the 2-category $H\text{-}\mathfrak{BGrb}^{\nabla}(\pi)$.  The 2-category $H\text{-}\mathfrak{BGrb}^{\nabla}(M)$ of non-abelian $H$-bundle gerbes over $M$ is defined as the direct limit,
\begin{equation*}
H\text{-}\mathfrak{BGrb}^{\nabla}(M) := \lim_{\overrightarrow{\pi}} H\text{-}\mathfrak{BGrb}^{\nabla}(\pi)\text{.}
\end{equation*}
We remark:
\begin{enumerate}[(i)]
\item
A non-abelian $S^1$-bundle gerbe is \textit{not} the same as an abelian $S^1$-bundle gerbe: for the non-abelian bundle gerbes also the automorphisms are important, and $\mathrm{Aut}(S^1)\cong \Z_2$. For transport 2-functors this is even more obvious: the Lie 2-groups $\mathcal{B}S^1$ and $\mathrm{AUT}(S^1)$ are not equivalent.

\item
The non-abelian bundle gerbes of Definition \ref{def:nagrb} are a priori not \quot{fake-flat}, i.e. the relation
\begin{equation*}
t_{*}(B) = \mathrm{d}A + [A \wedge A]
\end{equation*}
may not be satisfied, see Remark \ref{rem2}. We denote the sub-2-categories of fake-flat non-abelian $H$-bundle gerbes by $H\text{-}\mathfrak{BGrb}^{\nabla_{\!f\!f}}(-)$.
\end{enumerate}

\begin{proposition}
\label{prop:locequivnagrb}
For every surjective submersion $\pi:Y \to M$, there is a 2-functor
\begin{equation*}
H\text{-}\mathfrak{BGrb}^{\nabla_{\!f\!f}}(\pi) \to \transsmooth{i}{2}{}\text{.}
\end{equation*}
which is surjective and faithful on Hom-categories. In particular, it induces an injective map on isomorphism classes of objects.
\end{proposition}

\begin{proof}
All relations concerning the bimodules are analogous to those in the abelian case, upon generalizing Lemma \ref{lem6} to Proposition \ref{prop4} and Lemma \ref{lem7}. Relation \erf{56} for the 2-form $B$ can be proved in the same way as  Lemma \ref{lem8}, but now using the full version of equation \erf{55}.
The comments concerning the descent datum $\psi$ also remain valid. 
\end{proof}

In the direct limit, the 2-functors of Proposition \ref{prop:locequivnagrb} induce a 2-functor
\begin{equation}
\label{dirlimfun}
H\text{-}\mathfrak{BGrb}^{\nabla_{\!f\!f}}(M) \to \transsmoothpi{i}{2}{\mathrm{Gr}}{}_M\text{.}
\end{equation}
Since the direct limit respects the operation of forming isomorphism classes of objects, and on the level of sets a direct limit of injective maps is injective, the 2-functor \erf{dirlimfun} induces an injective map on isomorphism classes of objects. This completes the proof of Theorem \ref{th:nagrb}.

\subsection{2-Bundles with Connections}

\label{sec4_4}

Apart from the equivalence between transport functors and principal $G$-bundles with connection (Theorem \ref{th2}), our paper \cite{schreiber3} also contains an analogous equivalence for vector bundles with connection. It has
 an immediate generalization to 2-vector bundles with many applications, on which we shall give a brief outlook. 

Suppose $2\mathrm{Vect}$ is  some 2-category of 2-vector spaces. Given a 2-group $\mathfrak{G}$, a
representation of $\mathfrak{G}$ on such a 2-vector space is a 2-functor
\begin{equation*}
  \rho : \mathcal{B}\mathfrak{G} \to 2\mathrm{Vect}\text{.}
\end{equation*}
A \emph{2-vector bundle with connection} and structure 2-group $\mathfrak{G}$  is nothing but a  transport 2-functor
$
  \mathrm{tra} : \mathcal{P}_2(X) \to 2\mathrm{Vect}
$
with $\mathcal{B}\mathfrak{G}$-structure.
Important classes of 2-vector bundles are line 2-bundles and string bundles.

\subsubsection{Models for 2-Vector Spaces}

Depending on the precise application there is some flexibility in
what one may want to understand under a 2-vector space. Usually
 2-vector spaces are abelian module
categories over a given monoidal category. For $k$ a field, two important classes of examples are the following.
First, let  $\widehat{k}$ be
the discrete monoidal category over $k$. Then, $2\mathrm{Vect}$ is 2-category of module categories over $\widehat{k}$. This is 
equivalent to the 2-category of categories internal to $k$-vector spaces. These
\emph{Baez-Crans 2-vector spaces} \cite{baez7} are appropriate for the discussion of Lie 2-algebras.

The second model for $2\mathrm{Vect}$ is the 2-category of module categories over the monoidal category $\mathrm{Vect}(k)$ of $k$-vector spaces,
\begin{equation*}
2\mathrm{Vect} := \mathrm{Vect}(k)\text{-}\mathrm{Mod}\text{.}
\end{equation*}
In its totality this is rather unwieldy, but it contains two important  sub-2-categories: the 2-category $\mathrm{KV}(k)$ of \emph{Kapranov-Voevodsky 2-vector spaces} \cite{kapranov1} and the 2-category $\mathrm{Bimod}(k)$, whose objects are $k$-algebras, whose 1-morphisms are bimodules over these algebras and whose 2-morphisms are bimodule homomorphisms \cite{shulman1}. Indeed, there is an  inclusion 2-functor
\begin{equation*}
\iota: \mathrm{Bimod}(k) \hookrightarrow \mathrm{Vect}(k)\text{-}\mathrm{Mod}
\end{equation*}
that sends a $k$-algebra $A$ to the category $A\text{-}\mathrm{Mod}$ of ordinary (say, right) $A$-modules. This is a module category over $\mathrm{Vect}(k)$ by tensoring a right module from the left by a vector space. A 1-morphism, an $A$-$B$-bimodule $N$, is sent to the functor that tensores a right $A$-module  from the right by $N$, yielding a right $B$-module. A bimodule morphism induces evidently a natural transformation of these functors.

If one restricts the 2-functor $\iota$ to the full sub-2-category formed by
those algebras that are direct sums  $A = k^{\oplus n}$ of the ground
field algebra, the 2-vector spaces in the image of $\iota$ are of the form $\mathrm{Vect}(k)^n$, i.e. tuples of vector spaces. The 1-morphisms in the image are $(m\times n)$-matrices whose entries are $k$-vector spaces. These form  the 2-category of Kapranov-Voevodsky 2-vector spaces \cite{kapranov1}.

\subsubsection{The canonical Representation of a 2-Group}

Every automorphism 2-group $\mathrm{AUT}(H)$ of a Lie group $H$ has a
canonical representation on 2-vector spaces, namely
\begin{equation}
\label{30}
\alxydim{}{\mathcal{B}\mathrm{AUT}(H) \ar[r]^-{A} & \mathrm{Bimod}(k) \ar[r]^-{\iota} & \mathrm{Vect}(k)\text{-}\mathrm{Mod}\text{,}}
\end{equation}
where the  2-functor $A$ is defined similar as the one we have used for the non-abelian bundle gerbes in \erf{48}. It sends the single object to $k$ regarded as a $k$-algebra, it sends a 1-morphism $\varphi\in \mathrm{Aut}(H)$ to the bimodule $_{\varphi}k$ in the notation of Section \ref{sec4_3_2}, and it sends a 2-morphism $(\varphi,h):\varphi\Rightarrow  c_h \circ \varphi$ to the multiplication with $h$ from the left. 

Now let $\mathfrak{G}$ be any smooth Lie 2-group corresponding to a smooth crossed module $(G,H,t,\alpha)$. We have a canonical 2-functor
\begin{equation}
\label{39}
\mathcal{B}\mathfrak{G} \to \mathcal{B}\mathrm{AUT}(H)\quad:\quad \bigon{\ast}{\ast}{g}{g'}{h}\mapsto \bigon{\ast}{\ast}{\alpha_g}{\alpha_{g'}}{h}
\end{equation}
whose composition with \erf{30} gives a representation of $\mathfrak{G}$, that we call the \emph{canonical $k$-representation}.

\begin{example} 
\label{ex1}
A very simple but useful example is the canonical $\C$-representation of
$\mathcal{B}\C^{\times}$.
In this case the composition \erf{30} is the 2-functor
\begin{equation*}
\rho\quad:\quad \mathcal{BB}\C^{\times} \to \mathrm{Vect}(\C)\text{-}\mathrm{Mod}\quad:\quad \bigon{\ast}{\ast}{\id}{\id}{z} \mapsto \bigon{\mathrm{Vect}(\C)}{\mathrm{Vect}(\C)}{- \otimes \C}{-\otimes \C}{- \cdot z}
\end{equation*}
for all $z \in \C^{\times}$. Notice that $\mathrm{Vect}(\mathbb{C})$ is the canonical
1-dimensional 2-vector space over $\C$ in the same sense in that $\mathbb{C}$ is
the canonical 1-dimensional complex 1-vector space. Therefore, transport 2-functors
\begin{equation*}
\mathrm{tra}:\mathcal{P}_2(M) \to \mathrm{Vect}(\C)\text{-}\mathrm{Mod}
\end{equation*}
with $\mathcal{BB}\C^{\times}$-structure  deserve to be addressed as
\emph{line 2-bundles with connection}. Two remarks:
\begin{enumerate}
\item 
Going through the discussion of abelian bundle gerbes with connection in Section \ref{sec3_2} it is easy to see  that line 2-bundles with connection are equivalent to bundle gerbes with connection defined via line bundles instead of circle bundles.
\item
The fibre $\mathrm{tra}(x)$ of a line 2-bundle $\mathrm{tra}$ at a point $x$ is an algebra which is Morita equivalent to the ground field $\C$. These are exactly the finite rank operators on a separable Hilbert space. Thus, line 2-bundles with connection are a form of bundles of finite rank operators with connection, this is the point of view taken in \cite{bouwknegt1}.
\end{enumerate}
\end{example}
 
The  2-functor $A:\mathcal{B}\mathrm{AUT}(H) \to \mathrm{Bimod}(k)$ we have used above can be deformed to a 2-functor $A^{\rho}$ using an ordinary representation $\rho: \mathcal{B}H \to \mathrm{Vect}(k)$ of $H$. It sends the object of $\mathcal{B}\mathrm{AUT}(H)$ to the algebra $A^{\rho}(\ast)$ which is the vector space generated from all the linear maps $\rho(h)$. A 1-morphism $\varphi\in\mathrm{Aut}(H)$ is again sent to the bimodule $_{\varphi}A^{\rho}(\ast)$, and the 2-morphisms as before to left multiplications. The original 2-functor is reproduced by $A=A^{\mathrm{triv}_{k}}$ from the trivial representation of $H$ on $k$.

\begin{example}
For $G$ a compact simple and simply-connected Lie group, we consider the level $k$ central extension $H_k:=\hat\Omega_k G$ of the group of based loops in $G$. For a positive energy representation $\rho: \mathcal{B}\hat\Omega_k G \to \mathrm{Vect}(k)$ the algebra $A^{\rho}(\ast)$ turns out to be a von Neumann-algebra while the bimodules $_{\varphi}A^{\rho}(\ast)$ are Hilbert bimodules. In this infinite-dimensional case we have to make the composition of 1-morphisms more precise: here we take not the algebraic tensor product of these Hilbert bimodules but the  Connes fusion tensor product
\cite{stolz1}. Connes fusion product still respects the composition: for $A$ a von Neumann algebra
and ${}_\varphi A$ the bimodule structure on it induced from twisting the left action by an algebra
automorphism $\varphi$, we have
\begin{equation*}
  {}_\varphi A \otimes {}_{\varphi'}A \simeq {}_{\varphi' \circ \varphi} A
\end{equation*}
under the Connes fusion tensor product. Now let $\mathfrak{G}=\mathrm{String}_k(G)$ be the string 2-group defined from the  crossed module $\hat\Omega_k G \to P _0G$ of Fréchet Lie groups \cite{baez9}. Together with the projection 2-functor \erf{39} we obtain an induced representation
\begin{equation*}
i:\mathcal{B}\mathrm{String}_k(G) \to \mathrm{Bimod}_{\mathrm{CF}}(k)
\end{equation*}
The fibres of a transport 2-functor 
\begin{equation}
\mathrm{tra}: \mathcal{P}_2(M) \to \mathrm{Bimod}_{\mathrm{CF}}(k)
\label{38}
\end{equation}
with $\mathcal{B}\mathrm{String}_k(G)$-structure are hence von Neumann algebras, and its parallel transport along a path is a Hilbert bimodule for these fibres. 
In conjunction with the result \cite{baez8,baas2}
that $\mathrm{String}_k(G)$-2-bundles have the same classification as
ordinary fibre bundles whose structure group is the topological
String group, this says that transport 2-functors \erf{38} have to be addressed as \emph{String 2-bundles with connection}, see \cite{stolz1}.
\end{example}

\subsubsection{Twisted Vector Bundles}

Vector bundles over $M$ twisted by a class $\xi\in H^3(M,\Z)$ are the same thing as gerbe modules for a bundle gerbe $\mathcal{G}$ whose Dixmier-Douady class is $\xi$ \cite{bouwknegt1}. These modules are in turn nothing else but certain (generalized) 1-morphisms in the 2-category of bundle gerbes $\mathfrak{BGrb}(M)$ \cite{waldorf1}. The same is true for connections on twisted vector bundles. More precisely, a twisted vector bundle with connection is the same as a 1-morphism
\begin{equation*}
\mathcal{E}:\mathcal{G} \to \mathcal{I}_{\rho}
\end{equation*}
from the bundle gerbe $\mathcal{G}$ with connection to the trivial bundle gerbe $\mathcal{I}$ equipped with the connection 2-form $\rho \in \Omega^2(M)$. 

Now let 
\begin{equation*}
\mathrm{tra}:\mathcal{P}_2(M) \to \mathrm{Vect}(\C)\text{-}\mathrm{Mod}
\end{equation*}
be a transport 2-functor which plays the role of the bundle gerbe $\mathcal{G}$, but we allow an arbitrary structure 2-group $\mathfrak{G}$ and any representation $\rho:\mathcal{B}\mathfrak{G} \to \mathrm{Vect}(\C)\text{-}\mathrm{Mod}$. Let $\mathrm{tra}^{\infty}: \mathcal{P}_2(M) \to \mathcal{B}\mathfrak{G}$ be a smooth 2-functor which plays the role of the trivial bundle gerbe. We shall now consider transport transformations
\begin{equation*}
A:\mathrm{tra} \to \mathrm{tra}^{\infty}_{\rho}\text{.}
\end{equation*} 
Let $\pi:Y \to M$ be a surjective submersion for which $\mathrm{tra}$ admits a local trivialization with smooth descent data $(\mathrm{triv},g,\psi,f)$. The descent data of $\mathrm{tra}^{\infty}$ is of course $(\pi^{*}\mathrm{tra}^{\infty},\id,\id,\id)$. Now the transport transformation $A$ has the following descent data: the first part is a \pt\ $h: \mathrm{triv} \to \pi^{*}\mathrm{tra}^{\infty}$ whose associated functor $\holo(h): \mathcal{P}_1(Y) \to \Lambda_{\rho}(\mathrm{Vect}(\C)\text{-}\mathrm{Mod})$ is a transport functor with $\Lambda \mathcal{B}\mathfrak{G}$-structure. The second part is a modification $\varepsilon: \pi_2^{*}h \circ g \Rightarrow \id \circ \pi_1^{*}h$ whose associated natural transformation
\begin{equation*}
\holo(\varepsilon): \pi_2^{*}\holo(h) \otimes \holo(g) \to \pi_1^{*}\holo(h)
\end{equation*}
is a morphism of transport functors over $Y^{[2]}$. According to the coherence conditions on descent 1-morphisms, it fits into the commutative diagram
\begin{equation}
\label{62}
\alxydim{@C=2cm}{\pi_3^{*}\holo(h) \otimes \pi_{23}^{*}\holo(g) \otimes \pi_{12}^{*}\holo(g) \ar[r]^-{\pi_{23}^{*}\holo(\varepsilon) \otimes \id} \ar[d]_{\id\otimes \holo(f)} & \pi_2^{*}\holo(h) \otimes \pi_{12}^{*}\holo(g) \ar[d]^{\pi_{12}^{*}\holo(\varepsilon)} \\ \pi_{3}^{*}\holo(h) \otimes \pi_{13}^{*}\holo(g) \ar[r]_{\pi_{13}^{*}\holo(\varepsilon)} & \pi_1^{*}\holo(h)}
\end{equation}
of morphisms of transport functors over $Y^{[3]}$ and satisfies $\Delta^{*}\holo(\varepsilon) \circ \holo(\psi) =\id$. The transport functor 
\begin{equation*}
\holo(h): \mathcal{P}_1(Y) \to \Lambda_{\rho}(\mathrm{Vect}(\C)\text{-}\mathrm{Mod})
\end{equation*}
together with the natural transformation $\holo(\varepsilon)$ is the general version of a vector bundle with connection twisted by a transport 2-functor $\mathrm{tra}$. According to Sections \ref{sec4_1} and \ref{sec3_3}, the twists can thus be Breen-Messing gerbes or non-abelian bundle gerbes with connection.

Depending on the choice of the representation $\rho$, our twisted vector bundles can be translated into more familiar language. Let us demonstrate this in the  case of Example \ref{ex1}, in which the twist is a line 2-bundle with connection, i.e. a transport 2-functor
\begin{equation*}
\mathrm{tra}: \mathcal{P}_2(M) \to \mathrm{Vect}(\C)\text{-}\mathrm{Mod}
\end{equation*}  
with $\mathcal{BB}\C^{\times}$-structure. In order to obtain the usual twisted vector bundles, we restrict the target 2-category to $\mathcal{B}\mathrm{Vect}(\C)$, the monoidal category of complex vector spaces considered as a 2-category.  The following Lie category $\mathrm{Gl}$ is appropriate: its objects are the natural numbers $\N$, and it has only morphisms between equal numbers, namely all matrices $\mathrm{Gl}_n(\C)$. The composition is the product of matrices. The Lie category $\mathrm{Gl}$ is strictly monoidal: the tensor product of two objects $m,n\in\N$ is the product $nm\in\N$, and the one of two matrices $A\in\mathrm{Gl}(m)$ and $B\in\mathrm{Gl}(n)$ is the ordinary tensor product $A\otimes B\in \mathrm{Gl}(m\times n)$. In fact, $\mathrm{Gl}$ carries a second monoidal structure coming from the sum of natural numbers and the direct sum of matrices, so that $\mathrm{Gl}$ is actually a \emph{bipermutative category}, see Example 3.1 of \cite{baas1}.

Notice that we have a canonical inclusion functor 
$\iota: \mathcal{B}\C^{\times} \hookrightarrow \mathrm{Gl}$,
which induces another inclusion
\begin{equation*}
\iota_{*}:\transport{}{2}{\mathcal{BB}\C^{\times}}{\mathcal{B}\mathrm{Vect}(\C)} \to \transport{}{2}{\mathcal{B}\mathrm{Gl}}{\mathcal{B}\mathrm{Vect}(\C)}
\end{equation*}
of line 2-bundles with connection into more general vector 2-bundles with connection. Here we have used the representation
\begin{equation*}
\rho: \mathcal{B}\mathrm{Gl} \to \mathcal{B}\mathrm{Vect}
\end{equation*} 
obtained as a generalization of Example \ref{ex1} from $\C^{\times}=\mathrm{Gl}_1(\C)$ to $\mathrm{Gl}_n(\C)$ for all $n\in \N$. The composition $\rho \circ \iota$ reproduces the representation of Example \ref{ex1}.

Using the above inclusion, the given transport 2-functor $\mathrm{tra}$ induces a transport 2-functor $\iota_{*}\circ \mathrm{tra}: \mathcal{P}_2(M) \to \mathcal{B}\mathrm{Vect}(\C)$ with $\mathcal{B}\mathrm{Gl}$-structure, and one can study transport transformations
\begin{equation*}
A:\mathrm{tra} \to \mathrm{tra}^{\infty}_{\rho}
\end{equation*} 
in that greater 2-category $\transport{}{2}{\mathcal{B}\mathrm{Gl}}{\mathcal{B}\mathrm{Vect}(\C)}$. Along the lines of the general procedure described above, we have transport functors $\holo(g)$ and $\holo(h)$ coming from the descent data of $\mathrm{tra}$ and $A$, respectively. In the present particular situation, the first one takes values in the category $\Lambda_{\rho\circ\iota}\mathcal{B}\mathrm{Vect}_1(\C)$ whose objects are one-dimensional complex vector spaces and whose morphisms from $V$ to $W$ are invertible linear maps $f: W \otimes \C \to \C \otimes V$. Similar to Lemma \ref{lem6}, this category is equivalent to the category $\mathrm{Vect}_1(\C)$ of one dimensional complex vector spaces itself. Thus, the transport functor $\holo(g)$ with $\mathcal{B}\C^{\times}$-structure is a complex line bundle $L$ with connection over $Y^{[2]}$. The second transport functor, $\holo(h)$, takes values in the category $\Lambda_{\rho\circ\iota}\mathcal{B}\mathrm{Vect}(\C)$. This category is equivalent to the category $\mathrm{Vect}(\C)$ itself. It has $\Lambda_{\iota}\mathcal{B}\mathrm{Gl}$-structure, which is equivalent to $\mathrm{Gl}$. Thus, $\holo(h)$ is a transport functor with values in $\mathrm{Vect}(\C)$ and $\mathrm{Gl}$-structure. It thus corresponds to a finite rank vector bundle $E$ over $Y$ with connection. 

Since all identifications we have made so far a functorial, the morphisms $\holo(f)$ and $\holo(\varepsilon)$ of transport functors induce morphisms of vector bundles that preserve the connections, namely an associative morphism
\begin{equation*}
\mu: \pi_{23}^{*}L \otimes \pi_{12}^{*}L \to \pi_{13}^{*}L
\end{equation*}
of line bundles over $Y^{[2]}$, and a morphism
\begin{equation*}
\varrho: \pi_2^{*}E \otimes L \to \pi_1^{*}E
\end{equation*}
of vector bundles over $Y$ which satisfies a compatibility condition corresponding to \erf{62}. This reproduces the definition of a twisted vector bundle with connection \cite{bouwknegt1}.
We remark that the 2-form $\rho$ that corresponds to the smooth 2-functor $\mathrm{tra}^{\infty}_{\rho}$ which was the target of the transport transformation $A$ we have considered, is related to the curvature of the connection on the vector bundle $E$: it requires that
\begin{equation*}
\mathrm{curv}(E) = I_n \cdot (\mathrm{curv}(L) - \pi^{*}\rho)\text{,}
\end{equation*}
where $I_n$ is the identity matrix and $n$ is the rank of $E$. This condition can be derived similar to Lemma \ref{lem8}.

\setsecnumdepth{3}

\section{Surface Holonomy}

\label{sec5}
\label{sec:paralleltransport}

From the viewpoint of a  transport 2-functor, parallel transport and holonomy are basically evaluation on paths or bigons.  
Let $\mathrm{tra}: \mathcal{P}_2(M) \to T$ be a transport 2-functor with $\mathcal{B}\mathfrak{G}$-structure on $M$. Its fibres over points $x,y\in M$ are objects $\mathrm{tra}(x)$ and $\mathrm{tra}(y)$ in $T$, and we say that its \emph{parallel transport along a path} $\gamma:x \to y$ is given by the 1-morphism
\begin{equation*}
\mathrm{tra}(\gamma): \mathrm{tra}(x) \to \mathrm{tra}(y)
\end{equation*} 
in $T$, and its \emph{parallel transport along a bigon} $\Sigma: \gamma \Rightarrow \gamma'$ is given by the 2-morphism
\begin{equation*}
\mathrm{tra}(\Sigma): \mathrm{tra}(\gamma) \Rightarrow \mathrm{tra}(\gamma')
\end{equation*}
in $T$. 
The rules how parallel transport behaves under the composition of paths and bigons are precisely the axioms of the 2-functor $\mathrm{tra}$; see \cite[Definition A.5]{schreiber6}. We give some examples. If $\gamma_1:x \to y$ and $\gamma_2:y \to z$ are composable paths, the separate parallel transports along the two paths are related to the one along their composition by the compositor
\begin{equation*}
c_{\gamma_1,\gamma_2}: \mathrm{tra}(\gamma_2) \circ \mathrm{tra}(\gamma_1) \Rightarrow \mathrm{tra}(\gamma_2\circ\gamma_1)
\end{equation*}
of the 2-functor $\mathrm{tra}$.
If $\id_x$ is the constant path at $x$, the parallel transport along $\id_x$ is related to the identity at the fibre $\mathrm{tra}(x)$ by the unitor
\begin{equation*}
u_x: \mathrm{tra}(\id_x) \Rightarrow \id_{\mathrm{tra}(x)}\text{.}
\end{equation*}
The parallel transports along vertically composable bigons $\Sigma:\gamma_1\Rightarrow \gamma_2$ and $\Sigma':\gamma_2 \Rightarrow \gamma_3$ obey
\begin{equation*}
\mathrm{tra}(\Sigma' \bullet \Sigma) = \mathrm{tra}(\Sigma') \bullet \mathrm{tra}(\Sigma)\text{.}
\end{equation*}
In the following we focus on certain bigons that parameterize surfaces; the parallel transport along these bigons will be called the holonomy of the transport 2-functor $\mathrm{tra}$.

\subsection{Markings and Fundamental Bigons}

\label{sec:holonomy}

Surface holonomy of gerbes has so far only be studied in the abelian case, i.e. for gerbes with structure 2-group $\mathfrak{G}=\mathcal{B}S^1$; see e.g. \cite{murray}. It is understood  that an abelian gerbe with connection over $M$  provides an $S^1$-valued surface holonomy for smooth maps $\phi:S \to M$ defined on \emph{closed oriented surfaces} $S$. Various extensions have been studied for oriented surfaces with boundary \cite{carey2,gawedzki1}, closed unoriented (in particular unorientable) surfaces \cite{schreiber1}, and unoriented surfaces with boundary \cite{gawedzki8}.

Let us try  to explain why the step from  abelian to non-abelian surface holonomy is so difficult, by looking at the analogous but easier situation of ordinary holonomy of a  connection on a principal bundle $P$ over $M$. In the abelian case, say with structure group $S^1$, holonomy is defined for \emph{closed oriented curves} in $M$. Passing to a non-abelian structure group $G$, we have two possibilities. The first is  to choose additional structure on the closed oriented curve, namely a point $x\in M$, in which case the holonomy is a well-defined automorphism of the fibre $P_x$ of the bundle $P$  over the point $x$. The second possibility is to project the value of the holonomy into an appropriate quotient, for example along the map $\mathrm{Aut}(P_x) \to G\red$ into the set $G\red$ of conjugacy classes of $G$. This quotient is designed such that a different choice of $x$ yields the same element in $G\red$, and so $P$ has a well-defined $G\red$-valued holonomy for closed oriented curves.

Choosing a point in a closed oriented curve means to represent it as a  path (up to thin homotopy). 
The objective of this section is to describe how to represent a closed oriented surface $\phi:S \to M$ as a bigon in $M$. 
It will be convenient to adapt the following standard terminology from the theory of Riemann surfaces to our setting. Suppose $S$ is a closed oriented surface of genus $g$.  A \emph{marking} of $S$ is a point $x\in S$ together with a   set $\EuScript{M}=\left \lbrace \alpha_i,\beta_i \right \rbrace_{i=1}^g$ of paths $\alpha_i:x \to x$ and $\beta_i:x\to x$ such that the homotopy classes $[\alpha_i]$ and $[\beta_i]$ of the paths form a presentation of the fundamental group of $S$ based at $x$ with the relation that $[\tau_{\EuScript{M}}] = 1$, where
\begin{equation*}
\tau_{\EuScript{M}} := \beta_g^{-1}\circ  \alpha_g^{-1} \circ  \beta_g \circ \alpha_g  \circ \hdots  \circ \beta_1^{-1}\circ  \alpha_1^{-1} \circ  \beta_1 \circ \alpha_1\text{.}
\end{equation*}
In other words,
\begin{equation*}
\pi_1(S,x) = \left \langle  [\alpha_i],[\beta_i] \;|\; [\tau_{\EuScript{M}}] = 1  \right \rangle\text{.}
\end{equation*}

Next we formulate --- in the language of bigons ---  the statement that the path $\tau_{\EuScript M}$ can be collapsed to the point $x$ in such a way that the collapsing homotopy parameterizes the surface $S$. The challenge is that such a homotopy cannot simultaneously be a bigon and a good parameterization. We propose the following notion.

\begin{definition}
\label{def:fundamentalbigon}
A \emph{fundamental bigon} with respect to the marking $(x,\EuScript M)$ is a bigon $\Sigma: \tau_{\EuScript{M}} \Rightarrow \id_x$ such that there exist:
\begin{enumerate}[(a)]

\item 
a polygon $\mathcal{P} \subset \R^2$ and a surjective smooth map $\pi: \mathcal{P} \to S$ that is an embedding when restricted to the interior of $\mathcal{P}$ and preserves the orientations. 

\item
a surjective continuous map $k: [0,1]^2 \to \mathcal{P}$ that restricts to a diffeomorphism between the interiors and preserves the orientations, such that there exists a vertex $v_0$ of $\mathcal{P}$ with $v_0=k(s,0)=k(s,1)=k(1,t)$ for all $0 \leq s,t\leq 1$.

\item
a smooth map $h: [0,1] \times \mathcal{P} \to S$ with $h_0 = \pi$ and $h_1 \circ k  = \Sigma$, such that for each edge $e$ of $\mathcal{P}$ the restriction $h|_{[0,1] \times e}$ has rank one, and $h(-,v)=x$ for each vertex $v$. 

\end{enumerate}
\end{definition}

In this definition the polygon $\mathcal{P}$ serves as the standard parameter domain of the surface, and the map $\pi: \mathcal{P} \to S$ is such a parameterization. The map $k: [0,1]^2 \to \mathcal{P}$ is responsible for the transition between  the standard parameter domain $\mathcal{P}$ and the parameter domain $[0,1]^2$ for bigons. 
The map $h$ ensures that the fundamental bigon $\Sigma$ is homotopic to the \quot{bigonized} standard parameterization, $\pi \circ k$. The various conditions on $k$ and $h$ ensure that the homotopy $h \circ k$ is constantly equal to $x$ over the boundary parts of $[0,1]^2$ that are parameterized by $(s,0)$, $(s,1)$, and $(1,t)$, and ensure that it restricts to a thin homotopy between the path $(\pi \circ k)(0,-)$ and $\tau_{\EuScript M}$ over the remaining boundary part $(0,t)$.

\begin{comment}
such that there exist
\begin{enumerate}[(a)]

\item 
a smooth surjective map $\pi: \mathcal{P}_{4g} \to S$ whose restriction to the interior is an embedding.

\item
a smooth map $h: [0,1] \times \C \to S$ such that $h_0|_{\mathcal{P}_{4g}} = \pi$ and $h_1|_{[0,1]^2} = \Sigma$.

\end{enumerate}

that there exists  a regular deformation of $\Sigma$: a map $f:[0,1]^2 \to S$ with the following properties:
\begin{enumerate}[(i)]

\item 
It is a fixed-ends-homotopy between the path $\tau_f(t) := f(0,t)$ and $\id_x$.  That is, it satisfies  $f(s,0)=f(s,1)=f(1,t)=x$ for all $s,t$. 

\item
It is a regular parameterization of $S$. That is, $f$ is continuous and surjective and a smooth embedding when restricted to the interior $(0,1)^2$. 

\item
It represents the bigon $\Sigma$ up to homotopy. That is, there exists a smooth map $h:[0,1] \times [0,1]^2 \to S$ such that $h(\sigma,s,0)=h(\sigma,s,1)=h(\sigma,1,t)=x$, such that $h(0,s,t)=f(s,t)$ and $h(1,s,t)=\Sigma(s,t)$, and such that the homotopy $h(-,0,-)$ between $\tau_f$ and $\tau_{\EuScript M}$ is thin.  

\end{enumerate}
\end{comment}

A marking will be called \emph{good}, if it admits a fundamental bigon. The following three lemmata elaborate some properties of good markings and fundamental bigons. 

\begin{lemma}
Every closed oriented surface $S$ has a good marking.
\end{lemma}

\begin{proof}
It is well-known that every surface $S$ of positive genus $g>0$ has a \quot{fundamental polygon} $\mathcal{P}\subset \C$ with $4g$ edges, together with a map $\pi: \mathcal{P} \to S$ that has all  required properties. It can be arranged such that all vertices go to a single point $x\in S$, and ---
if the $4g$ edges are parameterized by $A_1,B_1,A_1',B_1',A_2,...: [0,1] \to \mathcal{P}$ in counter-clockwise order starting at a vertex $v_0$ --- then \begin{equation}
\label{pathrel}
\pi(A_i(1-t)) = \pi(A_i'(t))
\quand
\pi(B_i(1-t)) = \pi(B_i'(t))\text{.}
\end{equation} 
Let $D \subset \C$ denote the disc centered at $-1\in \C$. It is obvious that $\mathcal{P}$ and $D$ are homeomorphic. The homeomorphism can be arranged such that
(i) it preserves orientations, (ii) 
it is a smooth embedding when restricted to the complement of the vertices of $\mathcal{P}$, and (iii)
the vertices map to points equidistantly distributed over the boundary of $D$, with $v_0$ at $0\in \C$.
We define
\begin{equation*}
k': [0,1]^2 \to D: (s,t) \mapsto (1-s)(e^{2\pi\mathrm{i}t}-1)\text{.}
\end{equation*}
Under the homeomorphism between $\mathcal{P}$ and $D$, this gives the map $k$ with all required properties.

The map $h: [0,1] \times \mathcal{P} \to S$ is chosen such that $h_1: \mathcal{P} \to S$ is  constantly equal to $x$ in neighborhoods of all vertices, and such that \erf{pathrel} remains true with $h_1$ instead of $\pi$. Then we define
\begin{equation*}
\Sigma := h_1 \circ k
\end{equation*}
This is smooth because $h_1$ is locally constant at all points where $k$ is not smooth. We consider for $0 \leq k \leq 4g-1$ the map $p_{k}: [0,1] \to [0,1]: t \mapsto \frac{k+t}{4g}$ that squeezes the interval into the $k$-th of $4g$ many pieces. Then, we define for $i=1,..,g$
\begin{equation*}
\alpha_i := \pi\circ \Sigma(0,-) \circ p_{4i-4}
\quand
\beta_i := \pi \circ \Sigma(0,-) \circ p_{4i-3}\text{,} \end{equation*}
which give smooth maps with sitting instants, each going from  $x$ to $x$, Thus, $(x,\EuScript M)$ with $\EuScript M := \left \lbrace \alpha_i,\beta_i \right \rbrace$ is a marking, and \erf{pathrel} imply that $\Sigma$ is a bigon $\Sigma: \tau_{\EuScript M} \Rightarrow \id_x$.

The case of genus $g=0$ can be treated in a similar way using the disc $D_1$ with one marked point; this is left as an  exercise.
\end{proof}

The main idea behind the definition of a fundamental bigon is that the integral of a 2-form  $\omega\in \Omega^2(S)$ over $S$ can be computed over a fundamental bigon.

\begin{lemma}
\label{fundamentalbigonintegral}
Suppose $S$ is a closed oriented surface with a good marking $(x,\EuScript M)$. Let $\Sigma$ be a fundamental bigon with respect to $(x,\EuScript M)$. Then,
\begin{equation*}
\int_S \omega = \int_{[0,1]^2}\Sigma^{*}\omega
\end{equation*}
for all 2-forms $\omega \in \Omega^2(S)$.
\end{lemma}

\begin{proof}
We choose the structure  $(\mathcal{P},k,\pi,h)$ of  Definition \ref{def:fundamentalbigon}. Then, 
\begin{equation*}
\int_{S} \omega = \int_{\mathcal{P}} \pi^{*}\omega = \int_{\mathcal{P}} h_1^{*}\omega = \int_{\mathrm{int}(\mathcal{P})} h_1^{*}\omega = \int_{(0,1)^2} \Sigma^{*}\omega = \int_{[0,1]^2} \Sigma^{*}\omega \text{.}
\end{equation*}
The second equality holds because the homotopy $h$ between $h_0=\pi$ and $h_1$ has rank one on the boundary. The other equalities hold immediately due to the assumptions on all the involved maps. 
\end{proof} 

Finally, we show that the choice of a fundamental bigon is essentially unique. 

\begin{lemma}
\label{lem:choicefundamentalbigon}
Let $S$ be a closed oriented surface and  $(x,\EuScript{M})$ be a marking. Suppose $\Sigma \maps \tau_{\EuScript{M}} \Rightarrow \id_x$ and $\Sigma': \tau_{\EuScript{M}} \Rightarrow \id_x$ are two fundamental bigons with respect to the marking $(x,\EuScript{M})$. Then, $\Sigma$ and $\Sigma'$ are thin homotopy equivalent in the sense of Section \ref{sec:path2groupoid}. 
\end{lemma}

\begin{proof}
We choose for both fundamental bigons the structures $(\mathcal{P},k,\pi,h)$ and $(\mathcal{P}',k',\pi',h')$ of Definition \ref{def:fundamentalbigon}. We can assume that $\mathcal{P}=\mathcal{P}'$, since any two polygons are diffeomorphic, and such diffeomorphism can be absorbed into $\pi'$ and $k'$. We can assume that $\mathcal{P}$ is convex; in that case it is easy to see that $k$ and $k'$ are homotopic, with the homotopy  fixing the three boundary components that map to $v_0$, and restricting to a homotopy with values in $\partial P$ over the forth boundary component. 
\begin{comment}
The homotopy is given by
\begin{equation*}
H(\sigma,s,t) := \sigma k(s,t) + (1-\sigma)k(s,t)
\end{equation*}
\end{comment}
The homotopies $h$ and $h'$ restrict on the boundary to thin homotopies $\pi|_{\partial P} \simeq \tau_{\EuScript M}$ and $\pi'|_{\partial P} \simeq \tau_{\EuScript M}$. 
By a version of Alexander's trick, we obtain a smooth homotopy $\pi \simeq \pi'$ that fixes all vertices and has rank one over the edges. 
\begin{comment}
One first changes $\pi$ into a homotopic map $\tilde \pi$ that coincides with $\pi'$ over the boundary. One can still invert $\pi'$ to give a smooth map
\begin{equation*}
\alxydim{}{\mathrm{int}(\mathcal{P}) \ar[r]^-{\tilde\pi} & S\setminus \mathrm{im}(\pi'|_{\mathrm{int}(\mathcal{P})}) \ar[r]^-{\pi'^{-1}} & \mathrm{int}(\mathcal{P})}
\end{equation*}
and then extends that by the identity to a smooth map $\mathcal{P} \to \mathcal{P}$. By Alexander's trick, this map is homotopic to the identity. Composing with $\pi'$ gives a  homotopy between $\tilde\pi$ and $\pi'$ that fixes the boundary. Replacing $\tilde\pi$ by $\pi$ gives the claimed homotopy.
\end{comment}
Then we have chains of homotopies
\begin{equation*}
\Sigma \simeq h_1 \circ k \simeq h_0 \circ k \simeq \pi \circ k\simeq \pi \circ k' \simeq \pi' \circ k' \simeq h_0' \circ k' \simeq h_1' \circ k' \simeq \Sigma' 
\end{equation*}
which are all smooth and have rank one restricted over the boundary. By dimensional reasons, this homotopy can at most have rank two everywhere. This means that it is a thin homotopy between $\Sigma$ and $\Sigma'$.
\end{proof}

Next we come to the definition of surface holonomy. Since the generic transport 2-functor is not strict, we have to deal with the difference between the 1-morphism $\mathrm{tra}(\phi_{*}\tau_{\EuScript M})$ and the 1-morphism
\begin{equation*}
\mathrm{tra}_{\phi,\EuScript M} := \prod_{i=1}^g  \mathrm{tra}(\alpha_i) \circ \mathrm{tra}(\beta_i) \circ \mathrm{tra}(\alpha_i^{-1}) \circ \mathrm{tra}(\beta_i^{-1})\text{,}
\end{equation*}
for which we may agree for an arbitrary convention how to put parentheses in case $T$ has a non-trivial associator. The relation between these two 1-morphisms is given by a 2-morphism
\begin{equation*}
c^{\phi,\mathrm{tra}}: \mathrm{tra}_{\phi,\EuScript M} \Rightarrow \mathrm{tra}(\phi_{*}\tau_{\EuScript M}) 
\end{equation*}
made up from the compositors of $\mathrm{tra}$ in a unique way, due to the coherence axiom for compositors (this is axiom (F3) in \cite[Appendix A]{schreiber6}).

\begin{definition}
\label{def:surfaceholonomy}
Let  $\mathrm{tra}:\mathcal{P}_2(M) \to T$ be a transport 2-functor. Suppose
 $S$ is a closed oriented surface equipped with a good marking $(x,\EuScript{M})$ and a smooth map  $\phi:S \to M$. Let $\Sigma$ be a fundamental bigon for $(x,\EuScript{M})$. The 2-morphism
\begin{equation*}
\mathrm{Hol}_{\mathrm{tra}}(\phi,x,\EuScript{M}) := \mathrm{tra}(\phi_{*}\Sigma) \bullet c^{\phi,\mathrm{tra}}: \mathrm{tra}_{\phi,\EuScript M} \Rightarrow \mathrm{tra}(\id_{\phi(x)}) 
\end{equation*}
in $T$ is called the \emph{surface holonomy} of $\mathrm{tra}$. 
\end{definition}

Note that the surface holonomy  is independent of the choice of the fundamental bigon due to Lemma \ref{lem:choicefundamentalbigon}.
On the other hand, the  surface holonomy does depend on the choice of the  marking $(x,\EuScript{M})$. Even worse, it is not invariant under isomorphisms between transport 2-functors. It is the purpose of the following discussion to improve these dependence issues.

\subsection{Reduced Surface Holonomy}

We consider a Lie 2-group $\mathfrak{G}$  defined from a smooth crossed module $(G,H,t,\alpha)$, and restrict our attention to  transport 2-functors
\begin{equation*}
\mathrm{tra}: \mathcal{P}_2(M) \to T
\end{equation*}
with ${\mathcal{B}\mathfrak{G}}$-structure. Our goal is to replace the 2-functor $\mathrm{tra}$ by another 2-functor which takes values in a certain quotient of $\mathfrak{G}$ where surface holonomy becomes more rigid. In order to construct this quotient, we write $[G,H] \subset H$ for the  subgroup of $H$ that is generated by all elements of the form $h^{-1}\alpha(g,h)$, for $h\in H$ and $g\in G$. The following lemma follows from the  axioms of a crossed module; see Definition \ref{def:crossedmodule}).
\begin{lemma}
\label{subgroup}
The  subgroup $[G,H]$ is normal in $H$ and contains the commutator subgroup,
\begin{equation*}
[H,H] \;\unlhd\;  [G,H] \;\unlhd\; H\text{.}
\end{equation*}
\end{lemma}

\begin{comment}
\begin{proof}
For the first claim suppose $x,h\in H$ and $g\in G$. Using the axioms of the crossed module one checks  that
\begin{equation*}
x\cdot h^{-1}\alpha(g,h) \cdot x^{-1} = h'^{-1}\alpha(g',h')
\end{equation*}
for $h' := xhx^{-1}$ and $g' := t(x)gt(x)^{-1}$. For the second claim note that for $x,h\in H$ and $g=t(x)^{-1}$ we have $h^{-1}\alpha(g,h) = h^{-1}g^{-1}hg$, i.e. every commutator is an element of $[G,H]$.
\end{proof}
\end{comment}

We come to the following important definition.

\begin{definition}
\label{def:redgroup}
Let $\mathfrak{G}$ be a Lie 2-group. Then, the group $\mathfrak{G}\red := H/[G,H]$ is called the reduction of $\mathfrak{G}$.
\end{definition}

By Lemma \ref{subgroup}, the reduction $\mathfrak{G}\red$  is a subgroup of the abelianization $H/[H,H]$ and hence abelian. Note that the projection to the quotient yields a strict 2-functor
\begin{equation*}
\Red: \mathcal{B}\mathfrak{G} \to \mathcal{BB}\mathfrak{G}\red\text{.}
\end{equation*}
\begin{example}
\label{ex4}
Let us look at  examples for Lie 2-groups $\mathfrak{G}$:
\begin{enumerate}[(a)]
\item
In the case of the 2-group $\mathcal{B}A$ for $A$ an ordinary abelian Lie group,  $[1,A]$ is the trivial group, and $(\mathcal{B}A)\red=A$.  

\item
Let $G$ be a  Lie group and let $\mathcal{E}G$ the associated 2-group of inner automorphisms, see Section \ref{sec6_3}. Since $\alpha$ here is the conjugation action of $G$ on itself, $[G,G]$ is indeed the commutator subgroup. Thus $(\mathcal{E}G)\red$ is an abelian Lie group, the abelianization of $G$.

\item
Let $H$ be a connected Lie group and let $\AUT(H)$ be its automorphism 2-group, with $G=\aut(H)$. In this case $[H,H] = [G,H]$  if and only if all automorphisms of $H$ are inner. For $H=S^1$  with $\aut(S^1) = \Z_2$ we get $[\Z_2,S^1]=S^1$, so that $\AUT(S^1)\red=1$. 

\end{enumerate}
\end{example}

Let $\mathrm{tra}: \mathcal{P}_2(M) \to T$ be a transport 2-functor with $\mathcal{B}\mathfrak{G}$-structure. The following definition introduces the replacement for $\mathrm{tra}$ that we want to consider. 

\begin{definition}
A \emph{reduction} of $\mathrm{tra}$ is a transport 2-functor $\mathrm{tra}\red:\mathcal{P}_2(M) \to \mathcal{B}\mathfrak{G}$ with $\mathcal{B}\mathfrak{G}$-structure such that the following two conditions are satisfied:
\begin{enumerate}[(i)]

\item 
\label{reductioni}
There exists a pseudonatural equivalence $i \circ \mathrm{tra}\red \cong \mathrm{tra}$.

\item
\label{reductionii}
The 2-functor $\Red \circ \mathrm{tra}\red$ is normalized in the sense of \cite[Appendix A]{schreiber6}.

\end{enumerate}
\end{definition}

In the following proposition we prove an existence and uniqueness result for reductions. 

\begin{proposition}
\label{prop:reduction}
If $i:\mathcal{B}\mathfrak{G} \to T$ is an equivalence of categories, then every transport 2-functor with $i$-structure admits a reduction. Moreover, two  reductions of the same transport 2-functor are pseudonaturally equivalent. \end{proposition}

\begin{proof}
Since $i: \mathcal{B}\mathfrak{G} \to T$ is an equivalence of 2-categories we can choose an inverse 2-functor $j: T \to \mathcal{B}\mathfrak{G}$, so that $\mathrm{tra}':=j \circ \mathrm{tra}: \mathcal{P}_2(M) \to \mathcal{B}\mathfrak{G}$ is a transport 2-functor with $\mathcal{B}\mathfrak{G}$-structure; see Section \ref{operations}. Since $j$ is inverse to $i$ it is clear that $i \circ \mathrm{tra}' \cong \mathrm{tra}$. The 2-functor $\mathrm{tra}'$ is the first approximation of the reduction whose existence we want to prove.

We consider for any surjective submersion $\pi:Y \to M$ the following strictly commutative diagram of 2-categories and 2-functors:
\begin{equation*}
\alxydim{@C=1.5cm}{\diffco{\mathfrak{G}}{2}{\pi} \ar[r]^-{\mathcal{P}} & \transsmooth{\id_{\mathcal{B}\mathfrak{G}}}{2}{} \ar[d]_{\mathrm{Rec}} \ar[r]^-{\Red \circ -} & \transsmooth{\Red}{2}{} \ar[d]^-{\mathrm{Rec}} \\ & \transport{}{2}{\mathcal{B}\mathfrak{G}}{\mathcal{B}\mathfrak{G}} \ar[r]_-{\Red \circ -} & \transport{}{2}{\mathcal{B}\mathfrak{G}}{\mathcal{BB}\mathfrak{G}\red}\text{.}} \end{equation*}
By Theorem \ref{th:class} the surjective submersion $\pi:Y \to M$ can be chosen such that $\mathrm{Rec} \circ \mathcal{P}$ is essentially surjective, i.e. we can chose a degree two differential $\mathfrak{G}$-cocycle $\xi' \in \diffco{\mathfrak{G}}{2}{\pi}$  such that $\mathrm{Rec}(\mathcal{P}(\xi)) \cong \mathrm{tra}'$. After a refinement of the surjective submersion to an open cover, Lemma \ref{equiavelncenormalized} allows us to assume that $\xi'$ is isomorphic to another cocycle  $\xi''=((A,B),(g,\varphi),\psi,f)$ with $\psi_i=1$, $g_{ii}=1$, $f_{iij}=f_{ijj}=1$ and $f_{iji}=y_{ij}^{-1}\alpha(x_{ij},y_{ij})$ for all $i\in I$ and elements $x_{ij}\in G$, $y_{ij}\in H$. Let $\mathrm{tra}'':= \mathrm{Rec}(\mathcal{P}(\xi''))$; we will prove that this is the reduction we a looking for. 

Condition \erf{reductioni} for reductions is satisfied because $\mathrm{tra}''=\mathrm{Rec}(\mathcal{P}(\xi'')) \cong \mathrm{Rec}(\mathcal{P}(\xi'))\cong \mathrm{tra}'$. In order to show that the second condition is satisfied, we  claim that the object $\Red \circ \mathcal{P}(\xi'')$ in $\transsmooth{\Red}{2}{}$  is normalized in the sense of  \cite[Definition 2.2.1]{schreiber6}. Then,  by \cite[Lemma 3.3.4]{schreiber6}, the reconstructed transport 2-functor  is normalized. By commutativity of the diagram, this 2-functor is $\Red \circ \mathrm{tra}''$. 

In order to prove that claim, we first remark that  $(\mathrm{triv},g,\psi,f) := \mathcal{P}(\xi'')$ is \quot{almost} normalized. Indeed, the conditions $g_{ii}=1$ and $\psi_i=1$ imply under the 2-functor $\mathcal{P}$ that $\id_{\mathrm{triv}_i} = \Delta^{*}g$ and $\psi = \id_{\Delta^{*}g}$. The remaining conditions $g \circ \Delta_{21}^{*}g= \id_{\pi_1^{*}\mathrm{triv}_i}$ and $\Delta_{121}^{*}f = \id_{\Delta_{11}^{*}g}$ are not satisfied. However, since $f_{iji}=y_{ij}^{-1}\alpha(x_{ij},y_{ij})$, after composing with $\Red$ we do have $\Red(f_{iji})=1$; and hence satisfy the remaining conditions for a normalized descent object.

Finally, the claim that the reduction is unique up to pseudonatural equivalence follows immediately from the assumption that $i$ is an equivalence. \end{proof}

In the following we will always replace a given transport 2-functor $\mathrm{tra}:\mathcal{P}_2(M) \to T$ with $\mathcal{B}\mathfrak{G}$-structure by a reduction $\mathrm{tra}\red: \mathcal{P}_2(M) \to \mathcal{B}\mathfrak{G}$, and consider the \emph{reduced surface holonomy} 
\begin{equation*}
\mathrm{RHol}_{\mathrm{tra}}(\phi,x,\EuScript{M})  := \mathrm{Hol}_{\Red \circ \mathrm{tra}\red}(\phi,x,\EuScript{M}) \in \mathfrak{G}\red\text{.}
\end{equation*}
The following lemma shows that the reduced surface holonomy is well-defined.

\begin{lemma}
\label{independenceofreduction}
The reduced surface holonomy is independent of the choice of the reduction $\mathrm{tra}\red$. \end{lemma}

\begin{proof}
Suppose we have two choices, whose compositions with $\Red$ we denote by $\mathrm{t}_1,\mathrm{t}_2\maps \mathcal{P}_2(M) \to \mathcal{BB}\mathfrak{G}\red$. Since $\mathrm{t}_1$ and $\mathrm{t}_2$ are normalized we have $\mathrm{RHol}_{\mathrm{t}_i}(\phi,x,\EuScript{M}) = \mathrm{t}_i(\phi_{*}\Sigma) \bullet c^{\phi,\mathrm{t}_i}$ for $i=1,2$.

By Lemma \ref{prop:reduction} $\mathrm{t}_1$ and $\mathrm{t}_2$ are related by a pseudonatural equivalence $\eta: \mathrm{t}_1 \to \mathrm{t}_2$. In the following we use that the target 2-category $\mathcal{BB}\mathfrak{G}\red$ of $\mathrm{t_1}$ and $\mathrm{t_2}$ is strict, and both horizontal and vertical composition of 2-morphisms is just multiplication in the group $\mathfrak{G}\red$. Since this group is abelian, it follows that an arbitrary composition of 2-morphisms is just the product of their values, in any order.

Under these preliminaries, axiom (T2) for $\eta$ applied to the 2-morphism $\phi_{*}\Sigma$ becomes the identity
\begin{equation}
\label{RHol1}
\eta(\phi_{*}\tau_{\EuScript{M}}) \cdot t_2(\phi_{*}\Sigma) = t_1(\phi_{*}\Sigma) \cdot \eta(\id_{\phi(x)})\text{.}
\end{equation}
Axiom (T1) for $\eta$ applied to the 1-morphism $\phi_{*}\tau_{\EuScript{M}}$ becomes
\begin{equation}
\label{RHol2}
\eta(\phi_{*}\tau_{\EuScript{M}}) \cdot c^{\phi,\mathrm{t}_1} = c^{\phi,\mathrm{t}_2} \cdot \prod_{i=1}^{g} \eta(\alpha_i)\eta(\beta_i)\eta(\alpha_i^{-1})\eta(\beta_i^{-1})\text{.}
\end{equation} 
We are allowed to permute the factors on the right hand side. Since $\mathrm{t_1}$ and $\mathrm{t_2}$ are normalized, we have by \cite[Lemma A.7 (ii)]{schreiber6} that $\eta(\id_{\phi(x)})=1$ and $\eta(\alpha_i)\eta(\alpha_i^{-1})=1$. Now combining \erf{RHol1} with \erf{RHol2} yields $\mathrm{t}_1(\phi_{*}\Sigma) \bullet c^{\phi,\mathrm{t}_1}=\mathrm{t}_2(\phi_{*}\Sigma) \bullet c^{\phi,\mathrm{t}_2}$, which is the equality between the surface holonomies. \end{proof}

\subsection{Properties of Reduced Surface Holonomy}

The reduced surface holonomy has nice properties that we will reveal in the following.
First we treat  the dependence of the reduced surface holonomy on the marking. 

We arrange the  set of markings of $S$ into equivalence classes, as it is common in the theory of Riemann surfaces. 
Two markings $(x,\EuScript{M})$ and $(x',\EuScript{M}')$, with $\EuScript{M}=\left \lbrace \alpha_i,\beta_i \right \rbrace$ and $\EuScript{M}' = \left \lbrace \alpha_i',\beta_i' \right \rbrace$, are called \emph{equivalent}, if there exists a path $\gamma: x \to x'$ and bigons $\Delta_i: \alpha_i \Rightarrow   \gamma \circ \alpha_i' \circ \gamma^{-1}$ and $\Delta_i': \beta_i \Rightarrow   \gamma \circ \beta_i' \circ \gamma^{-1}$. 

\begin{lemma}
\label{lemmarkingindep}
Let $(x,\EuScript M)$ and $(x',\EuScript M')$ be equivalent good markings. Then, 
\begin{equation*}
\mathrm{RHol}_{\mathrm{tra}}(\phi,x,\EuScript M) = \mathrm{RHol}_{\mathrm{tra}}(\phi,x',\EuScript M')\text{.}
\end{equation*}
\end{lemma}

\begin{proof}
Each equivalence between markings splits into a sequence of steps, in which a step is either conjugation of all paths in the marking by a path $\gamma:x \to y$, or changing one of the paths via a bigon. Thus it suffices to prove the invariance of the reduced surface holonomy under each of these steps. Let $\mathrm{tra}\red:\mathcal{P}_2(M) \to \mathcal{B}\mathfrak{G}$ be a reduction of $\mathrm{tra}$, and let $\mathrm{t} := \Red \circ \mathrm{tra}\red$.

In the first part of the proof we look at a conjugation. Let $\Sigma: \tau_{\EuScript M} \Rightarrow \id_x$ be a fundamental bigon for a marking $(x,\EuScript M)$. Denote by $(x',\EuScript M_{\gamma})$ the marking obtained by conjugating all paths with $\gamma$. We have $\tau_{\EuScript M_{\gamma}} = \gamma \circ \tau_{\EuScript M} \circ \gamma^{-1}$, and  $\Sigma^{\gamma} :=  \id_{\gamma} \circ \Sigma \circ \id_{\gamma^{-1}}$ is a fundamental bigon for $(x',\EuScript M_{\gamma})$. We have the two paths $\mathrm{t}_{\phi,\EuScript M}$ and $\mathrm{t}_{\phi,\EuScript M_{\gamma}}$ and the compositors $c^{\phi,\mathrm{t}}: \mathrm{t}_{\phi,\EuScript M} \Rightarrow \mathrm{t}(\tau_{\EuScript M})$ and $c_{\gamma}^{\phi,\mathrm{t}}: \mathrm{t}_{\phi,\EuScript M_{\gamma}} \Rightarrow \mathrm{t}(\tau_{\EuScript M_{\gamma}})$, so that 
\begin{equation*}
\mathrm{Hol}_{\mathrm{t}}(\phi,x,\EuScript M) = \mathrm{t}(\Sigma) \cdot c^{\phi,t}
\quand
\mathrm{Hol}_{\mathrm{t}}(\phi,x',\EuScript M') = \mathrm{t}(\Sigma^{\gamma}) \cdot c_{\gamma}^{\phi,t}\text{.}
\end{equation*}
Here, and in the following we will suppress writing $\phi_{*}$ when we apply $\mathrm{t}$ to paths or bigons in $S$.

There is a unique compositor 2-morphism $c: \mathrm{t}_{\phi,\EuScript M_{\gamma}} \Rightarrow \mathrm{t}(\gamma) \circ \mathrm{t}_{\phi,\EuScript M} \circ \mathrm{t(\gamma)}^{-1}$, and another unique compositor 2-morphism $c': \mathrm{t}(\tau_{\EuScript M_{\gamma}}) \Rightarrow t(\gamma) \circ \mathrm{t}(\tau_{\EuScript M}) \circ t(\gamma)^{-1}$.  
We claim that the diagram
\begin{equation*}
\alxydim{@C=1.8cm}{\mathrm{t}_{\phi,\EuScript M_{\gamma}} \ar@{=>}[d]_{c} \ar@{=>}[r]^{c^{\phi,\mathrm{t}}_{\gamma}} & \mathrm{t}(\tau_{\EuScript M_{\gamma}}) \ar@{=>}[d]^{c'} \ar@{=>}[r]^-{\mathrm{t}(\Sigma^{\gamma})} & \id_{t(\phi(x))} \ar@{=>}[d]^{\id} \\ t(\gamma) \circ \mathrm{t}_{\phi,\EuScript M} \circ t(\gamma)^{-1} \ar@{=>}[r]_-{\id \circ c^{\phi,\mathrm{t}} \circ \id}  & t(\gamma) \circ \mathrm{t}(\tau_{\EuScript M}) \circ t(\gamma)^{-1} \ar@{=>}[r]_{\id \circ \mathrm{t}(\Sigma) \circ \id} &  t(\gamma) \circ \id \circ t(\gamma)^{-1} }
\end{equation*} 
is commutative. The subdiagram on the left commutes because of the coherence of compositors enforced by axiom (F3). In order to see that the subdiagram on the right commutes we note that $c' \eq c_{\tau_{\EuScript M_{\gamma}} \circ \gamma^{-1}, \gamma} \cdot c_{\gamma^{-1},\tau_{\EuScript M_{\gamma}}}$. Axiom (F2) applied to the 2-morphisms $\id$, $\Sigma^{\gamma}$ and $\id$ gives the commutativity with $c_{\gamma^{-1},\gamma} \cdot c_{\gamma^{-1},\id}$ on the right hand side of the diagram. But the latter expression is equal to $1$ for the normalized 2-functor $\mathrm{t}$.

 Next we compute $c$. We note the compositor between $\mathrm{t}(\gamma \circ \alpha \circ \gamma^{-1})$ and $\mathrm{t}(\gamma) \circ \mathrm{t}(\alpha) \circ \mathrm{t}(\gamma^{-1})$ is $c_{\alpha \circ \gamma^{-1},\gamma}\cdot c_{\gamma^{-1},\alpha}$,  while the one between  $\mathrm{t}(\gamma \circ \alpha^{-1} \circ \gamma^{-1})$ and $\mathrm{t}(\gamma) \circ \mathrm{t}(\alpha^{-1}) \circ \mathrm{t}(\gamma^{-1})$ is
$c_{\alpha^{-1} \circ \gamma^{-1},\gamma}\cdot c_{\gamma^{-1},\alpha^{-1}}$. These two compositors are actually inverse to each other; this follows from axiom (F3) for compositors and the fact that $\mathrm{t}$ is normalized. Also note that $\mathrm{t}(\gamma) \circ \mathrm{t}(\gamma^{-1})=\id$; again since $\mathrm{t}$ is normalized. Thus, we see that $c=1$. 
All together, we see that $\mathrm{Hol}_{\mathrm{t}}(\phi,x,\EuScript M)=\mathrm{Hol}_{\mathrm{t}}(\phi,x',\EuScript M')$.

In the second part of the proof we look at a change of one of the paths via a bigon. If $\EuScript M = \left \lbrace \alpha_i,\beta_i \right \rbrace_{i=1}^{g}$ is a marking, we look at another marking $\EuScript M'$ consisting of the same paths except for an index $i_0$ where a different path $\alpha_{i_0}'$ is present, related to $\alpha_{i_0}$ by a bigon $\Delta: \alpha'_{i_0} \Rightarrow \alpha_{i_0}$. Let  $\Delta^{\#}: \alpha_{i_0}'^{-1} \Rightarrow \alpha_{i_0}^{-1}$ be the \quot{horizontally inverted} bigon given by $\Delta^{\#}(s,t) := \Delta(s,1-t)$. 
We consider the 2-morphism
\begin{equation*}
\tilde \Delta := \id \circ (\Delta \circ \id_{\beta_{i_0}} \circ \Delta^{\#} \circ \id_{\beta_{i_0}^{-1}}) \circ \id : \tau_{\EuScript M'} \Rightarrow \tau_{\EuScript M}\text{,}
\end{equation*}
with the outer identities meant for all factors with indices not equal to $i_0$. We also consider the 2-morphism
\begin{equation*}
\mathrm{t}_{\Delta} := \id \circ (\mathrm{t}(\Delta) \circ \id \circ \mathrm{t}(\Delta^{\#}) \circ \id) \circ \id: \mathrm{t}_{\phi,\EuScript M'}\Rightarrow \mathrm{t}_{\phi,\EuScript M}\text{.}
\end{equation*}
If $\Sigma: \tau_{\EuScript M} \Rightarrow \id_x$ is a fundamental bigon for $\EuScript M$, then $\Sigma' := \Sigma \bullet \tilde\Delta$ is a fundamental bigon for $\EuScript M'$. We claim that the diagram
\begin{equation*}
\alxydim{@R=0.6cm@C=1.5cm}{\mathrm{t}_{\phi,\EuScript M'} \ar@{=>}[r]^-{c^{\prime\;\phi,\mathrm{t}}} \ar@{=>}[dd]_{t_{\Delta}} & \mathrm{t}(\tau_{\EuScript M'}) \ar@{=>}[dr]^{\mathrm{t}(\Sigma')} \ar@{=>}[dd]^{\mathrm{t}(\tilde\Delta)} \\ && \id\\ \mathrm{t}_{\phi,\EuScript M} \ar@{=>}[r]_-{c^{\phi,\mathrm{t}}} & \mathrm{t}(\tau_{\EuScript M}) \ar@{=>}[ur]_{\Sigma} }
\end{equation*}
is commutative. Indeed, the rectangular part commutes due to axiom (F2) for the 2-functor $\mathrm{t}$, and the triangular part commutes by definition of $\Sigma'$. It remains to notice that $t_{\Delta} = 1 \in \mathfrak{G}\red$, since the contributions of $\Delta$ and $\Delta^{\#}$ cancel. Then, the diagram implies the coincidence of the surface holonomies.
\end{proof}

We  summarize our results about the surface holonomy of non-abelian gerbes in the following theorem, which constitutes the main result of this section.

\begin{theorem}
\label{th:rhol}
Let $\mathfrak{G}$ be a Lie 2-group,  $T$ be a 2-category, and $i: \mathcal{B}\mathfrak{G} \to T$ be an equivalence of 2-categories. Let $\mathrm{tra}: \mathcal{P}_2(M) \to T$ be a transport 2-functor with $\mathcal{B}\mathfrak{G}$-structure. Let $S$ be a closed oriented surface with a good marking $(x,\EuScript M)$, and let $\phi:S \to M$ be a smooth map. Then, the reduced surface holonomy 
\begin{equation*}
\mathrm{RHol}_{\mathrm{tra}}(\phi,x,\EuScript M) \in \mathfrak{G}\red
\end{equation*}
depends only on the equivalence class of the marking, and only on the isomorphism class of $\mathrm{tra}$.
\end{theorem}

\begin{proof}
After Lemma \ref{lemmarkingindep} it only remains to prove that the reduced surface holonomy only depends on the isomorphism class of $\mathrm{tra}$. Indeed, if $\mathrm{tra}\cong\mathrm{tra}'$ is an isomorphism we can choose the \emph{same} reduction for both.  
\end{proof}

In Section \ref{sec4} we have collected  various equivalences between transport 2-functors and  concrete models of gerbes with connections. These equivalences were mostly established by  zigzags of canonically defined 2-functors. On the level of isomorphism classes, such zigzags give a well-defined bijection. Since the reduced surface holonomy is invariant under isomorphisms, these bijections convey the concept of reduced surface holonomy to a well-defined, isomorphism-invariant concept for each of these concrete models:
\begin{enumerate}[(i)]

\item
The bijection
\begin{equation*}
\hat H^2(M,\mathfrak{G})  \cong \hc 0 \transport{}{2}{\mathrm{Gr}}{T}
\end{equation*}
of Theorem \ref{th:class} between the non-abelian differential cohomology and transport 2-functors equips non-abelian cohomology with a well-defined, isomorphism-invariant  concept of surface holonomy. In particular, it equips degree two Deligne cocycles and (fake-flat) Breen-Messing  cocycles with a well-defined, isomorphism-invariant  concept of surface holonomy. We show below in Proposition \ref{propabsf} in the case of Deligne cocycles it reproduces  the existing concept of surface holonomy. In the case  of Breen-Messing cocycles such a concept was previously unknown; its is one of the main results of this article.   

\item 
The equivalence of Theorem \ref{th:bgrb} induces a canonical bijection
\begin{equation*}
\hc 0 \mathfrak{BGrb}^{\nabla}(M) 
\cong 
\hc 0 \mathrm{Trans}_{\mathcal{BB}S^1}(M,\mathcal{B}(S^1\text{-}\mathrm{Tor}))
\end{equation*}
of  between isomorphism classes of bundle gerbes with connections and isomorphism classes of transport 2-functors, and so equips bundle gerbes with connection with a well-defined, isomorphism-invariant  concept of surface holonomy. We show below in Proposition \ref{propabsf} that it reproduces  the existing concept of surface holonomy for abelian gerbes.

\item
The injective map 
\begin{equation*}
\hc 0 (H\text{-}\mathfrak{BGrb}^{\nabla_{\!f\!f}}(M)) \to \hc 0 (\transport{}{2}{\mathcal{B}\mathrm{AUT}(H)}{\mathcal{B}(\hbitor)})
\end{equation*}
of Theorem \ref{th:nagrb}
between isomorphism classes of fake-flat non-abelian $H$-bundle gerbes with connection and transport 2-functors, equips fake-flat non-abelian $H$-bundle gerbes with connection with a well-defined, isomorphism-invariant  concept of surface holonomy. Such a concept was previously unknown; it is one of the main results of this article.

\end{enumerate}

Now we prove that the reduced surface holonomy of Theorem \ref{th:rhol} reduces in the case $\mathfrak{G} =\mathcal{B}S^1$ to the established notion of surface holonomy for connections on $\mathcal{B}S^1$-gerbes. Let us first recall that established notion. 

It can be given for an element  $\xi\in \hat H^2(M,\mathcal{B}S^1)$ of the (abelian) differential cohomology group which classifies  $\mathcal{B}S^1$-gerbes. It is solely based on the fact that  differential cohomology   sits in an exact sequence
\begin{equation*}
\alxydim{}{1 \ar[r]& \Omega^2_{\mathrm{cl},\Z}(M) \ar[r]& \Omega^2(M) \ar[r]^-{\mathcal{I}}& \hat H^2(M,\mathcal{B}S^1) \ar[r]& H^3(M,\Z) \ar[r]& 0}\text{,}
\end{equation*}
in which $\Omega^2_{\mathrm{cl},\Z}(M)$ denotes the closed 2-forms with integral periods \cite{brylinski1}. If $\phi:S \to M$ is a smooth map from a closed oriented surface $S$ to $M$, then $\phi^{*}\xi \in \hat H^2(S,\mathcal{B}S^1)$ projects to zero in  $0\in H^3(S,\Z)$ for dimensional reasons, and so is the image of a 2-form $B \in \Omega^2(S)$ under the  map $\mathcal{I}$ in the sequence. The abelian surface holonomy of $\xi$ around $\phi$ is then defined as
\begin{equation*}
\mathrm{AbHol}_{\xi}(\phi) := \exp \left (- \int_S B \right ) \in S^1\text{.}
\end{equation*}
Since the difference between two choices of 2-forms is a closed 2-form with integral periods, this integral is independent of the choice of $B$.

\begin{proposition}
\label{propabsf}
Let $\mathrm{tra}: \mathcal{P}_2(M) \to T$ be a transport 2-functor with $\mathcal{BB}S^1$-structure. Let $\xi \in \hat H^2(M,\mathcal{B}S^1)$ be the associated class under the bijection of Theorem \ref{th:class}. Then,
\begin{equation*}
\mathrm{RHol}_{\mathrm{tra}}(\phi,x,\EuScript{M}) = \mathrm{AbHol}_{\xi}(\phi)\text{,}
\end{equation*}
i.e. the reduces surface holonomy of $\mathrm{tra}$ coincides with the established notion of surface holonomy for abelian gerbes. In particular, it is independent of the marking. 
\end{proposition}

\begin{proof}
For any smooth manifold $X$, we have an equality between $\Omega^2(X)$ and the objects of the 2-category $\diffco{\mathcal{B}S^1}{2}{X}$. The latter is isomorphic to the 2-category $\mathrm{Funct}^{\infty}(\mathcal{P}_2(X),\mathcal{B}S^1)$ of smooth 2-functors via Theorem \ref{th1}, and we have computed in Lemma \ref{sfdfbsone}  that
\begin{equation}
\label{sfholco1}
F(\Sigma) = \exp \left (- \int_{\Sigma} B \right ) 
\end{equation}
for any bigon $\Sigma \in BX$. 
By construction we have a commutative diagram
\begin{equation*}
\alxydim{}{\Omega^2(X) \ar[r]^-{\mathcal{I}} \ar[d] & \hat H^2(X,\mathcal{B}S^1) \ar[d]  \\  \mathrm{Funct}^{\infty}(\mathcal{P}_2(X),\mathcal{B}S^1) \ar[r] & \hc 0 \mathrm{Trans}^2_{\mathcal{BB}S^1}(X,\mathcal{BB}S^1)\text{,}}
\end{equation*}
in which the map on the bottom is the inclusion of smooth 2-functors in transport 2-functors.

Note that $\mathfrak{G}\red=S^1$ with Example \ref{ex4} (a) and $\Red = \id_{\mathcal{BB}S^1}$. Let $\mathrm{tra}\red$ be a reduction of the given transport 2-functor, corresponding to the class $\xi$ under the vertical map on the right hand side of the diagram. After pullback along $\phi:S \to M$ we obtain the 2-form $B \in \Omega^2(S)$ such that $\mathcal{I}(B) = \xi$. Let $F: \mathcal{P}_2(S) \to \mathcal{B}S^1$ be the smooth 2-functor that corresponds to  $B$. By commutativity of the diagram, we find an isomorphism 
\begin{equation}
\label{sfholco2}
F \cong \phi^{*}\mathrm{tra}\red
\end{equation}
between transport 2-functors. 

If now $\Sigma$ is a fundamental bigon for the marking $(x,\EuScript M)$, then we have
on one side
\begin{equation*}
\mathrm{AbHol}_{\xi} (\phi) := \exp \left(- \int_S B  \right ) \stackrel{\text{Lemma \ref{fundamentalbigonintegral}}}{=} \exp \left(- \int_\Sigma B  \right ) \stackerf{sfholco1}{=} F(\Sigma) = \mathrm{Hol}_F(\id,x,\EuScript M)\text{,}
\end{equation*}
where the last equality is Definition \ref{def:surfaceholonomy} combined with the fact that $F$ is a strict 2-functor (with trivial compositors). Lemma \ref{independenceofreduction} applied to the equivalence \erf{sfholco2} shows that \begin{equation*}
\mathrm{Hol}_F(\id,x,\EuScript M) = \mathrm{Hol}_{\phi^{*}\mathrm{tra}\red}(\id,x,\EuScript M)\text{.}
\end{equation*}
We have on the other side
\begin{equation*}
\mathrm{Hol}_{\phi^{*}\mathrm{tra}\red}(\id,x,\EuScript M)=\mathrm{Hol}_{\mathrm{tra}\red}(\phi,x,\EuScript M)=\mathrm{RHol}_{\mathrm{tra}}(\phi,x,\EuScript M)\text{,}
\end{equation*}
this shows the claimed coincidence.
\end{proof}

Finally, let us comment on the dependence of the reduced surface holonomy on the equivalence class of a marking.  Let $S$ be a closed oriented surface, and let $(x,\EuScript M)$ and $(x',\EuScript M')$ two markings. A standard result from the theory of Riemann surfaces is:
\begin{enumerate}[(i)]

\item 
There exists an orientation-preserving diffeomorphism $f:S \to S$ such that $(x',\EuScript M') = f(x,\EuScript M)$,

\item
The diffeomorphism $f$ is homotopic to the identity map $\id_S$ if and only if $(x,\EuScript M)$ and $(x',\EuScript M')$ are equivalent.

\begin{comment}
By assumption there exists a smooth homotopy $h: [0,1] \times S \to S$ between $\id_S$ and  $f$. Define the paths $\gamma_s(t) := h(st,x)$, so that $\gamma := \gamma_1: x \to f(x)$. Define further the paths $\alpha_i^s(t) := h(s,\alpha_i(t))$. Then define
\begin{equation*}
\Delta_i(s,t) := (\gamma_s^{-1} \circ \alpha_i^s \circ \gamma_s)(t)\text{,}
\end{equation*}
which is a bigon from $\alpha_i$ to $\gamma^{-1} \circ \alpha'_i \circ \gamma$. Similarly, we get bigons $\Delta_i'$ from $\beta_i$ to $\gamma^{-1} \circ \beta_i' \circ \gamma$.   
\end{comment} 

\end{enumerate}
The quotient of the group of orientation-preserving diffeomorphisms modulo those homotopic to the identity is the mapping class group of $S$, denoted $\mathcal{M}(S)$. Thus, $\mathcal{M}(S)$ acts on the reduced surface holonomies of a surface $\phi:S\to M$. Currently we do not know what this action is; in particular, we do not know if it is trivial or not.


\kobib{../../bibliothek/tex}

\end{document}